\documentclass[11pt,reqno]{amsart}
\usepackage{dsfont}
\usepackage{amssymb}
\usepackage{amsfonts}
\usepackage{amsthm}
\usepackage{amsmath}
\usepackage{array}
\usepackage{stackrel}
\usepackage{a4wide}
\usepackage{amssymb}
\usepackage[font=footnotesize]{subfig}
\usepackage{color}

\usepackage{ifthen} 
\provideboolean{shownotes} 
\setboolean{shownotes}{true} 
\newcommand{\margnote}[1]{
\ifthenelse{\boolean{shownotes}}%
{\marginpar{\raggedright\tiny\texttt{#1}}}%
{}%
}
\newcommand{\hole}[1]{
\ifthenelse{\boolean{shownotes}}%
{\begin{center} \fbox{ \rule {.25cm}{0cm}
\rule[-.1cm]{0cm}{.4cm} \parbox{.85\textwidth}{\begin{center}
\texttt{#1}\end{center}} \rule {.25cm}{0cm}}\end{center}}
{}
}

\newtheorem{lemma}{Lemma}[section]
\newtheorem{theorem}[lemma]{Theorem}

\newtheorem{assumption}[lemma]{Assumption}

\theoremstyle{definition}
\newtheorem{definition}[lemma]{Definition}
\theoremstyle{definition}
\newtheorem{remark}[lemma]{Remark}
\theoremstyle{definition}

{\catcode `\@=11 \global\let\AddToReset=\@addtoreset}
\AddToReset{equation}{section}

\newcommand{\cE}{\mathcal{E}}
\newcommand{\cH}{\mathcal{H}}
\newcommand{\R}{\mathbb R}
\newcommand{\T}{\mathbb T}

\newcommand{\eps}{\varepsilon}

\newcommand{\dx}{\mathop{{\rm div}}}

\newcommand{\curl}{\mathop{\rm curl} }

\newcommand{\brho}{\bar \rho}
\newcommand{\baru}{\bar u}
\newcommand{\barm}{\bar m}
\newcommand{\barc}{\bar c} 
\newcommand{\barq}{\bar q}

\newcommand{\norm}[1]{\ensuremath{\left|#1\right|}}

\newcommand{\Norm}[1]{\ensuremath{\left\|#1\right\|}}

\renewcommand{\div}{\operatorname{div}}

\newcommand{\del}{\partial}

\newcommand{\rhoa}{\rho^\alpha}
\newcommand{\ca}{c^\alpha}
\newcommand{\va}{{  u}^\alpha}
\newcommand{\ma}{{ m}^\alpha}
\newcommand{\rre}{\eta^\alpha_{\mathrm{R}}}

\title[Relative energy for Hamiltonian flows]
{Relative energy for the Korteweg theory and related Hamiltonian flows in gas dynamics}

 \author{Jan Giesselmann} 
 \address[Jan Giesselmann]{\newline
  Institute of Applied Analysis and Numerical Simulation\newline
 University of Stuttgart\newline
 Pfaffenwaldring 57\newline
 D-70563 Stuttgart\newline
   Germany} 
 \curraddr{}
 \email{{jan.giesselmann@mathematik.uni-stuttgart.de}} 
 \thanks{
JG partially supported by the 
 German Research Foundation (DFG) via SFB TRR 75 `Tropfendynamische Prozesse unter extremen Umgebungsbedingungen'}

\author{Corrado Lattanzio}
\address[Corrado Lattanzio]{\newline 
Dipartimento di Ingegneria e Scienze dell'Informazione e Matematica
\newline
Universit\`a degli Studi dell'Aquila
\newline
Via Vetoio
\newline
I-67010 Coppito (L'Aquila) AQ 
\newline 
Italy
}
\email{corrado@univaq.it}

\author{Athanasios E. Tzavaras}
\address[Athanasios E. Tzavaras]{\newline 
Computer, Electrical, Mathematical Sciences \& Engineering Division
\newline
King Abdullah University of Science and Technology (KAUST)
\newline
Thuwal,  Saudi Arabia
\newline
and
\newline
Institute for Applied and Computational Mathematics
\newline
Foundation for Research and Technology
\newline
GR 70013 Heraklion, Crete
\newline
Greece 
}
\email{athanasios.tzavaras@kaust.edu.sa}
 \thanks{
 AET acknowledges the support of the King Abdullah University of Science and Technology (KAUST) and of the
 Aristeia program of the Greek Secretariat for Research through the project DIKICOMA}

\begin{document}
\begin{abstract}
We consider an Euler system with dynamics generated by a potential energy functional. We propose
a form for the relative energy that exploits the  variational structure and derive a relative energy identity. 
When applied to specific energies, this yields relative energy identities for  Euler-Korteweg, the Euler-Poisson, 
the Quantum Hydrodynamics system, and  low order approximations of the Euler-Korteweg system.
For the Euler-Korteweg system we prove a stability theorem between a weak and a strong solution and an associated weak-strong uniqueness theorem.
In the second part we focus on the Navier-Stokes-Korteweg system (NSK) with non-monotone pressure laws: we prove stability for the NSK system
via a modified relative energy approach.  We  prove continuous dependence of solutions on initial data and convergence of solutions 
of a low order model to solutions of the NSK system. 
The last two  results provide  physically meaningful examples of how higher order regularization terms enable
the use of the relative energy framework for models with energies which are not poly- or quasi-convex, compensated by higher-order gradients.
\end{abstract}
\maketitle

\tableofcontents

\section{Introduction}

We study the system of partial differential equations
\begin{equation}
\label{funsys-intro}
\begin{aligned}
\frac{\del \rho}{\del t} + \div (\rho u) &= 0 
\\
\frac{Du}{Dt} :=   \frac{\del u}{\del t} +   (u \cdot \nabla) u  &= -  \nabla \frac{\delta \cE}{\delta \rho} (\rho)
\end{aligned}
\qquad x \in \R^d \, , \; t > 0 \, ,
\end{equation}
where $\rho \ge 0$ is a density obeying the conservation of mass equation, $u$ is the fluid velocity and $m = \rho u$ the momentum flux.
The dynamical equation determining the evolution of $u$ is generated  by a functional $\cE(\rho)$ on the density and
 $\tfrac{\delta \cE}{\delta \rho}(\rho)$ stands for the variational derivative of $\cE(\rho)$. The dynamics \eqref{funsys-intro} formally
 satisfies the conservation of energy equation
 \begin{equation}
\label{toteintro}
\frac{d}{dt} \left (  \int \tfrac{1}{2} \rho |u|^2 \, dx  + \cE (\rho)  \right ) = 0  \, .
\end{equation}

Depending on the selection of the  functional $\cE (\rho)$ several models of interest fit under this framework  (see Section \ref{sec:ham}). 
These include the equations of isentropic gas dynamics,
the Euler-Poisson system ({\it e.g.\ }\cite{PW05,LS09,JLLJ10}), the system of quantum hydrodynamics ({\it e.g.}\ \cite{AM09,AM12}), the Euler-Korteweg system ({\it e.g.\ }\cite{DS85,B10}),
and order-parameter models for the study of phase transitions ({\it e.g.\ }\cite{BLR95,RT00,Roh10}).
The objective here is to review the relation of  these problems with the formal structure \eqref{funsys-intro} and to use the structure in order to obtain a relative entropy
identity. The idea is quite simple: most (but not all) of the problems above are generated by convex functionals. 
For convex functionals, it is natural to employ the difference
between $\cE (\rho)$ and the linear part of
the Taylor expansion around $\bar\rho$, 
\begin{equation}
\label{repotedef-intro}
\cE (\rho | \brho) := \cE(\rho) - \cE (\brho) - \big \langle \frac{\delta \cE}{\delta \rho} (\brho) , \rho - \brho \big \rangle \, , 
\end{equation}
in order to compare  the distance between two states $\rho (t, \cdot)$ and $\brho (t, \cdot)$.
This definition involves the directional derivative of $\cE (\rho)$ in the direction $(\rho - \brho)$
and provides a functional called here relative potential energy. The relative energy is used, along with the relative kinetic energy
\begin{equation}
\label{rekinedef-intro}
K(\rho, m | \brho, \barm) = \int \rho \Big | \frac{m}{\rho} - \frac{\barm}{\brho} \Big |^2 dx \, ,
\end{equation}
as a yardstick for comparing the distance between two solutions $(\rho, u)$ and $(\brho, \baru)$.
An additional ingredient is needed: we postulate the existence of a stress tensor (functional) $S(\rho)$ such that
\begin{equation}
\label{stresshyp-intro}
- \rho \nabla \frac{\delta \cE}{\delta \rho} =  \nabla \cdot S \, .
\end{equation}
Hypothesis \eqref{stresshyp-intro} holds for all the above examples, it gives a meaning to the notion of weak solutions for \eqref{funsys-intro}
as it induces a conservative form
\begin{equation}
\label{wkfunsys-intro}
\begin{aligned}
\frac{\del \rho}{\del t} + \dx (\rho u) &= 0 
\\
\frac{\del}{\del t} (\rho u) + \dx (\rho u \otimes u) &=  \nabla \cdot S  \, ,
\end{aligned}
\end{equation}  
and plays an instrumental role in devising a relative entropy identity. For the Euler-Korteweg system it is a consequence of the invariance under translations
of the generating functional and Noether's theorem (see \cite{B10} and Appendix \ref{app-noether}). The relative energy identity takes the  abstract form 
\begin{equation}
\label{reltote-intro}
\begin{aligned}
&\frac{d}{dt} \left ( \cE (\rho | \brho) +  \int \tfrac{1}{2} \rho |u - \baru|^2 \, dx \right )
\\
&\qquad \qquad = \int  \nabla \baru : S (\rho | \brho )   \, dx  - \int \rho \nabla \baru : (u - \baru) \otimes (u - \bar u) \, dx \, ,
\end{aligned}
\end{equation}
where $\cE (\rho | \brho)$ is defined in \eqref{repotedef-intro}, while  the relative stress functional is
\begin{equation}
\label{relstress-intro}
S ( \rho | \brho) := S(\rho) - S (\brho) - \big \langle \frac{\delta S}{\delta \rho} (\brho) , \rho - \brho \big \rangle \, .
\end{equation}

Formula \eqref{reltote-intro} is the main result of this article. It should be compared to the well known relative entropy formulas initiated in the works 
Dafermos \cite{Daf79,Daf79b}, DiPerna \cite{Dip79}  and analogs that have been  successfully used in many contexts 
({\it e.g.}\ \cite{LT06, LV11,FN12,BT13,LT13,Gie_14a}). 
It has however a different origin from all these calculations:
while the latter are based on the thermodynamical structure induced by the Clausius-Duhem inequality, the formula \eqref{reltote-intro} is based
on the abstract Hamiltonian flow structure in \eqref{funsys-intro}. The reader  familiar with the ramifications of
continuum thermodynamics will observe that, as noted by Dunn and Serrin \cite{DS85}, constitutive theories with higher-order gradients (unless trivial) are inconsistent 
with the Clausius-Duhem structure, and in order to make them compatible one has to introduce the interstitial work term in the energy equation.
Nevertheless, as shown in Sections  \ref{sec-funcrel-pote}-\ref{sec-funcrel-tote},  the structure 
\eqref{funsys-intro}, \eqref{stresshyp-intro}   induces a relative entropy identity provided $\cE(\rho | \brho)$   is defined through
the Taylor expansion of the generating functional.

Despite this difference, there is also striking similarity between the formulas obtained in \cite{Daf79b,LT13} and the formula \eqref{reltote-intro} 
in connection to the mechanical interpretations for the relative mechanical stress and the relative convective stress.
This similarity  is conceptually quite appealing, yet the simplicity of the formula \eqref{reltote-intro} is somewhat misleading.
In fact, the actual formulas in specific examples are cumbersome, as can be noticed in Section \ref{sec-EK} in formula \eqref{reltotekort} 
derived  for the Euler-Korteweg system, or in Section \ref{sec-EP} in the relative energy formula \eqref{reltoteep} for the Euler-Poisson system. 
In addition, the simplicity of the derivation using the functional framework presented in Sections \ref{sec-funcrel-pote}, \ref{sec-funcrel-kine} 
should be contrasted to the lengthy direct derivation  in Appendix \ref{subsec:relenKort}
for the Euler-Korteweg system.

Our work is closely related to the observations of Benzoni-Gavage et al.\ \cite{BDDJ05,B10} that the Euler-Korteweg system 
can be formally expressed for irrotational flow velocity fields as a Hamiltonian system. While for general flows the structure of the problem
fails to be Hamiltonian (see \eqref{hamiltonianform}), this discrepancy is consistent with energy conservation and as shown here induces 
with \eqref{stresshyp-intro} the relative energy identity. 

Gradient flows generated by functionals $\cE(\rho)$ in Wasserstein distance
give rise to parabolic equations 
\begin{equation}
\label{limitdiff}
\del_t \rho = \nabla \cdot \rho \nabla \frac{\delta \cE}{\delta \rho}
\end{equation}
and have received ample  attention in multiple contexts ({\it e.g.}\ \cite{Ot01,Br14}). Such systems can be seen as describing the long-time response
of an Euler equation with friction, see \eqref{funsys}. At long times the friction term tends to equilibrate with the gradient of the energy, 
$\rho u \sim - \rho \nabla \frac{\delta \cE}{\delta \rho}$, and produces the diffusion equation \eqref{limitdiff}. 
A justification of this process using the relative energy framework 
is undertaken in the companion articles \cite{LTforth,LT13}; their analysis is based on the relative energy formula \eqref{reltote}, the variant of 
\eqref{reltote-intro} accounting for the effect of friction.

In the second part of this work, we establish some applications of \eqref{reltote-intro}. We consider first the Euler-Korteweg system
\begin{equation}
    \begin{cases}
	\displaystyle{\rho_{t} +\dx (\rho u) =0}&   \\
	& \\
       \displaystyle{ (\rho u)_{t}} + \dx (\rho u \otimes u)
	=  - \rho \nabla \Big ( F_\rho(\rho, \nabla \rho)  - \dx F_q (\rho, \nabla \rho)   \Big )\, , &  \\
    \end{cases}
    \label{Kortcap-intro}
\end{equation}
generated by a convex energy
\begin{equation}
\label{kortfunc-intro}
\cE (\rho) = \int F(\rho, \nabla \rho ) \, dx = \int h(\rho) + \frac{1}{2} \kappa(\rho) |\nabla \rho|^2  \, dx\, ,
\end{equation}
and proceed to compare a conservative (or dissipative) weak solution $(\rho, m)$ to a strong solution $(\brho, \barm)$ of \eqref{Kortcap-intro}.
First, we establish the relative entropy transport formula for comparing solutions of this regularity classes in Theorem \ref{theo:relenweak}.
This, in turn, is used to establish two stability results between weak and strong solutions of \eqref{Kortcap-intro}: 
\begin{itemize}
\item[ (i) ]
Theorem \ref{th:finalKortconst2}
valid for convex energies \eqref{kortfunc-intro} under the hypothesis that the densities $\rho$ of the weak solution 
and $\brho$ of the strong solution are bounded and do not contain vacuum regions;
this generalizes to energies \eqref{kortfunc-intro} the result of \cite{DFM15} valid for energies with $\kappa (\rho) = C_\kappa$.
\item[(ii)] Theorem \ref{th:finalKortconst} which restricts  \eqref{kortfunc-intro} to convex  energies 
with constant capillarity ($\kappa(\rho) = C_\kappa$) and to strong solutions with density $\brho$ bounded away from vacuum, but in return
does not place any boundedness assumptions on the density $\rho$ except for the natural energy norm bound.
\item[(iii)] Theorem \ref{th:finalstabqhd} providing a weak-strong stability result for the quantum hydrodynamics system \eqref{qhd}.
\end{itemize}

Augmenting the Euler-Korteweg system with  viscosity leads to the isothermal Navier-Stokes-Korteweg (NSK) model. 
When the capillarity coefficient $\kappa (\rho) = C_\kappa > 0$ is constant in \eqref{kortfunc-intro} and the pressure function
$p(\rho)$ is non-monotone, the NSK system
\begin{equation}\label{NSK-intro}
 \begin{split}
   \rho_t + \div(\rho u) &=0\\
   (\rho  u)_t + \div (\rho  u \otimes  u) + \nabla p(\rho) &= \div(\sigma[ u]) + C_\kappa \rho \nabla \triangle \rho\, ,
 \end{split}
\end{equation}
is a well-known model for compressible liquid-vapor flows undergoing phase transitions of  diffuse interface type
({\it e.g.} \cite{DS85,B10}).
The term
\begin{equation}\label{def:nss_int}
\sigma[u]:= \lambda \div( u) \mathbb{I}+ \mu (\nabla u + (\nabla u)^T)
\end{equation}
is the Navier-Stokes stress with coefficients $\lambda, \mu$ satisfying $\mu\geq0$ and $\lambda + \frac{2}{d} \mu \geq 0,$ and 
$\mathbb{I} \in \mathbb{R}^{d \times d}$ is the unit matrix and $d$ the spatial dimension.
To describe multi-phase flows the energy density $h=h(\rho)$ is non-convex and the associated pressure
$ p'(\rho)= \rho h''(\rho)$ is non-monotone, so that the first order part of \eqref{NSK-intro} is  a system of mixed hyperbolic-elliptic type.

The relative entropy technique usually applies to situations where the energy is convex
 or at least quasi-convex or poly-convex \cite{Daf86,LT06,Daf10}.
Here we provide some  examples on how -- in a physically meaningful and multi-dimensional situation --
the higher-order (second gradient) regularization mechanism
compensates for the non-convexity of the energy in such a way that the relative entropy technique still
provides stability estimates.
This extends results  from \cite{Gie_14a}, valid in a one-dimensional Lagrangian setting.
To highlight the use of a stability theory for \eqref{NSK-intro} derived via a modified relative entropy approach,
we prove two results:
\begin{itemize}
\item[(a)]
we show stability of smooth solutions to \eqref{NSK-intro} or to the associated Korteweg-Euler system with non-monotone pressures
for initial-data with equal initial mass.
\item[(b)]
we  show that  solutions of a lower order approximation to the Navier-Stokes-Korteweg system given by \eqref{lo-intro} below
converge to solutions of \eqref{NSK-intro} in the limit $\alpha \to \infty$.
\end{itemize}

To place (b) in the relevant context, we note that the system 
\begin{equation}\label{lo-intro}
 \begin{split}
\frac{\partial}{\partial t}   \rho + \div(\rho u) &=0\\
\frac{\partial}{\partial t}   (\rho u) + \div (\rho u \otimes u) + \nabla (p(\rho) + C_\kappa \frac{\alpha}{2} \rho^2)&= \div(\sigma[ u])
  + C_\kappa\alpha \rho \nabla  c\\
  c - \frac{1}{\alpha}\triangle c &= \rho\, ,
 \end{split}
\end{equation}
where  $C_\kappa > 0$ and $\alpha>0$  a parameter, was introduced in \cite{Roh10} with the goal
to approximate the NSK system \eqref{NSK-intro}. It has motivated efficient numerical schemes for the numerical treatment of diffuse interface
systems for the description of phase transitions (see the discussion in the beginning of Section \ref{sec:mc}).  The model convergence from \eqref{lo-intro} to \eqref{NSK-intro}
as $\alpha \to \infty$ was investigated in simple cases by \cite{Cha14} (using  Fourier methods) and \cite{Gie_14a}  (in a 1-d setting in Lagrangean coordinates). 
Here, we exploit the fact that the system \eqref{lo-intro} fits into the functional framework
of \eqref{funsys-intro}, \eqref{stresshyp-intro} and is thus equipped with a relative energy
identity (see Section \ref{sec-LOA}).  Using the relative energy framework,
as modified in Section \ref{sec:cd} pertaining to non-monotone pressures, we prove in Theorem \ref{thrm:mt2} convergence  from \eqref{lo-intro} to \eqref{NSK-intro}
as $\alpha \to \infty$ for initial data of equal initial mass.

%
%

The structure of the paper  is as follows: in Section \ref{sec:ham} we discuss the derivation of the formal relative energy estimate 
 \eqref{reltote-intro}  at the level of the abstract equation \eqref{funsys-intro}, and then we implement its application to various models:
the Euler-Korteweg system in Section \ref{sec-EK}, the Euler-Poisson system in Section  \ref{sec-EP}, and the order parameter model of lower order
in Section \ref{sec-LOA}.
The relation between \eqref{funsys-intro} and a Hamiltonian structure 
is discussed in Section \ref{sec-hamrem}.
In Section \ref{sec:kort}  we consider the Euler-Korteweg system \eqref{Kortcap-intro} and we establish 
 the relative energy estimate for solutions of limited smoothness, that is between a dissipative (or conservative) weak solution and a strong solution
 of \eqref{Kortcap-intro}. We then derive the two weak-strong stability Theorems  \ref{th:finalKortconst}  and \ref{th:finalKortconst2}.
In Section \ref{sec:cd} we consider the Navier-Stokes-Korteweg system \eqref{NSK-intro} with non-monotone pressure $p(\rho)$,
we present a relative energy calculation and show that solutions of \eqref{NSK-intro} depend continuously on their
initial data in Theorem \ref{thrm:mt1}.
In Section \ref{sec:mc} we briefly introduce the lower order model \eqref{lo-intro} and  derive an estimate for the
difference between solutions
of \eqref{lo-intro} and \eqref{NSK-intro}. In the appendices we present: a remark on the relation of \eqref{stresshyp-intro} to the invariance under spatial-translations
of the energy functional and Noether's theorem in Appendix \ref{app-noether}; a direct derivation of the relative energy identity for the Euler-Korteweg system
\eqref{Kortcap-intro} in Appendix \ref{subsec:relenKort}; the derivation of the relative energy identity for the lower order model \eqref{lo-intro} in Appendix  \ref{sec:proofmc}.

\section{Hamiltonian flow and the relative energy}
\label{sec:ham}

In this section we consider a system of equations consisting of a conservation of mass and a functional momentum equation
\begin{equation}
\label{funsys}
\begin{aligned}
\frac{\del \rho}{\del t} + \dx (\rho u) &= 0 
\\
\rho \frac{Du}{Dt} &= - \rho \nabla \frac{\delta \cE}{\delta \rho} - \zeta \rho u\, ,
\end{aligned}
\end{equation}
where $\rho \ge  0$ is the density, $u$ the velocity, and $\frac{D}{Dt}= \frac{ \del}{\del t} + u \cdot \nabla$ stands for the material derivative operator.
In \eqref{funsys}$_2$,  $\frac{\delta \cE}{\delta \rho}$ stands for  the variational derivative of  $\cE(\rho)$ (see the discussion below).  The term $(-\zeta \rho u)$   in \eqref{funsys} corresponds to a damping force with
frictional coefficient $\zeta > 0$;  the particular frictionless case $\zeta = 0$, leading to \eqref{funsys-intro}, is also allowed.
The objective of the section is to derive the main   relative entropy calculation 
of this work and to apply it to certain specific systems. The derivation is formal in nature and has
to be validated via alternative methods for solutions of limited smoothness ({\it e.g.}\ for weak or for measure-valued solutions).
Nevertheless, the formal derivation  uses in an essential way the functional structure of  \eqref{funsys} and it is fairly simple to achieve via this formalism.

In the sequel  $\cE (\rho)$ will stand for a functional on the density $\rho (\cdot , t)$. The directional derivative (Gateaux derivative) of the functional 
$\cE : U \subset X \to \R$, where $U$ is an open subset of $X$, a locally convex topological vector space, is defined by
$$
d\cE (\rho ; \psi ) = \lim_{\tau \to 0}  \frac{\cE (\rho + \tau \psi) - \cE (\rho)}{\tau} = \frac{d}{d\tau} \cE (\rho + \tau \psi ) \Big |_{\tau = 0}\, ,
$$
with the limit taken over the reals. When the limit exists the functional is Gateaux differentiable at $\rho$ and $d\cE (\rho ; \psi )$ is the directional derivative. 
The function $d \cE (\rho, \cdot)$ is homogeneous of order one, but may in general fail to be linear (in $\psi$). 
We will assume that  $d\cE (\rho ; \psi )$ is linear in $\psi$ and can be 
represented via a duality bracket
\begin{equation}
\label{funchyp1}
d\cE (\rho ; \psi ) = \frac{d}{d\tau} \cE (\rho + \tau \psi ) \Big |_{\tau = 0} 
= \big \langle \frac{\delta \cE}{\delta \rho }(\rho) , \psi  \big \rangle\, ,
\end{equation}
with $\frac{\delta \cE}{\delta \rho }(\rho)$ standing for the generator of the bracket. When this representation holds
$\frac{\delta \cE}{\delta \rho }(\rho)$ is called the variational derivative of $\cE (\rho)$.

Further, we define the second variation via
$$
d^2 \cE (\rho;  \psi, \varphi) = \lim_{\eps \to 0} 
\frac{ \big \langle \frac{\delta \cE}{\delta \rho }(\rho + \eps \varphi) , \psi  \big \rangle - \big \langle \frac{\delta \cE}{\delta \rho }(\rho) , \psi  \big \rangle }{\eps}
$$
(whenever the limit exists) and we assume that this can be represented as a bilinear functional in the form
\begin{equation}
\label{funchyp2}
d^2 \cE (\rho;  \psi, \varphi) = \lim_{\eps \to 0}  
\frac{ \big \langle \frac{\delta \cE}{\delta \rho }(\rho + \eps \varphi) , \psi  \big \rangle - \big \langle \frac{\delta \cE}{\delta \rho }(\rho) , \psi  \big \rangle }{\eps}
= \left \langle  \left \langle \frac{\delta^2 \cE }{\delta \rho^2} (\rho) , ( \psi, \varphi ) \right \rangle \right \rangle\, .
\end{equation}

We note that \eqref{funchyp1} and \eqref{funchyp2} hold when the functional $\cE (\rho)$ is Fr\'echet differentiable on a Banach space $X$ and with sufficient smoothness.
Moreover, in a framework when $X$ is a Fr\'echet space (locally convex topological vector space that is metrizable) there are available theorems that show  
that continuity of $ d \cE (\rho, \psi) :  X \times X \to \R$  guarantees the linearity of $ d \cE (\rho, \cdot)$;  corresponding theorems also hold for the
second variation; see \cite[Sec I.3]{hamilton82}. In the following formal calculations we will place  \eqref{funchyp1} and \eqref{funchyp2} as hypotheses and will
validate them for two examples: the Euler-Korteweg and the Euler-Poisson system.

By standard calculations the system \eqref{funsys} can be expressed in the form 
\begin{equation}
\label{wkfunsys}
\begin{aligned}
\frac{\del \rho}{\del t} + \dx (\rho u) &= 0 
\\
\frac{\del}{\del t} (\rho u) + \dx (\rho u \otimes u) &= - \rho \nabla \frac{\delta \cE}{\delta \rho} - \zeta \rho u\, .
\end{aligned}
\end{equation}  
The left part of \eqref{wkfunsys} is in conservation form, however this is not generally true for the term $\rho \nabla \frac{\delta \cE}{\delta \rho}$.
Nevertheless, for all examples of interest here it will turn out that
\begin{equation}
\label{stresshyp}
- \rho \nabla \frac{\delta \cE}{\delta \rho} =  \nabla \cdot S \, ,
\end{equation}
where $ S =S (\rho)$ will be a tensor-valued functional on $\rho$ that plays the role of a stress tensor and has components $S_{i j} (\rho)$ with 
$i, j = 1, ... , d$.

At this stage \eqref{stresshyp} is placed as a hypothesis. This hypothesis is validated in the sequel for various specific models.
It also turns out that it is a fairly general consequence of the invariance of the functional
$\cE(\rho)$ under space translations via Noether's theorem (see appendix). 
Note that \eqref{stresshyp} gives a meaning to weak solutions for \eqref{wkfunsys} and is
instrumental  in the 
forthcoming calculations for the potential, kinetic and total energy. It is assumed
 that no work is done at the boundaries and thus integrations by parts do not result to boundary contributions. 
In particular they are valid for periodic boundary conditions ($x \in \T^d$ the torus) or on the entire space ($x \in \R^d$).

 The potential energy is computed via
\begin{equation}
\label{pote}
\frac{d}{dt} \cE (\rho) \stackrel{(\ref{funchyp1})}{=}  \langle \frac{\delta \cE}{\delta \rho} (\rho) , \rho_t \rangle 
\stackrel{(\ref{funsys})}{=} 
- \langle \frac{\delta \cE}{\delta \rho} (\rho) , \dx (\rho u) \rangle
\stackrel{(\ref{stresshyp})}{=} 
  \int S  : \nabla u \, dx\, .
\end{equation}
Taking the inner product of \eqref{funsys}$_2$ by $u$ gives
$$
\tfrac{1}{2}\rho \frac{D}{Dt} |u|^2 = - u \cdot \rho \nabla \frac{\delta \cE}{\delta \rho} - \zeta \rho |u|^2 \stackrel{\eqref{stresshyp}}{=} u \cdot \dx S - \zeta \rho |u|^2\, .
$$
Then, using \eqref{funsys}$_1$, we obtain
$$
 \tfrac{1}{2} \del_t(\rho |u|^2) + \dx ( \tfrac{1}{2}\rho |u|^2 u) = u \cdot \dx S - \zeta \rho |u|^2\, ,
$$
which leads to the evolution of the kinetic energy
\begin{equation}
\label{kine}
\frac{d}{dt} \int \tfrac{1}{2} \rho |u|^2 \, dx = - \int \Big (S : \nabla u  + \zeta \rho |u|^2\Big ) \, dx \, .
\end{equation}
Combining \eqref{pote} and \eqref{kine} provides the balance of total energy 
\begin{equation}
\label{tote}
\frac{d}{dt} \left (  \int \tfrac{1}{2} \rho |u|^2 \, dx  + \cE (\rho)  \right ) = - \zeta \int \rho |u|^2 \, dx \, .
\end{equation}
In the frictionless case $\zeta = 0$ the total energy is conserved, while for $\zeta > 0$ there is dissipation due to friction.

\subsection{The relative potential energy} \label{sec-funcrel-pote}

Consider now the system \eqref{funsys} generated by the functional $\cE (\rho)$ and assume that the functional is convex. 
A natural quantity to monitor is provided by the quadratic part of the Taylor series expansion of the functional with respect to a reference solution $\bar\rho(x,t)$; this quantity
 is called here relative potential energy
and is defined by
\begin{equation}
\label{relendef}
\cE (\rho | \brho) := \cE(\rho) - \cE (\brho) - \big \langle \frac{\delta \cE}{\delta \rho} (\brho) , \rho - \brho \big \rangle\, .
\end{equation}

We next develop the relative potential energy calculation, which is based on the hypotheses \eqref{funchyp1}, \eqref{funchyp2}, \eqref{stresshyp}
and \eqref{funchyp3} below.
For $\varphi_i$ a vector valued test function, $i = 1, ..., d$,  the weak form of \eqref{stresshyp} reads
\begin{equation}
\label{varder1}
\big \langle \frac{\delta \cE}{\delta \rho} (\rho) , \frac{\del}{\del x_j} (\rho \varphi_j) \big \rangle = - \int S_{i j} (\rho) \, \frac{\del \varphi_i}{\del x_j } \, dx\, ,
\end{equation}
where we employ the summation convention over repeated indices.  

Recall that the stress $S(\rho)$ is a functional of the density $\rho$. We assume that the directional derivative of $S$ is expressed as 
a linear functional via a duality bracket,
\begin{equation}
\label{funchyp3}
d S  (\rho ; \psi ) = \frac{d}{d\tau} S (\rho + \tau \psi ) \Big |_{\tau = 0} 
= \big \langle \frac{\delta S}{\delta \rho }(\rho) , \psi  \big \rangle \, ,
\end{equation}
in terms of the variational derivative  $\frac{\delta S}{\delta \rho }(\rho)$ (in complete analogy to \eqref{funchyp1}).

We now take the directional derivative  of \eqref{varder1}  --- viewed as a functional in $\rho$ ---  along a
direction $\psi$,  with $\psi$ a smooth test function,  and use \eqref{funchyp2}, \eqref{funchyp3}
to arrive at the formula
\begin{equation}
\label{varder2}
 \left \langle  \left \langle \frac{\delta^2 \cE }{\delta \rho^2} (\rho) , \big ( \psi, \frac{\del}{\del x_j} (\rho \varphi_j) \big ) \right \rangle \right \rangle
 + 
  \left \langle  \frac{\delta \cE}{\delta \rho} (\rho) , \frac{\del}{\del x_j} (\psi \varphi_j) \right \rangle 
 =
 - \int  \left \langle  \frac{\delta S_{ij} }{\delta \rho} (\rho) , \psi  \right \rangle   \frac{\del \varphi_i }{\del x_j} \, dx \, . 
 \end{equation}

Let now $(\rho, u)$ and $(\brho, \baru)$ be two smooth solutions of \eqref{funsys}. Using \eqref{pote}, \eqref{funsys}, we compute
\begin{align*}
\del_t \Big (  &\cE(\rho) - \cE (\brho) - \big \langle \frac{\delta \cE}{\delta \rho} (\brho) , \rho - \brho \big \rangle \Big )
\\
&\quad = \big \langle \frac{\delta \cE}{\delta \rho} (\rho)  , - \dx (\rho u) \big \rangle 
- \big \langle \frac{\delta \cE}{\delta \rho} (\brho) , - \dx (\brho \baru) \big \rangle 
\\
&\qquad - \big \langle \big \langle \frac{\delta^2 \cE}{\delta \rho^2} (\brho) , ( \brho_t , \rho - \brho ) \big \rangle \big \rangle 
+  \big \langle \frac{\delta \cE}{\delta \rho} (\brho) , \dx ( \rho u - \brho \baru) \big \rangle  
\\
&\quad \stackrel{(\ref{funsys})}{=} \big \langle \frac{\delta \cE}{\delta \rho} (\rho)  , - \dx (\rho u) \big \rangle 
- \big \langle \frac{\delta \cE}{\delta \rho} (\brho) , - \dx (\brho \baru) \big \rangle 
+ \big \langle \frac{\delta \cE}{\delta \rho} (\brho) , \dx \big (\rho (u - \baru)\big ) \big \rangle  
\\
&\qquad + \big \langle \big \langle \frac{\delta^2 \cE}{\delta \rho^2} (\brho) , ( \dx (\brho \baru) , \rho - \brho ) \big \rangle \big \rangle 
+  \big \langle \frac{\delta \cE}{\delta \rho} (\brho) , \dx  \big ( ( \rho - \brho)   \baru \big ) \big \rangle   
\\
&\stackrel{(\ref{varder1}) , (\ref{varder2})}{=}   \big \langle S_{i j} (\rho)  , \frac{\del u_i}{\del x_j} \big \rangle - \big \langle S_{i j} (\brho)  , \frac{\del \baru_i}{\del x_j} \big \rangle 
- \big \langle \frac{\delta \cE}{\delta \rho} (\rho) - \frac{\delta \cE}{\delta \rho} (\brho) , \dx \big (\rho (u - \baru)\big ) \big \rangle 
\\
&\qquad - \big \langle S_{i j} (\rho)  ,  \frac{\del}{\del x_j} ( u_i - \baru_i ) \big \rangle  
   - \int  \left \langle  \frac{\delta S_{ij} }{\delta \rho} (\brho) , \rho - \brho  \right \rangle   \frac{\del \baru_i }{\del x_j} \, dx
\\
&\qquad =  \int  \Big ( S_{i j} (\rho) - S_{i j}  (\brho) - \big \langle  \frac{\delta S_{ij} }{\delta \rho} (\brho) , \rho - \brho  \big \rangle \Big ) \frac{\del \baru_i }{\del x_j} \, dx
- \big \langle \frac{\delta \cE}{\delta \rho} (\rho) - \frac{\delta \cE}{\delta \rho} (\brho) , \dx \big (\rho (u - \baru) \big ) \big \rangle 
\end{align*}

We next define the relative stress tensor 
\begin{equation}
\label{relstress}
S ( \rho | \brho) := S(\rho) - S (\brho) - \big \langle \frac{\delta S}{\delta \rho} (\brho) , \rho - \brho \big \rangle
\end{equation}
and conclude with the relative potential energy balance
\begin{equation}
\label{relpote}
\frac{d}{dt} \cE (\rho | \brho) = \int S_{i j} (\rho | \brho )  \frac{\del \baru_i }{\del x_j} \, dx
- \big \langle \frac{\delta \cE}{\delta \rho} (\rho) - \frac{\delta \cE}{\delta \rho} (\brho) , \dx \big (\rho (u - \baru)\big ) \big \rangle \, .
\end{equation}

\subsection{The relative kinetic energy} \label{sec-funcrel-kine}
Next consider the kinetic energy functional
\begin{equation}
\label{kinedef}
K (\rho, m) = \int \frac{1}{2} \frac{ |m|^2}{\rho} dx
\end{equation}
viewed as a functional on the density $\rho$ and the momentum $m = \rho u$. The integrand of
this functional,
$k(\rho, m) = \tfrac{1}{2} \frac{ |m|^2}{\rho}$, has as its Hessian the  $(d+1) \times (d+1)$ matrix
$$
\nabla^2_{(\rho, m)} k(\rho, m) =
 \begin{pmatrix}
       \displaystyle{ \frac{|m|^{2}}{\rho^{3}}} &
       \displaystyle{ - \frac{m}{\rho^{2}}}
       \\
	\displaystyle{- \frac{m^T}{\rho^{2}}} & 
	\displaystyle{\frac{1}{\rho} \mathbb{I}_{d\times d} } 
    \end{pmatrix} \; ,
$$
which has eigenvalues
\begin{equation*}
    \lambda_{1}=0,\ \lambda_{2} = ...  = \lambda_{d} = \frac{1}{\rho}> 0,\ \lambda_{d+1} = 
    \frac{1}{\rho} + \frac{|m|^{2}}{\rho^{3}}>0 \, ,
\end{equation*}
and, given a vector $A = (a , b )^T$, $a \in \R$, $b \in \R^d$, its associated quadratic form is given by
$$
A \cdot ( \nabla^2_{(\rho, m)} k(\rho, m)  )  A = 
( a, b^T  ) \cdot 
\begin{pmatrix}
       \displaystyle{ \frac{|m|^{2}}{\rho^{3}}} &
       \displaystyle{ - \frac{m}{\rho^{2}}}
       \\
	\displaystyle{- \frac{m^T}{\rho^{2}}} & 
	\displaystyle{\frac{1}{\rho} \mathbb{I}_{d\times d} } 
    \end{pmatrix} 
\begin{pmatrix}
a \\ b
\end{pmatrix}
=   \frac{1}{\rho} \left | \frac{m}{\rho} a  -b \right |^2 .
$$
Therefore, it is positive semidefinite for $\rho >  0$ and as a consequence the kinetic energy functional $K(\rho, m)$ is convex  (though not strictly convex) as a functional in $(\rho , m)$.

The relative kinetic energy is easily expressed in the form
\begin{equation}
\label{relkinenergy}
\begin{aligned}
K ( \rho, m | \brho, \barm ) &: = \int k(\rho, m) - k(\brho, \barm) - \nabla k (\brho, \barm) \cdot  (\rho - \brho, m - \barm) \, dx
\\
&= \int \frac{1}{2} \frac{ |m|^2}{\rho}  - \frac{1}{2} \frac{ |\barm|^2}{\brho} 
- \big ( -\frac{1}{2} \frac{|\barm |^2}{\brho^2} , \frac{\barm}{\brho} \big ) \cdot (\rho - \brho, m - \barm) \,  dx
\\
&= \int \frac{1}{2} \rho | u - \baru|^2 \, dx\, .
\end{aligned}
\end{equation}

To compute the evolution of the relative kinetic energy, we consider the difference of the two equations \eqref{funsys}$_2$
satisfied by $(\rho, u)$ and $(\brho, \baru)$,
$$
\del_t (u-\baru) + (u \cdot \nabla) (u-\baru)   + \big ( (u - \baru) \cdot \nabla ) \baru 
= - \nabla \Big ( \frac{\delta \cE}{\delta \rho} (\rho) - \frac{\delta \cE}{\delta \rho} (\brho) \Big ) - \zeta (u - \baru)
$$
and take the inner product with $(u - \baru)$ to deduce
$$
 \frac{1}{2} \frac{D}{Dt}   | u - \baru |^2  + \nabla \baru : (u - \baru) \otimes (u - \bar u) = 
- (u - \bar u) \cdot \nabla \Big ( \frac{\delta \cE}{\delta \rho} (\rho) - \frac{\delta \cE}{\delta \rho} (\brho) \Big )
- \zeta |u - \baru|^2\, ,
$$
where $\frac{D}{Dt} = \frac{\del}{\del t} + u \cdot \nabla$ is the material derivative with respect to $u$.
Using \eqref{funsys}$_1$, this is  expressed in the form
$$
\begin{aligned}
\del_t \left (\tfrac{1}{2} \rho |u - \baru|^2\right ) + \div \left (\tfrac{1}{2} \rho u |u - \baru|^2\right ) = - \rho \nabla \baru : (u - \baru) \otimes (u - \bar u)
\\
- \rho (u - \baru) \cdot \nabla \Big ( \frac{\delta \cE}{\delta \rho} (\rho) - \frac{\delta \cE}{\delta \rho} (\brho)  \Big ) - \zeta \rho |u - \baru|^2\, .
\end{aligned}
$$
Integrating over space leads to the balance of the  relative kinetic energy
\begin{equation}
\label{relkine}
\begin{aligned}
&\frac{d}{dt} \int \frac{1}{2} \rho |u - \baru|^2 \, dx + \zeta \int \rho |u - \baru|^2 \, dx 
\\
&\qquad = - \int \rho \nabla \baru : (u - \baru) \otimes (u - \bar u) \, dx 
+ \big \langle \frac{\delta \cE}{\delta \rho} (\rho) - \frac{\delta \cE}{\delta \rho} (\brho) , \dx\big( \rho (u - \baru) \big)\big \rangle  \, .
\end{aligned}
\end{equation}

\subsection{The functional form of the relative energy formula} \label{sec-funcrel-tote}
Combining \eqref{relpote} with \eqref{relkine} we obtain the equation for
the evolution of the (total) relative  energy
\begin{equation}
\label{reltote}
\begin{aligned}
&\frac{d}{dt} \left ( \cE (\rho | \brho) +  \int \tfrac{1}{2} \rho |u - \baru|^2 \, dx \right ) + \zeta \int  \rho |u - \baru|^2 \, dx
\\
&\qquad \qquad = \int  \nabla \baru : S (\rho | \brho )   \, dx  - \int \rho \nabla \baru : (u - \baru) \otimes (u - \bar u) \, dx \, ,
\end{aligned}
\end{equation}
where $\cE (\rho | \brho)$ and $S(\rho | \brho)$ stand for the relative potential energy and relative stress functionals 
defined in \eqref{relendef} and \eqref{relstress}, respectively.

An interesting feature of these calculations is how the contributions of the term
$$
D = \big \langle \frac{\delta \cE}{\delta \rho} (\rho) - \frac{\delta \cE}{\delta \rho} (\brho) , \dx \big ( \rho (u - \baru)\big) \big \rangle 
$$
in \eqref{relpote} and \eqref{relkine} offset each other, in complete analogy to the workings of the derivation
of the total energy \eqref{tote} from the potential \eqref{pote} and kinetic \eqref{kine} energies. The errors are again
formally quadratic in nature as in the corresponding calculations of \cite{Daf79} for the system of thermoelasticity,
but now functionals are involved in the final formulas.

While this abstract derivation has some elegance and ease of derivation, it has the drawback that it requires smoothness for the fields $(\rho, u)$ and
$(\brho, \baru)$. In applications the relative energy is often used to compare a weak to a strong solution; in such cases the calculation has to be rederived by alternate means
for solutions of limited smoothness.
Indeed, this is done for the Euler-Korteweg system in Section \ref{subsec:weakdissKort} and Appendix \ref{subsec:relenKort}. 

\medskip

The calculation is next applied to various specific examples: the Euler-Korteweg system, the Euler-Poisson system, and order parameter approximations
of the Euler-Korteweg system.

\subsection{The Euler-Korteweg system} \label{sec-EK}
The functional 
\begin{equation}
\label{kortfunc}
\cE (\rho) = \int F(\rho, \nabla \rho ) \, dx
\end{equation}
generated by the smooth function $F = F(\rho , q) : \R^+ \times \R^d \to \R$ is associated to 
the Korteweg theory of capillarity;  the dynamics of \eqref{funsys} or \eqref{wkfunsys} 
along the functional \eqref{kortfunc} generates the Euler-Korteweg system.

For the functional \eqref{kortfunc} we next explore the precise meaning of the formulas \eqref{funchyp1} and \eqref{funchyp2}. 
Note that, for a test function $\psi$, 
\begin{equation}
\label{firstvar}
\begin{aligned}
d\cE (\rho ; \psi ) &= \frac{d}{d\tau} \cE (\rho + \tau \psi ) \Big |_{\tau = 0} 
\\
&= \frac{d}{d\tau}  \Big |_{\tau = 0}  \int F(\rho + \tau \psi , \nabla \rho + \tau \nabla \psi) \, dx
\\
&= \int \Big( F_\rho (\rho, \nabla \rho) \psi + F_q (\rho, \nabla \rho) \cdot \nabla \psi \Big)\, dx
\\
&= : \big \langle \frac{\delta \cE}{\delta \rho }(\rho) , \psi  \big \rangle \, .
\end{aligned}
\end{equation}
This defines  the meaning of the bracket $\langle \cdot , \psi \rangle$ as the duality bracket between $H^{-1}$ and $H_0^1$ and identifies 
the variational derivative $\frac{\delta \cE}{\delta \rho}$ as
\begin{equation}
\frac{\delta \cE}{\delta \rho} (\rho)   = F_\rho (\rho , \nabla \rho)  - \nabla \cdot F_q (\rho , \nabla \rho)\, .
\end{equation}

Next, we compute
\begin{equation}
\label{secondvar}
\begin{aligned}
d^2 \cE (\rho;  \psi, \phi) 
&= \lim_{\eps \to 0}  
\frac{ \big \langle \frac{\delta \cE}{\delta \rho }(\rho + \eps \phi) , \psi  \big \rangle - \big \langle \frac{\delta \cE}{\delta \rho }(\rho) , \psi  \big \rangle }{\eps}
\\
&= \frac{d}{d\eps} \Big |_{\eps = 0} \int \Big (F_\rho (\rho + \eps \phi, \nabla\rho + \eps \nabla \phi ) \psi + F_q ( \rho + \eps \phi, \nabla\rho + \eps \nabla \phi ) \cdot \nabla \psi \Big )\, dx
\\
&= \int \Big (F_{\rho \rho} \, \phi \psi + \psi F_{\rho q} \cdot \nabla \phi + \phi  F_{ \rho q} \cdot \nabla \psi + F_{q q} : (\nabla _x\phi \otimes \nabla \psi) \Big )\, dx
\\
&= \int (\phi, \nabla \phi) \cdot  
\begin{pmatrix}  F_{\rho \rho} & F_{\rho q} \\ F_{\rho q} & F_{qq} \end{pmatrix} (\rho, \nabla \rho) 
\begin{pmatrix} \psi \\ \nabla \psi  \end{pmatrix} \, dx
\\
&= :  \left \langle  \left \langle \frac{\delta^2 \cE }{\delta \rho^2} (\rho) , ( \psi, \phi ) \right \rangle \right \rangle\, ,
\end{aligned}
\end{equation}
where the last equation defines the meaning of the bracket $\langle \langle \cdot , (\phi, \psi) \rangle \rangle$ and we have used the notations $F_q = \nabla_q F$
and $F_{qq} = \nabla^2_q F$. It is clear that $d^2 \cE (\rho;  \psi, \varphi) $ is a bilinear form and also that the convexity of $F$ implies that the second variation
of $\cE$ is positive.

Using \eqref{firstvar} (for the test function $\psi = \rho - \brho$) we compute the relative potential energy in \eqref{relendef}
and express it in the form
\begin{equation}
\label{relpotekort}
\begin{aligned}
\cE (\rho | \brho) &= \int \Big (F(\rho, \nabla \rho) - F(\brho, \nabla \brho) - F_\rho ( \brho, \nabla \brho) (\rho - \brho) - F_q ( \brho, \nabla \brho) \cdot (\nabla \rho - \nabla \brho)\Big) \, dx
\\
&= \int F (\rho, \nabla \rho \,  | \,  \brho , \nabla \brho ) \, dx \, ,
\end{aligned}
\end{equation}
where $F(\rho, q | \brho , \barq)$ stands for the quadratic Taylor polynomial of $F(\rho, q)$ around $(\brho, \barq)$.

The Euler-Korteweg system (with friction when $\zeta >0$)  of the form \eqref{funsys} generated by the Korteweg functional \eqref{kortfunc}
takes the form
\begin{equation}
\label{euler-kort}
\begin{aligned}
\frac{\del \rho}{\del t} + \dx (\rho u) &= 0 
\\
\rho \Big [ \frac{\del u}{\del t} +  (u \cdot \nabla) u \Big ]  &= - \rho \nabla  \big ( F_\rho - \nabla \cdot  F_q \big )  - \zeta \rho u\, .
\end{aligned}
\end{equation}

By direct computation, one checks the formula
\begin{equation}
\label{formula}
\begin{aligned}
-\rho \frac{\del}{\del x_i} \frac{\delta \cE}{\delta \rho} 
&= -\rho \frac{\del}{\del x_i}  \big ( F_\rho - \frac{\del}{\del x_k } F_{q_k} \big )
\\
&=   - \frac{\del}{\del x_i}  \big ( \rho F_\rho -  \rho \frac{\del}{\del x_k } F_{q_k} \big ) + \frac{\del \rho}{\del x_i} F_\rho + \frac{\del^2 \rho}{\del x_i \del x_k}  F_{q_k}  
-  \frac{\del}{\del x_k} \big ( \frac{\del \rho}{\del x_i} F_{q_k} \big ) 
\\
&= \frac{\del}{\del x_j} \Big (  \big (  F - \rho F_\rho + \rho \frac{\del}{\del x_k} F_{q_k} \big  )  \, \delta_{i j} - \frac{\del \rho}{\del x_i} F_{q_j} \Big )
\\
&=  \frac{\del}{\del x_j} S_{ij}\, ,
\end{aligned}
\end{equation}
where the Korteweg stress tensor is defined by
\begin{equation}
\label{kortstres}
\begin{aligned}
S &=   ( F - \rho F_\rho  + \rho \dx F_q  )  \mathbb{I} -  \nabla \rho \otimes F_q  
\\[4pt]
&= \Big [  ( F - \rho F_\rho - (\nabla \rho)  \cdot F_q) + \dx ( \rho F_q) \Big ]  \mathbb{I} - \nabla \rho \otimes F_q \, ,
\end{aligned}
\end{equation}
with $\mathbb{I}$ the identity matrix and $P = \rho F_\rho - F$ the pressure. 
For the stress tensor to be symmetric it is often required that $F_q(rho, \nabla \rho) = a(\rho, \nabla \rho) \nabla \rho$
with  $a$ a scalar valued function.
The following  expression of $S(\rho)$ will be convenient for expressing the relative stress:
\begin{equation}
\label{stressreexp}
\begin{aligned}
S_{i j} (\rho) &=   \big ( F - \rho F_\rho -  \frac{\del \rho }{\del x_k} F_{q_k}  \big ) \, \delta_{i j}  +  \frac{\del}{\del x_k} \big ( \rho  F_{q_k}  \big ) \, \delta_{i j}  
- \frac{\del \rho}{\del x_i} F_{q_j}
\\
&= - s (\rho, \nabla \rho) \delta_{i j} +  \Big ( \frac{\del}{\del x_k} r_k (\rho, \nabla \rho) \Big )  \delta_{i j} - H_{i j} (\rho, \nabla \rho)\ , 
\end{aligned}
\end{equation}
where the  functions $s : \R^+ \times \R^d \to \R$, $r : \R^+ \times \R^d \to \R^d$ and $H : \R^+ \times \R^d \to \R^{d \times d}$ are defined by
\begin{equation}
\label{stressdeffunc}
\begin{aligned}
s (\rho, q) &= \rho F_\rho (\rho, q)  + q \cdot F_q (\rho, q)   - F (\rho, q)\, ,
\\
r (\rho, q) &= \rho F_q (\rho, q)\, ,
\\
H(\rho , q) &= q \otimes F_q (\rho, q) \, .
\end{aligned}
\end{equation}

We  proceed to compute the relative stress tensor defined in 
\eqref{relstress}. $S(\rho)$ is viewed as a functional and 
using \eqref{stressreexp} we compute, for a test function $\phi$, the first variation of $S_{ij}(\rho)$ via the formula
\begin{align}
&\frac{d}{d\tau} \Bigg |_{\tau =0} S_{ij} (\rho + \tau \phi) 
\nonumber
\\
&\quad =
- ( s_\rho \phi +  s_q \cdot \nabla \phi )  \delta_{i j} 
+ \frac{\del}{\del x_k} \Big ( \frac{ \del r_k}{\del \rho} \phi + \frac{\del r_k}{\del q_l} \frac{\del \phi}{\del x_l} \Big ) \delta_{i j} 
- \Big ( \frac{\del H_{i j} }{\del \rho} \phi + \frac{\del H_{i j} }{\del q_l } \frac{\del \phi}{\del x_l} \Big )
\nonumber
\\
&\quad = \langle \frac{\delta S_{i j} }{\delta \rho} , \phi \rangle\, .
\label{stressvarkort}
\end{align}
Note that \eqref{stressvarkort} gives a meaning to the bracket defining the first variation of the Korteweg
stress tensor.

In turn, using \eqref{stressreexp},  \eqref{relstress} and \eqref{stressvarkort} 
(with $\rho \to \brho$ and for the choice of  $\phi = \rho - \brho$)  we obtain an expression for  the relative stress  Korteweg tensor:
\begin{equation}
\label{relstresskort}
\begin{aligned}
S_{ij} ( \rho | \brho)  &= S_{i j} (\rho) - S_{ij}  (\brho) - \big \langle \frac{\delta S_{ij} }{\delta \rho} (\brho) , \rho - \brho \big \rangle
\\
&= - s ( \rho, \nabla \rho | \brho , \nabla \brho ) \delta_{i j}  + \frac{\del}{\del x_k} \Big ( r_k ( \rho, \nabla \rho | \brho , \nabla \brho ) \Big ) \delta_{i j} 
-  H_{i j}  ( \rho, \nabla \rho | \brho , \nabla \brho )\, ,
\end{aligned}
\end{equation}
where the notation
\begin{equation}
\label{defrels}
s (\rho, q | \brho , \bar q) = s (\rho, q) - s(\brho, \bar q) - s_\rho (\brho, \bar q) (\rho - \brho) - s_q (\brho, \bar q) \cdot (q - \bar q)
\end{equation}
denotes the relative $s$ function: the difference between $s(\rho, q)$ and  the linear part of its Taylor expansion around $(\bar\rho, \bar q)$,
and similarly for the relative functions $r (\rho, q | \bar\rho, \bar q)$ and $H(\rho, q | \bar\rho, \bar q)$.

We conclude by stating the relative energy formula induced by \eqref{reltote} for the specific case of the Euler-Korteweg system \eqref{euler-kort}.
Using \eqref{relpotekort} and \eqref{relstresskort}, we end up with
\begin{equation}
\label{reltotekort}
\begin{aligned}
&\frac{d}{dt} \left (  \int  F (\rho, \nabla \rho \,  | \,  \brho , \nabla \brho ) + \tfrac{1}{2} \rho |u - \baru|^2 \, dx \right ) + \zeta \int  \rho |u - \baru|^2 \, dx
\\
&\quad =  -  \int  \Big [  (\div \baru)  \, s ( \rho, \nabla \rho | \brho , \nabla \brho ) + \Big ( \frac{\del}{\del x_k} \div \baru \Big ) r_k (\rho, \nabla \rho | \brho , \nabla \brho )
\\
&\qquad \qquad \qquad \qquad \qquad \qquad \qquad 
+  \nabla \baru : H( (\rho, \nabla \rho | \brho , \nabla \brho ) \Big ] \, dx
\\
&\quad \quad   - \int \rho \nabla \baru : (u - \baru) \otimes (u - \bar u) \, dx \, .
\end{aligned}
\end{equation}

\medskip
\noindent
Special choices of the energy functional \eqref{kortfunc} lead to some frequently 
occurring  systems in fluid dynamics. Some of them are reviewed below:

\subsubsection{The Euler system of isentropic gas flow}  The choice
\begin{align}
\label{gasdynfunc}
\cE(\rho) = &\int h(\rho) \, dx \,  , \quad S = - p(\rho) \mathbb{I}  \, , 
\\
&p(\rho) = \rho h'(\rho) - h(\rho)\, ,
\label{Gibbs-Duhem}
\end{align}
produces the Euler system of compressible isentropic gas dynamics. The resulting relative energy formula coincides with the one computed in
\cite{Daf79b} and \cite{LT13}.

\subsubsection{Special instances of the Euler-Korteweg system}
An often used functional, within the framework of the Korteweg theory, is
\begin{equation}
\label{specialkort}
\cE(\rho) = \int \Big (h(\rho) + \tfrac{1}{2} \kappa (\rho) |\nabla \rho |^2 \Big )  \, dx \, .
\end{equation}
The formulas \eqref{formula}  and  \eqref{stressreexp}  now become
\begin{align}
\label{specialformula}
&- \rho \nabla \Big ( h'(\rho) + \tfrac{1}{2} \kappa'(\rho) |\nabla \rho|^2 - \dx (\kappa(\rho) \nabla \rho) \Big )  = \nabla \cdot S\ , 
\\
\label{specialkortstress}
S &=  \Big [  - p(\rho)  - \tfrac{1}{2} \big ( \rho \kappa'(\rho)  + \kappa(\rho) \big ) |\nabla \rho|^2   + \nabla \cdot \big (\rho \kappa (\rho) \nabla \rho \big )  \Big ] \mathbb{I} - \kappa (\rho) \nabla \rho \otimes \nabla \rho \, ,
\end{align}
and the system \eqref{wkfunsys} takes the form
\begin{equation}
\label{specialkortsys}
\begin{aligned}
\frac{\del \rho}{\del t} + \dx (\rho u) &= 0 
\\
\frac{\del}{\del t}( \rho u) + \dx (\rho u \otimes u) &=  \nabla \Big (  - p(\rho)  -  \tfrac{1}{2}(\rho \kappa'(\rho)  + \kappa(\rho) ) |\nabla \rho|^2  + \div  \big ( \rho \kappa(\rho) \nabla \rho \big )     \Big ) 
\\
&\quad - \div (\kappa(\rho) \nabla \rho \otimes \nabla \rho )  - \zeta \rho u\, ,
\end{aligned}
\end{equation}  
where $p = \rho h' - h$. As shown later in Lemma \ref{lem:Funifconvex} the convexity of the functional \eqref{specialkort} is equivalent to the hypotheses
\begin{equation*}
h''(\rho) = \frac{p'( \rho)}{\rho} > 0 \, , \quad   \kappa(\rho) > 0 \, , \quad  \kappa(\rho) \kappa'' (\rho) - 2( \kappa'(\rho) )^2 \geq  0\, .
\end{equation*}

\subsubsection{The Navier-Stokes-Korteweg system}
The Euler-Korteweg model can be augmented by viscosity leading to the
 isothermal Navier-Stokes-Korteweg (NSK) system which is a well-known model for compressible liquid-vapor flows undergoing phase transitions.
It is a so-called diffuse interface model in which the fields are not discontinuous at the phase boundary, but undergo change smoothly from states in the one phase to states 
in the other, though usually steep gradients do occur. 
For the choice $\kappa(\rho)=C_\kappa=const.$\ the NSK model reads
\begin{equation}\label{NSK-sec}
 \begin{split}
   \rho_t + \dx(\rho u) &=0\\
   (\rho  u)_t + \dx (\rho  u \otimes  u) + \nabla p(\rho) &= \dx(\sigma[ u]) + C_{\kappa} \rho \nabla \triangle \rho\, ,
 \end{split}
\end{equation}
where 
\begin{equation}
\label{def:nss-sec}
\sigma[u]:= \lambda \dx( u) \mathbb{I}+ \mu (\nabla u + (\nabla u)^T)
\end{equation}
is the Navier-Stokes stress with coefficients $\lambda, \mu$ satisfying $\mu\geq0$ and $\lambda + \frac{2}{d} \mu \geq 0$.
Note that following \eqref{kortstres}  the third order term in the momentum balance can be written in divergence form.
The potential energy for NSK is the same as for Euler-Korteweg.
Adding in viscosity introduces a dissipative mechanism which actually helps with our relative energy calculations and which increases regularity of solutions.

The first order part of \eqref{NSK-sec} is hyperbolic when $p$ is monotone.
 The non-monotone pressure makes the first order part of \eqref{NSK-sec}
 a system of mixed hyperbolic-elliptic type.
It should be emphasized that
for \eqref{NSK-sec} to describe multi-phase flows it is mandatory that the energy density $h=h(\rho),$ 
related to the pressure by \eqref{Gibbs-Duhem},  is non-convex, which makes the pressure non-monotone.
We refer to \cite{DS85} for a discussion of the thermodynamic structure of \eqref{NSK-sec} and its relation to higher-order gradient theories.
Also, to \cite{B10,BDDJ05} for a discussion of the general structure of Euler-Korteweg and Navier-Stokes-Korteweg models.

\subsubsection{The Quantum Hydrodynamics system}\label{subsec:quantum}
 Another special case arises when in  \eqref{specialkort} we set $\kappa(\rho) = \tfrac{1}{4} \frac{\eps^2 }{\rho}$,
where $\eps$ is a constant (the Planck constant). This leads to the energy
\begin{equation}
\label{qhdenergy}
\cE(\rho) = \int \Big (h(\rho) + \frac{1}{8} \eps^2 \frac{1}{\rho}  |\nabla \rho |^2\Big )  \, dx  = \int \Big (h(\rho) +  \frac{1}{2} \eps^2 | \nabla \sqrt{\rho} |^2 \Big )\, dx\, .
\end{equation}
In that case we have the identities
\begin{equation}
\begin{aligned}
\frac{1}{2} \rho \nabla \Big ( \frac{1}{\sqrt{\rho}} \triangle \sqrt{\rho} \Big ) 
&\quad \; \; = \quad \; \;  \rho \nabla \Big ( \frac{1}{8 \rho^2} |\nabla \rho |^2 + \dx \big ( \frac{1}{4\rho} \nabla \rho \big )  \Big )
\\
&\stackrel{(\ref{specialformula}) , (\ref{specialkortstress})}{=} \nabla \cdot \Big ( \frac{1}{4} \triangle \rho\, \mathbb{I} - \frac{1}{4 \rho} \nabla \rho \otimes \nabla \rho \Big )
\end{aligned}
\end{equation}
and \eqref{specialkortsys} becomes (for $\zeta = 0$) the quantum hydrodynamics system (QHD) 
\begin{equation}
\label{qhd}
\begin{aligned}
\frac{\del \rho}{\del t} + \dx (\rho u) &= 0 \, ,
\\
 \frac{\del  \rho u}{\del t} +  \dx ( \rho u \otimes u ) + \nabla p(\rho)  &=  \tfrac{\eps^2}{2} \rho \nabla \Bigg ( \frac{\triangle \sqrt{\rho} }{\sqrt{\rho}} \Bigg )\, .
\end{aligned}
\end{equation}
We refer to \cite{AM09,AM12} for details on the existence theory of finite energy weak solutions of the QHD system and detailed references on this interesting subject.

\subsection{The Euler-Poisson system} \label{sec-EP}
As a second application we consider the Euler-Poisson system,
\begin{equation}
\label{euler-poisson}
\begin{aligned}
\frac{\del \rho}{\del t} + \dx (\rho u) &= 0 \, ,
\\
\rho \Big ( \frac{\del u}{\del t} +  (u \cdot \nabla) u \Big )  &= - \nabla p(\rho) + \rho \nabla c\, ,
\\[2pt]
- \triangle c + \beta c &= \rho \;  - < \rho >\, ,
\end{aligned}
\end{equation}
which is often used for describing charged gases in semiconductor devices or gases moving under the influence of a gravitational field. 
Here, $p(\rho)$ is the pressure of the gas of charged
particles and $\nabla c$ is the electrostatic force induced by the charged particles. The constant $\beta \ge 0$ is often referred as screening
constant while $<\rho> = \int \rho dx$ stands for the average charge in the torus $\T^d$ or the total charge in  $\R^d$.
Here, we restrict on the case of the torus $\T^d$ and leave to the reader to provide the necessary modifications for $\R^d$.
Also, we restrict to the frictionless case ($\zeta = 0$), it is straightforward to adapt and include frictional effects.
Our objective is to recast \eqref{euler-poisson} under the framework of \eqref{funsys}; as a byproduct we will infer a relative energy
calculation.

Consider the functional
\begin{equation}
\label{epfunc1}
\begin{aligned}
\cE (\rho) = \int \Big(h(\rho) - \tfrac{1}{2} \rho c \Big)\, dx\, ,&
\\
\mbox{where $c$ is the solution of} \qquad - \triangle c + \beta c &= \rho \;  - < \rho >\, .
\end{aligned}
\end{equation}
The elliptic equation \eqref{epfunc1}$_2$ is solvable on $\T^d$. It has a unique solution
for $\beta >0$ while the solution is given in terms of an arbitrary constant for $\beta = 0$. This constant plays no role in determining the electrostatic force $\nabla c$
and might be fixed by requiring $\int c = 0$ for $\beta = 0$. One may express the solution operator in terms of the Green's function
$$
c (x) = (\mathcal{K} \ast \rho) (x) = \int \mathcal{K}( x - y) \rho (y) \, dy\, ,
$$
where  $\mathcal{K}$ is a symmetric function, and $\cE$ is expressed in the form
\begin{equation}
\label{epfunc}
\cE (\rho) = \int \Big ( h(\rho) - \tfrac{1}{2} \rho ( \mathcal{K} \ast \rho)\Big) \, dx\, .
\end{equation}
The directional derivative of $\cE$ is computed by using the symmetry of $\mathcal{K}$, 
\begin{equation}
\label{firstvarep}
\begin{aligned}
d\cE (\rho ; \psi ) 
&= \frac{d}{d\tau} \cE (\rho + \tau \psi ) \Big |_{\tau = 0} 
\\
&= \int \Big (h' (\rho) \psi - \tfrac{1}{2}  \psi (\mathcal{K} \ast \rho )  -   \tfrac{1}{2} \rho (\mathcal{K} \ast \psi)\Big ) \, dx
\\
&= \int (h' (\rho) - \mathcal{K} \ast \rho) \psi \, dx
\\
&= \big \langle \frac{\delta \cE}{\delta \rho }(\rho) , \psi  \big \rangle\, .
\end{aligned}
\end{equation}
Accordingly,  for $p (\rho)$ satisfying \eqref{Gibbs-Duhem},  
the Euler-Poisson system is expressed as a
Hamiltonian flow  \eqref{funsys} for the energy functional \eqref{epfunc}.

Next, we prove 
\begin{align}
- \rho \nabla \frac{\delta \cE}{\delta \rho }(\rho) &= - \rho \nabla  ( h'(\rho) - c) = \div S\, ,
\label{hamflowep}
\end{align}
where
\begin{align}
S &= - \big ( p(\rho) -\tfrac{1}{2} |\nabla c|^2 - \tfrac{\beta}{2} c^2 - < \rho > c \big ) \mathbb{I} - \nabla c \otimes \nabla c \, .
\label{stressep}
\end{align}
Note that \eqref{hamflowep} validates the hypothesis \eqref{stresshyp} and its weak form \eqref{varder1} for the case of the Euler-Poisson system with the stress functional
$S(\rho)$ defined by \eqref{stressep}. To prove \eqref{hamflowep}, we multiply 
\eqref{euler-poisson}$_3$ by $\nabla c$ then after some rearrangement of terms we obtain
$$
\rho \nabla c = \nabla \Big ( \tfrac{1}{2} |\nabla c|^2 + \tfrac{\beta}{2} c^2 + <\rho> c \Big ) - \dx ( \nabla c \otimes \nabla c )\, ,
$$
which readily provides, for $p' = \rho h''$, the formula
$$
- \rho \nabla  ( h'(\rho) - c) = \div \Big ( - \big [ p(\rho) -  \tfrac{1}{2} |\nabla c|^2 - \tfrac{\beta}{2} c^2 - <\rho> c \big ] \mathbb{I} - \nabla c \otimes \nabla c \Big )
$$
and expresses the stress in the form \eqref{stressep}.

Suppose now that  $(\rho, u)$ with $c = \mathcal{K} \ast \rho$ and $(\brho, \baru)$ with $\barc = \mathcal{K} \ast \brho$ are two solutions of the Euler-Poisson system.
Our goal is to identify the form that the abstract relative entropy identity takes for the specific case of the Euler-Poisson functional \eqref{epfunc}.
First, using \eqref{relendef}, \eqref{epfunc} and \eqref{firstvarep} (for $\psi = \rho - \brho$) we compute
$$
\begin{aligned}
\cE (\rho | \brho) = \int \Big ( h(\rho) - \tfrac{1}{2} \rho c - h(\brho) + \tfrac{1}{2} \brho \barc - ( h' (\brho) - \bar c) (\rho - \brho)\Big ) \, dx\, .
\end{aligned}
$$
Next, use the formulas
$$
\begin{aligned}
h(\rho | \brho) &:= h(\rho) - h(\brho) - h' (\brho)  (\rho - \brho)\, ,
\\
\int c \brho &= \int \bar c \rho\, , \qquad \mbox{since $\mathcal{K}$ is symmetric and $c = \mathcal{K} \ast \rho$, $\barc = \mathcal{K} \ast \brho$}\, ,
\end{aligned}
$$
to recast $\cE (\rho | \brho)$ in its final form
\begin{equation}
\label{relpoteep}
\cE (\rho | \brho) = \int \Big ( h(\rho | \brho) - \tfrac{1}{2} (\rho - \brho) \mathcal{K} \ast (\rho - \brho) \Big) \, dx\, .
\end{equation}
Note that for the case considered here, of an electrically attracting fluid, the relative energy consists of two competing terms, and to exploit
the relative energy $\cE (\rho | \brho)$ additional considerations will be needed. These are undertaken in a companion work \cite{LTforth}.

The relative stress functional $S ( \rho | \brho)$  is computed using \eqref{relstress}  for the Euler-Poisson case.  First, we compute
the directional derivative of the stress functional \eqref{stressep}: for $\psi$ a test function,  $c= \mathcal{K} \ast \rho$ and $C = \mathcal{K} \ast \psi$, the variation of $S(\rho)$ is
given by
\begin{equation}
\label{firstvarstressep}
\begin{aligned}
&d S [\rho, \psi] = \frac{d}{d\tau} \Bigg |_{\tau = 0} S(\rho + \eps \psi )
\\
&= \frac{d}{d\tau} \Bigg |_{\tau = 0} \Big (  \big [- p(\rho + \tau \psi) + \tfrac{1}{2} |\nabla c + \tau \nabla C |^2  + \tfrac{\beta}{2} (c + \tau C)^2 + <\rho + \tau \psi> (c + \tau C)  \big ] \mathbb{I}
\\
&\qquad \qquad - (\nabla c + \tau \nabla C) \otimes (\nabla c + \tau \nabla C) \Big )
\\
&= \big [ -p'(\rho) \psi + \nabla (\mathcal{K}\ast \rho) \cdot \nabla (\mathcal{K}\ast \psi) + \beta (\mathcal{K}\ast \rho) (\mathcal{K} \ast \psi) + <\psi> \mathcal{K}\ast \rho + <\rho> \mathcal{K}\ast \psi \big ] \mathbb{I}
\\
&\qquad - \nabla (\mathcal{K} \ast \psi) \otimes \nabla (\mathcal{K} \ast \rho) - \nabla (\mathcal{K} \ast \rho) \otimes \nabla (\mathcal{K} \ast \psi)
\\
&=: \langle \frac{\delta S}{\delta \rho} (\rho) , \psi \rangle\, .
\end{aligned}
\end{equation}
Using \eqref{stressep} and \eqref{firstvarstressep} in \eqref{relstress} we obtain, after rearranging the terms
\begin{equation}
\label{relstressep}
\begin{aligned}
&S(\rho | \brho) = S(\rho ) - S(\brho) - \big \langle \frac{\delta S}{\delta \rho} (\brho) , \rho - \brho \big \rangle
\\
&= \big ( - p(\rho | \brho) + \tfrac{1}{2} |\nabla (c-\barc)|^2 + \tfrac{\beta}{2} (c - \barc)^2 + <\rho - \brho>(c-\barc) \big ) \mathbb{I}
- \nabla (c - \barc) \otimes \nabla (c - \barc)\, .
\end{aligned}
\end{equation}

Using \eqref{relpoteep}, \eqref{relstressep}   we express the relative energy identity \eqref{reltote} for the Euler-Poisson system,
\begin{equation}
\label{reltoteep}
\begin{aligned}
&\frac{d}{dt} \left (  \int \Big (h(\rho | \brho) - \tfrac{1}{2} (\rho - \brho) \mathcal{K} \ast (\rho - \brho)  + \tfrac{1}{2} \rho |u - \baru|^2 \Big) \, dx \right ) 
\\
&\qquad =  \int  \div \baru  \Big ( - p(\rho | \brho) + \tfrac{1}{2} |\nabla (c-\barc)|^2 + \tfrac{\beta}{2} (c - \barc)^2 +  <\rho - \brho>(c-\barc) \Big ) \, dx
\\
&\qquad \qquad - \int  \nabla \baru :  \nabla (c - \barc) \otimes \nabla (c - \barc)  \, dx  - \int \rho \nabla \baru : (u - \baru) \otimes (u - \bar u) \, dx  \, .
\end{aligned}
\end{equation}

\bigskip
\subsection{Order parameter approximation of the Euler-Korteweg system}\label{sec-LOA}
Another example of a system fitting into our framework is the following model introduced in \cite{Roh10}:
\begin{equation}\label{lon}
 \begin{split}
\frac{\partial}{\partial t}   \rho + \div(\rho u) &=0\\
\frac{\partial}{\partial t}   (\rho u) + \div (\rho u \otimes u) + \nabla (p(\rho) + C_\kappa \frac{\alpha}{2} \rho^2)&=
  C_\kappa\alpha \rho \nabla  c\\
  c - \frac{1}{\alpha}\triangle c &= \rho\, ,
 \end{split}
\end{equation}
where  $C_\kappa$ is as in \eqref{NSK-sec}, and $\alpha>0$ is a parameter.
We will show in Section \ref{sec:mc} that for $\alpha \rightarrow \infty$ classical solutions of \eqref{lon} endowed with viscosity converge to solutions of \eqref{NSK-sec}.
The motivation for introducing \eqref{lon}  in \cite{Roh10} was due to numerical considerations.
We will comment on this briefly at the beginning of Section \ref{sec:mc}. 
It should be noted that similar models were derived in \cite{BLR95,RT00} as models in their own right without reference to NSK.

In \cite{Roh10} the potential energy
\begin{equation}\label{loe1}\mathcal{E}(\rho,c,\nabla c):= \int\Big( h(\rho) + \frac{C_\kappa \alpha}{2} (\rho - c )^2 + \frac{C_\kappa}{2}|\nabla c|^2\Big)\, dx
\end{equation}
with $p(\rho)=\rho h'(\rho)- h(\rho)$
was considered. It is important to note that in this sense $c$ is a variable which is independent of $\rho$
and without an immediate physical interpretation.
Using \eqref{loe1} we may write \eqref{lon} as
\begin{equation}\label{lo2}
 \begin{split}
   \rho_t + \div(\rho u) &=0\\
  \rho \frac{\operatorname{D} u}{\operatorname{D} t} &= - \rho \nabla \frac{\delta \mathcal{E}}{\delta \rho}\\
   \frac{\delta \mathcal{E}}{\delta c} &= 0\, .
 \end{split}
\end{equation}

However, if we understand \eqref{lon}$_3$ not as an energy minimization condition but as the definition of $c, $ 
which coincides with the definition in Section \ref{sec-EP} up to an additive constant and replacing $\beta$ by $\alpha^{-1}$,
we may rewrite the energy as follows:
\begin{equation}\begin{split}\label{loe2}
   &\mathcal{E}(\rho):= \int\Big( h(\rho) + \frac{C_\kappa \alpha}{2} \rho ^2 - \frac{C_\kappa \alpha}{2} \rho c\Big)\, dx\\
   &\text{with $c$ solving}\quad  c - \frac{1}{\alpha}\triangle c = \rho\, ,
  \end{split}
 \end{equation}
 since \eqref{lon}$_3$ implies
 \begin{multline} \int\Big( \rho^2 -  \rho c \Big)\, dx
  =\int  \Big((\rho-c) (\rho -c) + c(\rho - c)\Big)\, dx \\
  = \int \Big( (\rho-c)^2 + c (- \frac{1}{\alpha} \triangle c)\Big)
  \, dx
   = \int \Big( (\rho - c )^2 + \frac{1}{2\alpha}|\nabla c|^2
   \Big)\, dx.
 \end{multline}
Thus, we may recast \eqref{lon} as an Euler-Poisson system with modified energy $h.$
The variation of \eqref{loe2} and derivation of the stress tensor are completely analogous to Section \ref{sec-EP}.

Finally, we like to point out another (equivalent) version of the potential energy \eqref{loe1}:
\begin{equation}\begin{split}\label{loe3}
   &\mathcal{E}(\rho)= \int\Big( h(\rho) + \frac{C_\kappa}{2} \nabla \rho \cdot \nabla c\Big) \, dx\\
   &\text{with $c$ solving}\quad  c - \frac{1}{\alpha}\triangle c = \rho\, .
  \end{split}
 \end{equation}
This is based on
 \[ \alpha \int  (\rho^2 - \rho c)\, dx= \int \Big( \rho(- \triangle c)\Big)\, dx
 = \int \Big(\nabla \rho \cdot \nabla c\Big)\, dx\, .\]
 In \eqref{loe3} it is apparent that $\nabla c$ converges to $\nabla \rho$, at least for sufficiently smooth $\rho,$ 
 and, thus, the energy of the model at hand converges to the energy of the Euler-Korteweg model.
Indeed, it was shown in \cite{SV03} that the energy functional of this model $\Gamma$-converges to that of the Euler-Korteweg model for $\alpha \rightarrow \infty.$

\subsection{Remarks on the Hamiltonian structure of the problem}\label{sec-hamrem}
Here we consider the system \eqref{funsys-intro} and briefly outline an idea, adapted from \cite{BDDJ05}, on the relation of 
\eqref{funsys-intro} and Hamiltonian systems.  The reader is refered to \cite[Sec 1.5, 1.6]{MR99} and \cite[\S 30.5]{DFN85}
for various approaches towards the Hamiltonian structure of the (incompressible or compressible) Euler equations and related systems.

Consider the case of three space dimension, $d=3$, and using the vector calculus formula
$$
(u \cdot \nabla) u = \frac{1}{2} \nabla |u|^2 - u \times {\curl}_x  u
$$
rewrite \eqref{funsys-intro} in the form
\begin{equation}
\label{funsys-mod}
\begin{aligned}
\frac{\del \rho}{\del t} &= -  \nabla \cdot  (\rho u) 
\\
 \frac{\del u}{\del t} &=  -  \nabla \Big ( \tfrac{1}{2} |u|^2 +  \frac{\delta \cE}{\delta \rho} \Big )  + u \times {\curl}_x u\, .
\end{aligned}
\end{equation}

Define the Hamiltonian
\begin{equation}
\label{hamdef}
\cH (\rho, u) = \cE (\rho) + \hat K (\rho, u)  = \cE(\rho) + \int \frac{1}{2} \rho |u|^2 \, dx\, .
\end{equation}
Note that the kinetic energy $\hat K$ is viewed here as a functional on $(\rho, u)$; this should be contrasted to \eqref{kinedef}.
Then we easily compute the variational derivatives
$$
\frac{\delta \cH}{\delta \rho} = \frac{\delta \cE}{\delta \rho} + \frac{1}{2} |u|^2 \, , \quad
\frac{\delta \cH}{\delta u} = \rho u
$$
and rewrite \eqref{funsys-mod} as
\begin{equation}
\label{hamiltonianform}
\frac{\del}{\del t} \begin{pmatrix} \rho \\[2pt] u  \end{pmatrix} = 
\begin{pmatrix}  0 &  - \div  \\[2pt]  - \nabla  & 0  \end{pmatrix} \begin{pmatrix}  \frac{\delta \cH}{\delta \rho}  \\[4pt] \frac{\delta \cH}{\delta u} \end{pmatrix} 
+ 
\begin{pmatrix} 0  \\[4pt]  \frac{1}{\rho} \frac{\delta \cH}{\delta u} \times {\curl}_x  ( \frac{1}{\rho}\frac{\delta \cH}{\delta u}) \end{pmatrix}  \, .
\end{equation}
The operator
$$
J = \begin{pmatrix}  0 & - \div  \\[2pt]  - \nabla  & 0  \end{pmatrix}
$$
is skew adjoint and the system \eqref{hamiltonianform}  is Hamiltonian whenever ${\curl}_x u = 0$, but it has an additional term
when $\curl u \ne 0$.  This additional term does not affect the conservation of energy as can be seen starting from \eqref{hamiltonianform} via the following
formal calculation
\begin{equation}
\begin{aligned}
\frac{d}{dt} \cH (\rho, u) &= \big \langle \frac{\delta \cH }{\delta \rho} , \rho_t \big \rangle + \big \langle \frac{\delta \cH}{\delta u} , u_t \big \rangle 
\\
&= - \Big \langle \frac{\delta \cH }{\delta \rho} , \div   \frac{\delta \cH}{\delta u}   \Big \rangle  - \Big \langle \frac{\delta \cH }{\delta u} , \nabla   \frac{\delta \cH}{\delta \rho}   \Big \rangle 
\\
&= 0 \, .
\end{aligned}
\end{equation}
We note that the kinetic energy functional $ \hat K (\rho, u)$ is not convex as a functional in $(\rho, u)$, but it becomes convex when viewed as a functional in
the variables $(\rho, m)$ in \eqref{kinedef}.

%
%


\section{Weak-strong stability for the Euler-Korteweg system}
\label{sec:kort}

In this section we consider  the  Euler-Korteweg system 
\begin{equation}
    \begin{cases}
	\displaystyle{\rho_{t} +\dx (\rho u) =0}&   \\
	& \\
       \displaystyle{ (\rho u)_{t}} + \dx (\rho u \otimes u)
	=  - \rho \nabla \Big ( F_\rho(\rho, \nabla \rho)  - \dx F_q (\rho, \nabla \rho)   \Big )\, , &  \\
    \end{cases}
    \label{eq:Kortcap2}
\end{equation}
where $\rho \geq 0$ is the density, $u$ the velocity, and $F(\rho, q)$ a smooth function standing for the potential energy.
As noted in \eqref{formula} 
\begin{align}
\label{formulakorteweg}
- \rho \nabla \big ( F_\rho - \dx F_q     \big )  = \dx S\, ,  
\end{align}
where 
\begin{equation}
\label{stresskort}
\begin{aligned}
 S 
 &= \Big [ (F - \rho F_\rho - \nabla \rho \cdot F_q ) + \dx \big ( \rho  F_q \big )  \Big ]  \mathbb{I} - \nabla \rho \otimes F_q 
 \\
 &= - s (\rho, \nabla \rho) \mathbb{I}  +  \big ( \dx  r (\rho, \nabla \rho) \big )  \mathbb{I} - H (\rho, \nabla \rho)\ , 
 \end{aligned}
\end{equation}
is the stress tensor of the Korteweg fluid, and the functions $s$, $r$ and $H$ are given in \eqref{stressdeffunc}.
The formula \eqref{formulakorteweg}  enables to express the Euler-Korteweg system in conservative form,
\begin{equation}
\label{Euler-Kort}
\begin{aligned}
\del_t \rho + \dx m &= 0\, ,   
\\
\del_t m + \dx \Big ( \frac{m \otimes m}{\rho} \Big ) &= \dx S\, ,	   
\end{aligned}
\end{equation}
with $m=\rho u$ the momentum variable, and to define weak solutions (see Definition \ref{def:wksol} below).

We proceed along the lines of  \eqref{pote}, \eqref{kine} and \eqref{tote} for the case of \eqref{eq:Kortcap2}. A formal  computation
shows that the potential energy satisfies the equation
$$
\del_t F(\rho, \nabla \rho) + \dx \Big ( m (  F_\rho   - \dx  F_q)  - F_q   \rho_t \Big ) = m \cdot \nabla (  F_\rho   - \dx  F_q   ) \, ,
$$
the kinetic energy satisfies the equation
$$
\del_t ( \tfrac{1}{2} \rho |u|^2 )  + \dx \big (  u \tfrac{1}{2} \rho |u|^2 \big ) = - \rho u \cdot \nabla (  F_\rho - \dx  F_q ) \, , 
$$
and adding the two leads to
\begin{align}
&\partial_t \left  (\frac{1}{2} \frac{|m|^2}{\rho} + F(\rho,\nabla\rho) \right )  \nonumber\\
   &\ + \dx  \left ( \frac{1}{2}m\frac{|m|^{2}}{\rho^{2}} + 
     m \Big( F_\rho(\rho,\nabla\rho)- \dx  \big(F_q(\rho,\nabla\rho)\big) \Big)    + F_q(\rho,\nabla\rho)\dx m 
          \right ) =  0 \, .
    \label{eq:entrKortGen2}
\end{align}


\subsection{Relative energy estimate for weak  solutions}\label{subsec:weakdissKort}
Next, we consider the Euler-Korteweg system \eqref{Euler-Kort} 
and proceed to compare a weak solution $(\rho, m)$ with a strong solution $(\brho, \barm)$ via a relative energy computation.
For simplicity, we focus on periodic solutions defined on $\T^d$  the $d$-dimensional torus.  
We recall:

\begin{definition}  \label{def:wksol}
(i) A function $( \rho,  m)$ with  $\rho \in C([0, \infty) ; L^1 ( \T^d ) )$, $m \in C  \big ( ( [0, \infty) ;  \big ( L^1(\T^d) \big )^d \big )$,
$\rho \ge 0$,  is a weak
solution of \eqref{Euler-Kort}, if $\frac{m \otimes m}{\rho}$, $S \in L^1_{loc}  \left ( [0, \infty) \times \T^d  \right )^{d \times d}$  and $(\rho, m)$ satisfy
\begin{equation}
\begin{aligned}
- \iint \rho \psi_t + m \cdot \nabla \psi dx d\tau  = \int \rho (x, 0) \psi (x, 0) dx \, ,  \qquad  \forall \psi \in C^1_c \left ( [0, \infty) ; C^1 (\T^d) \right )\, ;
\\
- \iint m \cdot \varphi_t + \frac{m \otimes m}{\rho} : \nabla \varphi - S : \nabla \varphi \,  dx dt = \int m(x,0) \cdot \varphi(x,0) dx  \, ,  \qquad \qquad 
\\
 \forall \varphi  \in C^1_c \left ( [0, \infty) ;  \big ( C^1 (\T^d)  \big )^d \right ) \, .
\end{aligned}
\end{equation}

\medskip\noindent
(ii) If,  in addition, $\frac{1}{2} \frac{|m|^2}{\rho} + F(\rho, \nabla \rho) \in C([0, \infty) ; L^1 ( \T^d ) )$ and it satisfies
\begin{equation}
 \label{eq:Kortcap2Integr}
 \begin{aligned}
  - \iint \left ( \frac{1}{2} \frac{|m |^2}{\rho} + F(\rho,\nabla\rho)  \right ) \dot\theta(t) 
 \le   \int \left ( \frac{1}{2} \frac{|m |^2}{\rho } + F(\rho,\nabla\rho) \right ) \Big |_{t=0}  \theta(0)dx \, ,
 \\
 \qquad     \text{for any non-negative}\ \theta  \in W^{1, \infty} [0, \infty) \ \mbox{compactly supported on $[0, \infty)$},
 \end{aligned}
\end{equation}
then $(\rho, m)$ is called a dissipative weak solution. 

 \medskip\noindent
(iii) By contrast,
if $\frac{1}{2} \frac{|m|^2}{\rho} + F(\rho, \nabla \rho) \in C([0, \infty) ; L^1 ( \T^d ) )$ and it satisfies \eqref{eq:Kortcap2Integr}
as an equality, then $(\rho, m)$  is called a conservative weak solution.
\end{definition}

There is no complete agreement about the nature of weak solutions for \eqref{eq:Kortcap2}, that is 
whether one should consider \emph{conservative} or \emph{dissipative} weak solutions.
It appears the appropriate definition might depend on the way the solutions emerge
(whether they emerge by a limiting viscosity mechanism or arise from the nonlinear Schroedinger equation as in the case of the QHD system). 
In any case, here we will allow for both eventualities.
 
\smallskip
We place the following assumptions:
\begin{itemize}
\item[\textbf{(H1)}] $( \rho,  m)$ is  a dissipative (or conservative) weak periodic solution of \eqref{eq:Kortcap2} with $\rho \ge 0$
in the sense of Definition \ref{def:wksol} which satisfies
\begin{align}
\sup_{t\in (0,T)} \int_{\T^d}  \rho  dx &\le K_1 < \infty\, ,
\label{hypCauchyK1}
\\
 \sup_{t\in (0,T)} \int_{\T^d}
 \frac{1}{2} \frac{|  m |^2}{ \rho} + F(  \rho,\nabla \rho)  \, dx
 &\le K_2 < \infty \, .
 \label{hypCauchyK2}
\end{align}
\item[\textbf{(H2)}] $(\bar \rho, \, \bar u)  : (0,T)\times \T^d \to \R^{d+1}$ is  a strong  conservative periodic solution 
of \eqref{eq:Kortcap2} with $\bar\rho\geq 0$ and $\bar m = \brho \baru$. 
The regularity ``strong''  refers to the requirement that $\bar \rho$, $\bar u$ and the derivatives 
$\frac{\del \brho}{\del t}$,  $\frac{\del^2 \brho}{\del t \del x_i }$,
$\frac{ \del^2{\brho}}{\del x_i \del x_j}$, $\frac{ \del^3{\brho}}{\del x_i \del x_j \del x_k}$ 
as well as $\frac{\del \baru_i}{\del t}$, 
$\frac{\del \baru_i}{\del x_j}$ and  $\frac{\del^2 \baru_i}{\del x_i \del x_j}$ are in  $L^\infty \big ( (0, T)\times \T^d \big )$.  
\end{itemize}

For \eqref{hypCauchyK2} to induce useful bounds a convexity condition is imposed on the function $F(\rho, q)$.
The same condition is also used to exploit the relative energy identity. Instead of specifying hypotheses on $F$, we will assume at this
stage that the weak solution enjoys the regularity
 \begin{equation}
 \label{smoothnessass}
 \rho \in C([0,T];L^1 (\T^d))\ \hbox{and}\ \nabla \rho \in C([0,T]; L^1(\T^d))
 \tag{{\bf A}}
 \end{equation}
and  proceed to establish the relative energy identity.

\begin{theorem}\label{theo:relenweak}
Assume that hypotheses {\rm \textbf{(H1)},  \textbf{(H2)}} and \eqref{smoothnessass} hold. Then, 
\begin{align}
\label{eq:RelEnKorGenFinalweak}
& \int_{\T^d} \left. \Big ( \frac{1}{2} \rho \Big | \frac{m}{\rho} -  \frac{\bar m}{\bar\rho}\Big | ^2 +  F(\rho, \nabla\rho | \, \bar \rho, \nabla\bar \rho \right. ) \Big )dx \Big |_t \leq
\int_{\T^d} \left. \left ( \frac{1}{2} \rho \Big | \frac{m}{\rho} -  \frac{\bar m}{\bar\rho}\Big | ^2 +  F(\rho, \nabla\rho   | \, \bar \rho, \nabla\bar \rho \right. ) \right )dx \Big |_0
\nonumber\\
&\  - \iint_{[0,t)\times\T^d}\left [ \rho \left( \frac{m}{\rho}-   \frac{\bar m}{\bar\rho}\right )\otimes \left( \frac{m}{\rho}-   \frac{\bar m}{\bar\rho}\right ) :\nabla  \left(    \frac{\bar m}{\bar\rho}\right )
+ \dx \left (\frac{\bar m}{\bar\rho}\right ) s(\rho, \nabla\rho\left | \bar \rho, \nabla\bar \rho \right. )\right ] \; dxdt
\nonumber\\
&\ 
- \iint_{[0,t)\times\T^d} \left [ \nabla \left (\frac{\bar m}{\bar\rho} \right) : H(\rho, \nabla\rho  | \bar \rho, \nabla\bar \rho  ) 
+ \nabla \dx \left (\frac{\bar m}{\bar\rho}\right ) \cdot r(\rho, \nabla\rho\left | \bar \rho, \nabla\bar \rho \right. )\right] \; dxdt\, ,
\end{align}
where $s(\rho , q | \bar \rho , \bar q)$, $r(\rho , q |  \bar \rho , \bar q )$ and $H(\rho,q | \bar \rho , \bar q)$ are the relative functions, defined by 
\eqref{stressdeffunc}, \eqref{defrels}  from the constituents of the stress tensor \eqref{stresskort}.
\end{theorem}

A derivation of this formula via direct calculation valid for solutions of the Euler-Korteweg system that are {\it both smooth}  is provided in Appendix \ref{subsec:relenKort}.
Here, we relax the smoothness of one of the two solutions by using an argument that validates the formal variational  calculations of Section \ref{sec:ham}
for the Korteweg energy \eqref{kortfunc}.

\begin{proof} Let $(\rho, m)$ be a weak dissipative (or conservative) solution and $(\brho , \bar u)$ with $\bar m = \brho \bar u$ a strong conservative
 solution. We introduce in \eqref{eq:Kortcap2Integr} the choice of test function
\begin{equation}
\theta(\tau) := 
\begin{cases}
1, &\hbox{for}\ 0\leq \tau < t, \\
\frac{t-\tau}{\epsilon} + 1, &\hbox{for}\ t\leq \tau < t+\epsilon , \\
0, &\hbox{for}\ \tau \geq t+\epsilon ,
\end{cases}
\label{testthetaS}
\end{equation}
and let  $\epsilon \downarrow 0$; we then obtain 
\begin{equation}
\label{lem:ret1}
  \left. \int_{\mathbb{T}^d}\left ( \frac{1}{2} \frac{|m|^2}{\rho} + F(\rho,\nabla\rho) \right)dx  \right |^{t}_{\tau=0}
\leq  0\, .
\end{equation}
The same argument applied now to the strong conservative solution $(\brho, \bar m)$ gives
\begin{equation}\label{lem:ret}
 \left. \int_{\mathbb{T}^d}\left ( \frac{1}{2} \frac{|\bar m|^2}{\bar \rho} + F(\bar \rho,\nabla\bar\rho) \right)dx \right |^{t}_{\tau=0}
=0\,  .
\end{equation}

Next, consider the equation satisfied by the differences $(\rho - \bar\rho, m- \bar m)$:
\begin{align}
- \iint_{[0,+\infty)\times\T^d} \Big ( {\psi}_t (\rho - \bar \rho) +  \psi_{x_i} (m_i - \bar m_i)  \Big ) dx dt
- \int_{\mathbb{T}^d} \psi (\rho - \bar \rho) \Big |_{t=0} \, dx = 0\, ,
\label{weakmass}
\\
- \iint_{[0,+\infty)\times\T^d} \varphi_t \cdot (m - \bar m) +  \partial_{x_i}\varphi_j  \left (\frac{m_i m_j}{\rho} - \frac{\bar m_i \bar m_j}{\bar\rho} \right)
- \partial_{x_i} \varphi_j \Big(S_{ij} - \bar S_{ij} \Big) \,  dx dt
\nonumber
\\
- \int_{\mathbb{T}^d} \varphi(x,0) \cdot (m - \bar m) \Big |_{t=0}  dx 
= 0 \, ,
\label{weakmomentum}
\end{align}
which can be easily extended to hold for   Lipschitz test functions $\varphi$, $\psi$ 
compactly supported in $ [0,T)\times \T^d$.
In the above relations we introduce the test functions
$$
\psi = \theta(\tau)\left( \bar F_\rho  - \dx\big ( \bar F_q\big) -\frac{1}{2}\frac{|\bar m|^2}{\bar\rho^2}   \right ) , \quad
\varphi = \theta(\tau) \frac{\bar m}{\bar\rho} \,  ,
$$
with $\theta(\tau)$ as in \eqref{testthetaS}. Letting $\epsilon \to 0$ in  \eqref{weakmass}
we arrive at
\begin{align}
& \int_{\mathbb{T}^d}  \Big (  \bar F_\rho  (\rho - \brho) + \bar F_q\cdot\nabla (\rho - \bar \rho)   -\frac{1}{2}\frac{|\bar m|^2}{\bar\rho^2}  (\rho - \bar \rho) 
  \Big ) \Bigg |_{\tau=0}^t dx 
\nonumber\\
&\quad - \iint_{[0,t)\times\T^d} \left [ \partial_\tau\left( \bar F_\rho   -\frac{1}{2}\frac{|\bar m|^2}{\bar\rho^2}   \right ) (\rho - \bar \rho) 
+ \partial_\tau\big ( \bar F_q\big) \cdot  \nabla (\rho - \bar \rho) \right ]\, dx d\tau 
\nonumber\\
&\quad - \iint_{[0,t)\times\T^d} \nabla\left( \bar F_\rho  - \dx\big ( \bar F_q\big) -\frac{1}{2}\frac{|\bar m|^2}{\bar\rho^2}   \right ) \cdot (m- \bar m)   dx d\tau = 0\, .
\label{weakmass3}
\end{align}
The regularity $\rho (\cdot, t)  \in W^{1,1} (\T^d)$ for each $t \in [0,\infty)$ is sufficient to
give meaning to the calculation, taking advantage of the fact that $(\brho, \bar u)$ is a strong solution. 
Similarly, from \eqref{weakmomentum} we get
\begin{align}
& \int_{\mathbb{T}^d} \frac{\bar m}{\bar\rho}\cdot (m - \bar m)  \Big |_{\tau=0}^tdx
= \iint_{[0,t)\times\T^d} \partial_\tau\left( \frac{\bar m}{\bar\rho}   \right ) \cdot (m-\bar m) dx d\tau 
\nonumber\\
&\ \ +
\iint_{[0,t)\times\T^d} \partial_{x_i}\left( \frac{\bar m_j}{\bar\rho}   \right ) \left  ( \Big ( \frac{m_im_j}{\rho}- \frac{\bar m_i\bar m_j}{\bar\rho}  \Big )
  -  ( S_{ij} - \bar S_{ij} )  \right )   dx d\tau \, .
\label{weakmomentum2}
\end{align}

Adding \eqref{weakmass3} and \eqref{weakmomentum2} and using  
the equation 
$$
\del_t \bar u + ( \bar u \cdot \nabla ) \bar u  = -  \nabla \big ( \bar F_\rho - \dx \bar F_q     \big ) 
$$
and 
the chain rule for the Lipschitz solution $(\brho, \bar u)$, we obtain after some lengthy but straightforward calculations
\begin{align}
& \int_{\mathbb{T}^d} \Big [  \bar F_\rho (\rho - \bar \rho)  + \bar F_q\cdot\nabla (\rho - \bar \rho)   -\frac{1}{2}\frac{|\bar m|^2}{\bar\rho^2}  (\rho - \bar \rho)  +
\frac{\bar m}{\bar\rho}\cdot (m - \bar m) \Big ]   \Bigg |_{\tau=0}^t dx
\nonumber\\
&\ = - \iint_{[0,t)\times\T^d} \Big ( 
\bar F_{\rho\rho}  (\dx \bar m)   (\rho -\bar \rho) + \bar F_{\rho q}\cdot  (\nabla \dx\bar m) (\rho -\bar \rho)  
  \nonumber \\
&\quad 
 \qquad \qquad \qquad  + \bar F_{q \rho}\cdot\nabla(\rho - \bar\rho)\dx\bar m  + \bar F_{q_j q_i} \partial_{x_j}(\rho -\bar\rho)\partial_{x_i}(\dx\bar m)  \Big ) dx d\tau 
\nonumber\\
&\quad -
\iint_{[0,t)\times\T^d} \partial_{x_i}\left( \baru_j  \right )  ( S_{ij} - \bar S_{ij} )    dx d\tau 
\nonumber \\
&\quad + \iint_{[0,t)\times\T^d} \nabla \Big ( \bar F_\rho - \dx \bar F_q \Big ) \cdot (\rho u - \bar\rho\baru )  dx d\tau
\nonumber \\
&\quad +
\iint_{[0,t)\times\T^d}  \bigg [    \del_\tau \Big (-\tfrac{1}{2}|\baru|^2  \Big )  (\rho - \bar \rho) 
- \nabla \left (  \tfrac{1}{2} |\baru|^2 \right) \cdot (\rho u  - \bar\rho\baru)  + \del_\tau \left ( \baru \right ) \cdot (\rho u  - \bar \rho\baru)  
 \nonumber\\
&\quad 
 \qquad \qquad \qquad \qquad +  \partial_{x_i}\left (\baru_j \right)\left (\rho u_iu_j  - \bar\rho\baru_i\baru_j \right)
\vphantom{\tfrac{1}{2}} \bigg ]dx d\tau 
 \nonumber\\
 &\ = - \iint_{[0,t)\times\T^d} \bigg [ 
\bar F_{\rho\rho}  (\dx \bar m)   (\rho -\bar \rho) + \bar F_{\rho q}\cdot  (\nabla \dx\bar m) (\rho -\bar \rho)  
  \nonumber \\
&\quad
 \qquad \qquad \qquad  + \bar F_{q \rho}\cdot\nabla(\rho - \bar\rho)\dx\bar m  + \bar F_{q_j q_i} \partial_{x_j}(\rho -\bar\rho)\partial_{x_i}(\dx\bar m)  \bigg ] dx d\tau 
\nonumber\\
&\quad + \iint_{[0,t)\times\T^d}  \nabla \Big ( \bar F_\rho - \dx \bar F_q \Big ) \cdot   (\rho - \brho)  \bar u \,  dx d\tau 
\nonumber\\
 &\quad- \iint_{[0,t)\times\T^d} \Big [ \nabla \bar u : (S - \bar S) - \nabla \bar u : \rho (u - \bar u) \otimes (u - \bar u) \Big]
 dx d\tau  \nonumber
 \\
 &\ =: J_1 + J_2 + J_3 \, .
 \label{eq:linearcorrweak}
\end{align}

Consider next the weak form of the identity \eqref{formulakorteweg} for the strong solution $\brho$ and for $\varphi$ a vector-valued test function
$$
- \int  \nabla ( \bar F_\rho - \div \bar F_q ) \cdot (\brho \varphi)  dx =
\int  \bar F_\rho \div (\brho \varphi) + \bar F_q \cdot \nabla \div (\brho \varphi) dx = - \int S_{i j} (\brho) \frac{\del \varphi_i}{\del x_j} \, dx \, .
$$
We take the variational derivative of this formula along the direction of a smooth test function $\psi$.
Using \eqref{secondvar}, \eqref{stressvarkort} and recalling  \eqref{stressdeffunc}, we obtain
$$
\begin{aligned}
&\int \bar F_{\rho \rho} \psi \div (\brho \varphi) + \bar F_{q \rho} \psi \cdot \nabla \div (\brho \varphi) + \div (\brho \varphi) \bar F_{\rho q} \cdot \nabla \psi 
+ \nabla \div (\brho \varphi) \cdot \bar F_{qq} \nabla \psi \, dx
\\
&-  \int  \nabla \Big (  \bar F_\rho - \div \bar F_q \Big )   \cdot (\psi \varphi) \, dx
\\
&= - \int \Big [
- ( \bar s_\rho \psi +  \bar s_q \cdot \nabla \psi )  \delta_{i j} 
+ \frac{\del}{\del x_k} \Big (    { \frac{\del  \bar r_k}{\del \rho} }  \psi +   {\frac{\del \bar r_k}{\del q_l}} \frac{\del \psi}{\del x_l} \Big ) \delta_{i j} 
- \Big (    {\frac{\del \bar H_{i j} }{\del \rho}} \psi +   { \frac{\del \bar H_{i j} }{\del q_l }}  \frac{\del \psi}{\del x_l} \Big )  \Big ] \frac{\del \varphi_i}{\del x_j} \, dx\, .
\end{aligned}
$$
Now, set $\varphi = \bar u$,  $\psi = \rho - \brho$ and integrate over $(0,t)$ to obtain 
\begin{equation}
\label{secondvarek}
\begin{aligned}
J_1 + J_2 = 
\iint_{[0,t)\times\T^d} \Big (
- [ \bar s_\rho (\rho - \brho) &+  \bar s_q \cdot \nabla (\rho - \brho) ]  \delta_{i j} 
+ \frac{\del}{\del x_k} \Big (    { \frac{\del  \bar r_k}{\del \rho} }  (\rho - \brho)  +   {\frac{\del \bar r_k}{\del q_l}} \frac{\del (\rho - \brho)}{\del x_l} \Big ) \delta_{i j} 
\\
&\qquad - \Big (    {\frac{\del \bar H_{i j} }{\del \rho}} (\rho - \brho) +   { \frac{\del \bar H_{i j} }{\del q_l }}  \frac{\del (\rho - \brho)}{\del x_l} \Big )  \Big ) \frac{\del \bar u_i}{\del x_j} 
\, dx d\tau \, .
\end{aligned}
\end{equation}
Combining \eqref{relpotekort}, \eqref{relkinenergy}, \eqref{lem:ret1}, \eqref{lem:ret}, \eqref{eq:linearcorrweak} and \eqref{secondvarek} and \eqref{relstresskort} leads to
 \eqref{eq:RelEnKorGenFinalweak}.
\end{proof}

\subsection{Weak-strong stability in energy norms for constant capillarity}\label{subsubsec:energysol}
Using  \eqref{eq:RelEnKorGenFinalweak}, we may obtain stability estimates in various situations. We start with a potential 
energy
\begin{equation}\label{eq:constantcap}
F(\rho, \nabla\rho ) = h(\rho ) + \frac{1}{2}C_\kappa |\nabla \rho|^2,
\end{equation}
with constant capillarity $\kappa (\rho) = C_\kappa>0$.  
The internal energy and pressure are connected through the 
usual thermodynamic relation and are restricted to monotone pressures
\begin{equation}
\label{hypthermo}
 \quad \rho h'(\rho) = p(\rho) + h(\rho) \, ,
 \qquad p'(\rho) > 0.
\tag{H$_m$}
\end{equation}
Moreover, we impose the technical conditions that, for $\gamma > 1$  and some constants $k > 0$, $A > 0$,
\begin{align}
\label{eq:growthh}
h(\rho) &= \frac{k}{\gamma -1} \rho^\gamma + o(\rho^\gamma) \, , \quad \hbox{as}\ \rho\to +\infty
\tag{A$_1$}
\\
\label{hyppress}
| p'' (\rho ) | &\le A \frac{p'(\rho)}{\rho} =  A h'' (\rho)   \qquad \forall \rho  > 0 \, .
\tag{A$_2$}
\end{align}
These conditions are satisfied by the usual $\gamma$--law:  $p(\rho) = k\rho^{\gamma}$ with $\gamma > 1$.
We start with a preliminary result.

\begin{lemma}\label{lem:generalconvEuler3d}
(a)  Let $h\in C^0[0,+\infty)\cap C^2(0,+\infty)$ satisfy \eqref{eq:growthh} for  $\gamma >1$. 
If $\bar\rho\in M_{\brho} :=[\delta, \bar R]$ with $\delta>0$ and $\bar R<+\infty$, then there exist positive constants
$R_0$ (depending on $K$) and $C_1$, $C_2$ (depending on $K$ and $R_0$)
such that
\begin{equation}
\label{eq:hnorm}
h(\rho \left | \bar \rho \right.) \geq 
\begin{cases}
C_1 |\rho-\bar\rho|^2, &\hbox{for}\ 0\leq \rho\leq R_0,\ \bar\rho\in K, \\
C_2 |\rho-\bar\rho|^\gamma, &\hbox{for}\ \rho> R_0,\ \bar\rho\in K.
\end{cases}
\end{equation}

\noindent
(b) If $p(\rho)$ and $h(\rho)$ satisfy \eqref{hypthermo} and \eqref{hyppress} then
\begin{equation}
\label{eq:hyppress}
| p(\rho | \bar \rho) | \le A \, h (\rho | \bar \rho) \quad \forall \rho, \bar \rho > 0 \, .
\end{equation}
\end{lemma}

\begin{proof}
Part (a) is proved in \cite[Lemma 2.4]{LT13}.  To show (b), one checks the identity
$$
\begin{aligned}
p( \rho | \bar \rho)  &= p(\rho) - p ( \bar \rho) - p'(\bar \rho) (\rho - \bar \rho)
\\
&= (\rho - \bar \rho)^2 \int_0^1 \int_0^\tau p'' (s\rho + (1-s) \bar \rho) ds d\tau \, ,
\end{aligned}
$$ 
and a similar identity holds for $h(\rho | \bar \rho)$. Recall  now that $h'' = \frac{p'}{\rho}$. 
Then hypothesis \eqref{hyppress} implies 
$$
\begin{aligned}
| p (\rho | \bar \rho ) | &\le (\rho - \bar \rho)^2 \int_0^1 \int_0^\tau |p'' (s\rho + (1-s) \bar \rho) | ds d\tau
\\
&\le A  (\rho - \bar \rho)^2 \int_0^1 \int_0^\tau h'' (s\rho + (1-s) \bar \rho) ds d\tau 
\\
&= A \,  h(\rho | \bar \rho )
\end{aligned}
$$
and \eqref{eq:hyppress} follows.
\end{proof}

Consider next the relative entropy inequality \eqref{eq:RelEnKorGenFinalweak} and apply it to the framework of \eqref{eq:constantcap}.
The bound \eqref{hypCauchyK2} implies the regularity $\rho\in C([0,T];L^\gamma(\T^d))$ and $\nabla \rho \in C([0,T]; L^2(\T^d))$ for the
solution, while  the relative internal energy reads
\begin{equation}
\label{relpotecc}
F(\rho, \nabla\rho\left | \bar \rho, \nabla\bar \rho \right. ) = h(\rho \left | \bar \rho  \right. ) + \frac{1}{2}C_\kappa|\nabla \rho - \nabla\bar\rho|^2 \, .
\end{equation}
In the following stability theorem we use the relative energy
 \begin{equation}
\varphi(t) = \int_{\T^d} \Big ( \frac{1}{2} \rho \left | \frac{m}{\rho} -  \frac{\bar m}{\bar\rho}\right | ^2 +  F(\rho, \nabla\rho\left | \bar \rho, \nabla\bar \rho \right. ) \Big ) (x,t) dx 
\label{phicc}
\end{equation}
as a yardstick to estimate the distance between the solutions.

\begin{theorem}\label{th:finalKortconst}
Let  $F(\rho, q)$ be given by \eqref{eq:constantcap} with $h(\rho)$ satisfying \eqref{hypthermo}, \eqref{eq:growthh} and \eqref{hyppress}.
Let $T>0$ be fixed and let $(\rho , u)$ be a weak solution satisfying  {\rm \textbf{(H1)}} and 
$(\bar \rho, \bar u)$ a strong solution satisfying  {\rm \textbf{(H2)}}.
Assume that the strong solution $\bar\rho$ is bounded and away from vacuum, that is $\bar\rho \in M_{\bar\rho}$. Then, 
for any data if $\gamma \ge 2$, and for initial data satisfying the restriction
$$
\int_{\T^d} (\rho _0 - \bar \rho_0 ) dx = 0 \, ,
$$
when $1 < \gamma < 2$, we have the stability estimate
\begin{equation}
\varphi(t) \leq C \varphi(0), \quad t\in [0,T] \, , 
\label{eq:finalKortconst}
\end{equation}
where $C$ is a positive constant depending on $T$, $\bar u$ and its derivatives.
In particular, if $\varphi(0)=0$ , then
\begin{equation}
\sup_{t\in[0,T]}\varphi(t) = 0\, ,
\label{eq:weakstronguniqu}
\end{equation}
implying weak--strong uniqueness for the model  and the framework of solutions under consideration.
\end{theorem}

A similar result to Theorem  \ref{th:finalKortconst} is obtained in \cite{DFM15} for the constant capillarity case.

\begin{proof} For the potential energy \eqref{eq:constantcap}, the  functions \eqref{stressdeffunc} take the form
\begin{align*}
 H(\rho , q) &= C_\kappa q \otimes q\, ;\\ 
 s(\rho , q) & = h(\rho)  + \tfrac{1}{2} C_\kappa |q|^2;\\
 r(\rho , q) &= C_\kappa \rho q\, .
\end{align*}
The relative potential energy \eqref{relpotecc}  controls the terms 
$s(\rho, \nabla\rho\left | \bar \rho, \nabla\bar \rho \right. )$, $H(\rho, \nabla\rho\left | \bar \rho, \nabla\bar \rho \right. )$, since
due to \eqref{eq:hyppress} and \eqref{relpotecc}
\begin{align*}
\big | H(\rho, \nabla\rho\left | \bar \rho, \nabla\bar \rho \right. ) \big | &=  
\big | C_\kappa (\nabla \rho - \nabla\bar\rho) \otimes (\nabla \rho - \nabla\bar\rho) \big | 
\leq C_\kappa  |\nabla \rho - \nabla\bar\rho|^2  \, ,
\\[6pt]
| s(\rho, \nabla\rho\left | \bar \rho, \nabla\bar \rho \right. ) | & =   \big | p(\rho \left | \bar \rho  \right. )  + \tfrac{1}{2} C_\kappa |\nabla \rho - \nabla\bar\rho|^2  \big |
\leq C F(\rho, \nabla\rho\left | \bar \rho, \nabla\bar \rho \right. )\, ,
\end{align*}
and 
\begin{equation*}
r(\rho, \nabla\rho\left | \bar \rho, \nabla\bar \rho \right. ) = C_\kappa(\rho-\bar\rho)(\nabla \rho - \nabla\bar\rho)\, .
\end{equation*}

Let $(\rho, u)$ be a weak dissipative (or conservative) solution of \eqref{eq:Kortcap2} and $(\bar \rho, \bar u)$ be a strong
solution. Applying Theorem \ref{theo:relenweak} we obtain
\begin{equation}
\label{interest}
\varphi (t) \le \varphi (0) + C \int_0^t  \varphi (\tau) d\tau  +  C  \int_0^t \int |\rho-\bar\rho | | \nabla \rho - \nabla\bar\rho | \, dx d\tau\,.
\end{equation}

It remains to estimate the last term in \eqref{interest}. Consider first the range of exponents $\gamma \ge 2$.
For exponents $\gamma \ge 2$, by enlarging if necessary $R_0$ so that $|\rho - \bar\rho| \geq 1$ for $\rho > R_0$ and $\bar\rho\in K =[\delta, \bar R]$, we   obtain
\begin{equation}\label{eq:hnorm2}
h(\rho \left | \bar \rho \right.) \geq c_0 |\rho-\bar\rho|^2, \quad \hbox{for}\ \gamma \geq2,\ 
\rho\geq0,\ \bar\rho\in K\, ,
\end{equation}
where $c_0>0$ depends solely on $K$.  By Lemma \ref{lem:generalconvEuler3d}, 
it follows
\begin{equation*}
|\rho-\bar\rho | | \nabla \rho - \nabla\bar\rho |  \leq C  F(\rho, \nabla\rho\left | \bar \rho, \nabla\bar \rho \right. )\, .
\end{equation*}
If $1 < \gamma < 2$ and the initial data satisfy $ \int_{\T^d}  (\rho_0  - \bar \rho_0)  dx = 0$, then   the mass conservation equation
$$
\int_{\T^d} (\rho - \bar \rho) \, dx = 0
$$
and the Poincar\'{e} inequality allow to conclude
$$
\int_{\T^d} |\rho-\bar\rho | | \nabla \rho - \nabla\bar\rho | \, dx  \le C \int_{\T^d } | \nabla \rho - \nabla \bar \rho |^2 \, dx\, .
$$
In either case we conclude by
\begin{equation}
\varphi (t) \le \varphi (0) + C \int_0^t  \varphi (\tau) d\tau  
\end{equation}
and the result follows from Gronwall's inequality.
\end{proof}

\begin{remark}\label{rem:gamma=2}
The absence of vacuum for $\bar\rho$ required in the theorem above is needed solely to apply Lemma \ref{lem:generalconvEuler3d}.
In the special case $p(\rho) = \rho^\gamma$, $\gamma=2$, for which $h(\rho \left | \bar \rho \right.)=|\rho-\bar\rho|^2$, this extra condition is not needed, 
and the stability estimate \eqref{eq:finalKortconst} is obtained even in the presence of vacuum 
for both solutions $\rho$ and $\bar\rho$. However, the regularity assumption  \textbf{(H2)} is still needed; 
this might be inconsistent  for certain models with the presence of vacuum.
\end{remark}

\subsection{Stability estimates for bounded weak solutions with  general capillarity}\label{subsubsec:linfsol}
Next, we study energies  with non-constant, smooth capillarity $\kappa(\rho)$,
\begin{equation}
\label{hyposp}
F(\rho, q ) =  h(\rho ) +  \frac{1}{2}  \kappa(\rho)  |q|^2 \qquad \mbox{ with  $\kappa (\rho) > 0$}\, .
\end{equation}
The assumption $\kappa (\rho) > 0$ guarantees coercivity of the energy. It is also related to the convexity of the energy.

\begin{lemma}\label{lem:Funifconvex}
Let $F(\rho, q)$ be defined by \eqref{hyposp}.
\begin{itemize}
\item[(i)]  If \begin{equation}
\label{hyp4c}
h''(\rho ) > 0, \quad \kappa (\rho) > 0\, , \quad \kappa(\rho)\kappa''(\rho) - 2(\kappa'(\rho))^2\geq 0\, ,
\tag{ H$_{4c}$}
\end{equation}
then $F$  is strictly convex for any $(\rho, q)\in \mathbb{R}\times\mathbb{R}^d$;

\item[(ii)] If for some constants $\alpha_1 > 0$ and $\alpha_2 > 0$,
\begin{equation}
\label{hyp4u}
h''(\rho)\geq \alpha_1 \, , \quad \kappa (\rho)  - \frac{ 2(\kappa'(\rho))^2  }{  \kappa''(\rho) } \geq \alpha_2 \, , \quad \kappa''(\rho)  > 0 \, ,
\tag{ H$_{4uc}$}
\end{equation} 
then $F$ is uniformly convex. 
\end{itemize} 
\end{lemma}
\begin{proof}
A direct calculation shows that the Hessian matrix of $F(\rho, q)$ is given by
\begin{equation*}
\nabla_{(\rho,q)}^2 F (\rho,q) = 
\begin{pmatrix}
h''(\rho) + \frac12 \kappa''(\rho)|q|^2 & \kappa'(\rho)q \\
\kappa'(\rho)q^T & \kappa(\rho)\mathbb{I}
\end{pmatrix}\, .
\end{equation*}
Let $A = (a , b )^T$ be a vector with $a \in \R$, $b \in \R^d$ and consider the quadratic form
$$
\begin{aligned}
A \cdot ( \nabla_{(\rho,q)}^2 F (\rho, q)  )  A &= 
( a, b^T  ) \cdot 
\begin{pmatrix}
h''(\rho) + \frac12 \kappa''(\rho)|q|^2 & \kappa'(\rho)q \\
\kappa'(\rho)q^T & \kappa(\rho)\mathbb{I}
\end{pmatrix}
\begin{pmatrix}
a \\ b
\end{pmatrix}
\\
&= \big ( h'' + \tfrac{1}{2} \kappa'' |q|^2 ) a^2  + 2 \kappa' a q \cdot b + \kappa |b|^2
\\
&= h'' a^2 + \frac{1}{\kappa''} \big [ \kappa'' \kappa - 2 (\kappa')^2 \big ] |b|^2 + \frac{2}{\kappa''} \big | \frac{\kappa''}{2} a q + \kappa' b \big |^2\, .
\end{aligned}
$$
From here (i) and (ii) follow.
\end{proof}

The calculation
$$
\begin{aligned}
&F(\rho, q | \bar \rho ,  \bar q ) =  F(\rho, q) - F(\bar \rho , \bar q) - F_\rho (\bar \rho , \bar q) (\rho - \bar \rho) - F_q (\bar \rho , \bar q) \cdot (q - \bar q)
\\
&\quad =  ( \rho - \bar \rho , (q - \bar q)^T  )  \cdot   \int_0^1 \int_0^\tau   \nabla_{(\rho,q)}^2 F ( s\rho + (1-s) \bar \rho , sq + (1-s) \bar q   )   ds d\tau
\begin{pmatrix}
\rho - \bar \rho  \\ q - \bar q
\end{pmatrix}
\end{aligned}
$$
indicates the relevance of uniform convexity in bounding the relative potential energy.
We develop a stability estimate in terms of the following distance:
 \begin{equation}
\Phi(t) = \int_{\T^d}\left. \left ( \frac{1}{2} \rho \left | \frac{m}{\rho} -  \frac{\bar m}{\bar\rho}\right | ^2 +  |\rho - \bar\rho|^2 + |\nabla\rho - \nabla\bar\rho|^2 \right )dx \right |_t .
\label{eq:normfin}
\end{equation}

\begin{theorem}\label{th:finalKortconst2}
Consider two solutions $(\rho, m)$ and $(\brho, \bar m)$ satisfying the hypotheses {\rm \textbf{(H1)}} and  {\rm \textbf{(H2)}} respectively for some $T > 0$ 
and suppose that \eqref{hyposp} and  \eqref{hyp4u} hold.
Furthermore, assume that $\rho$ is bounded in $L^\infty$ and away from vacuum, that is  $\rho_{min} \le \rho(x,t)\leq \rho_{max}$ for a.e.\ $(x,t)$;
also  that  $\bar\rho$ is bounded away from vacuum, that is $\bar\rho \in M_{\bar\rho}$.
Then, the stability estimate
\begin{equation}
\Phi(t) \leq C \Phi(0) , \quad t\in [0,T] \, ,
\label{eq:finalKortconst2}
\end{equation}
holds,  where $C$ is a positive constant depending only on $T$, $\rho_{min}$, $\rho_{max}$,  $\bar u$ and its derivatives.
In particular, if $\Phi(0)=0$ , then
\begin{equation}
\sup_{t\in[0,T]}\Phi(t) = 0\, ,
\label{eq:weakstronguniqu2}
\end{equation}
which implies weak--strong uniqueness for the model and the framework of solutions under consideration.
\end{theorem}

\begin{proof}
First, note that \eqref{hyposp}-\eqref{hyp4u} guarantee sufficient regularity for the weak solution $(\rho,m)$ 
so that \eqref{smoothnessass} holds and we may use \eqref{eq:RelEnKorGenFinalweak}  in  Theorem \ref{theo:relenweak}. 
In view of Lemma \ref{lem:Funifconvex},  there exists $\alpha >0$ such that 
\begin{equation*}
 \int_{\T^d}\left. \left ( \frac{1}{2} \rho \left | \frac{m}{\rho} -  \frac{\bar m}{\bar\rho}\right | ^2 +  F(\rho, \nabla\rho\left | \bar \rho, \nabla\bar \rho \right. ) \right )dx \right |_t \geq \alpha \Phi(t)\, ,
\end{equation*}
where $\Phi(t)$ is given in \eqref{eq:normfin}.

We proceed to bound the right hand side of  the relative energy inequality  \eqref{eq:RelEnKorGenFinalweak} in terms of $\Phi(t)$.
Using \eqref{stressdeffunc} and \eqref{hyposp} we compute
\begin{align*}
 H(\rho , q) &= \kappa(\rho) q \otimes q\, ,
 \\
 s(\rho , q) & = p (\rho) +   A(\rho) |q|^2 \, ,  \quad \mbox{where $A(\rho) = \tfrac{1}{2} (\rho \kappa'(\rho) + \kappa (\rho) )$}  \, ,
 \\
 r(\rho , q) &=  B(\rho)   q   \, ,    \qquad \qquad \mbox{where $B(\rho) = \rho \kappa(\rho) $}     \, ,
\end{align*}
and the associated relative functions
\begin{align*}
s ( \rho, q | \bar \rho , \bar q) &=  p (\rho | \bar \rho) + A( \rho | \bar \rho ) | \bar q|^2 + A(\rho) | q - \bar q|^2 + ( A(\rho) - A (\bar \rho) ) 2 \bar q \cdot (q - \bar q)\, ,
\\
H ( \rho, q | \bar \rho , \bar q)  &=  \kappa (\rho) (q - \bar q) \otimes (q - \bar q) + \kappa ( \rho | \bar \rho ) \bar q \otimes \bar q + 
 ( \kappa (\rho) - \kappa  (\bar \rho) ) \big [ \bar q \otimes (q - \bar q) + (q - \bar q) \otimes \bar q \big ]\, ,
 \\
 r ( \rho, q | \bar \rho , \bar q) &= B ( \rho | \bar \rho ) \bar q  +  ( B (\rho) - B (\bar \rho) ) (q - \bar q)\, .
\end{align*}

Since $\rho$  satisfies $0 < \rho_{min} \leq \rho(x,t)\leq \rho_{max}$
for a.e.\ $(x,t)$,  we have that 
$$
\begin{aligned}
\|\kappa(\rho)\|_{\infty}\leq \kappa_\infty=\sup_{\xi\in[\rho_{min} ,\rho_{max}]}|\kappa(\xi)|\, ,
 \\
\|  A(\rho) \|_{\infty}\leq A_\infty = \sup_{\xi\in[\rho_{min} ,\rho_{max}]}|A (\xi)|\, .
\end{aligned}
$$
Thus we can estimate the quadratic terms on the right hand side of \eqref{eq:RelEnKorGenFinalweak} via 
\begin{equation*}
\big | H(\rho, \nabla\rho\left | \bar \rho, \nabla\bar \rho \right. ) \big |, \
|s(\rho, \nabla\rho\left | \bar \rho, \nabla\bar \rho \right. )|, \
\big | r(\rho, \nabla\rho\left | \bar \rho, \nabla\bar \rho \right. ) \big |
 \leq C   |\rho - \bar\rho|^2 + C |\nabla\rho - \nabla\bar\rho|^2,\\ 
\end{equation*}
where the constant $C$ depends solely on the smooth functions $h$ and $\kappa$, the uniform bounds of $\rho$  (that is $\rho_{min}$, $ \rho_{max}$), and  $\bar u$ and its derivatives. Finally, as in the previous framework, the  right hand side of \eqref{eq:RelEnKorGenFinalweak} is bounded in terms of $\Phi$ and thus the Gronwall Lemma gives the desired result.
\end{proof}

\subsection{Stability for the quantum hydrodynamics system}
The system of quantum hydrodynamics \eqref{qhd} has a special structure that allows to obtain stability estimates for a weak solution
that is not necessarily bounded. Consider the energy
\begin{equation}  
F(\rho , q) = h(\rho)  + \frac{1}{2 \rho} |q|^2  ,
\label{qhdintpot}
\end{equation}
associated to the quantum hydrodynamics system \eqref{qhd} (where we took $\eps = 2$  in \eqref{qhdenergy} for concreteness).
We assume that $h(\rho)$ and $p(\rho)$ are connected through \eqref{hypthermo} and satisfy \eqref{eq:growthh} and \eqref{hyppress}.
The relative  potential energy reads
$$
F(\rho , q | \bar \rho , \bar q ) =  h (\rho | \bar \rho) +  \frac{1}{2} \rho \Big |  \frac{q}{\rho} - \frac{ \bar q}{ \bar \rho} \Big | ^2 .
$$

For the energy \eqref{qhdintpot}, the functions determining the stress $S$ defined  by \eqref{stressdeffunc} , read
$$
\begin{aligned}
s(\rho, q) &= p(\rho) + \tfrac{1}{2} ( \rho \kappa' (\rho) + \kappa (\rho) ) |q|^2 = p(\rho)  \, ,
\\
r(\rho, q) &= q\, ,
\\
H(\rho , q) &= \frac{1}{\rho} q \otimes q\, ,
\end{aligned}
$$
and thus
\begin{equation}
\label{relfcalc}
\begin{aligned}
s ( \rho, q | \bar \rho , \bar q) &=  p  (\rho | \bar \rho) \, ,
\\
H ( \rho, q | \bar \rho , \bar q)  &=  \rho  \Big ( \frac{q}{\rho}  - \frac{\bar q}{\bar \rho}  \Big ) \otimes \Big ( \frac{q}{\rho}  - \frac{\bar q}{\bar \rho}  \Big ) \, ,
 \\
 r ( \rho, q | \bar \rho , \bar q) &= 0\, .
 \end{aligned}
\end{equation}

Consider now a weak solution $(\rho, u)$ of the quantum hydrodynamics system \eqref{qhd} satisfying {\rm \textbf{(H1)}}.
By \eqref{hypCauchyK1}, \eqref{hypCauchyK2} and \eqref{qhdintpot}
$$
\int  | \nabla \rho | dx \le \int \rho dx + \int \frac{1}{\rho} | \nabla \rho |^2 dx < \infty
$$
so that \eqref{smoothnessass} is satisfied. 
 Due to \eqref{eq:hyppress}, $|p(\rho | \bar \rho)| \le C h(\rho | \bar \rho)$. 
We may thus use the total relative energy
\begin{equation}
\label{eq:normrelqhd}
\Psi (t) = \int_{\T^d}  \Big ( \frac{1}{2} \rho \left | \frac{m}{\rho} -  \frac{\bar m}{\bar\rho}\right | ^2 +  h(\rho | \bar \rho)  
+ \frac{1}{2} \rho  \Big |  \frac{\nabla \rho }{\rho} - \frac{ \nabla \bar \rho }{ \bar \rho} \Big | ^2 \Big ) (x, t)  dx  
\end{equation}
as a yardstick to estimate the distance of two solutions.
Using  \eqref{qhdintpot}, \eqref{relfcalc} and \eqref{eq:hyppress}, the relative energy inequality \eqref{eq:RelEnKorGenFinalweak} 
implies
$$
\Psi (t) \le \Psi (0) + C \int_0^t \Psi (\tau) \, d\tau
$$
with $\Psi(t)$ as in \eqref{eq:normrelqhd}. We conclude:

\begin{theorem}\label{th:finalstabqhd}
Consider the energy \eqref{qhdintpot} with $h(\rho)$ satisfying \eqref{hypthermo}, \eqref{eq:growthh} and \eqref{hyppress}.
Let  $(\rho, m)$ and $(\brho, \bar m)$ be solutions of \eqref{qhd} satisfying the hypotheses {\rm \textbf{(H1)}} and  {\rm \textbf{(H2)}} 
respectively for some $T > 0$.
Then, the stability estimate holds
\begin{equation}
\Psi(t) \leq C \Psi(0) , \quad t\in [0,T] \, , 
\label{eq:finalqhd}
\end{equation}
where $C$ is a positive constant depending only on $T$, $\baru$ and its derivatives up to second order.
In particular, if $\Psi(0)=0$, then
\begin{equation}
\sup_{t\in[0,T]}\Psi(t) = 0\, ,
\label{eq:weakstrongqhd}
\end{equation}
which implies weak--strong uniqueness for the model and the framework of solutions under consideration.
\end{theorem}

\section{Stability estimates for non-convex energies}\label{sec:cd}
In this section we study the model \eqref{eq:Kortcap2} on $ (0,T) \times \T^3$
in case the local part $h(\rho)$ of the internal energy is non-convex.
Subsequently, we assume $h \in C^3((0,\infty),[0,\infty)),$
but  no convexity of $h.$
We will see that the higher order terms compensate for the non-convex $h,$
in the sense that we are still able to use (a modified version of) the relative energy to obtain 
continuous dependence on initial data.
To simplify the analysis we restrict ourselves to the case that $\kappa(\rho)=C_\kappa >0$.
We are convinced that analogous results also hold for $\rho$ dependent capillarity in 
case $\kappa$ is bounded from above and below.

A different approach to stability (and also existence) results for the Euler-Korteweg model can be found in \cite{BDD07} where well-posedness was shown in all Sobolev spaces of supercritical index.
The results in \cite{BDD07} hold for quite general capillarities and in their analysis $\rho\kappa(\rho)=\text{const}$ is a particular, simple case.
The strategy employed there is similar to our approach in that it uses the sum of kinetic and capillary energy as the key part of the energy functional while pressure terms are treated as source terms.
The analysis in \cite{BDD07}
is based on  rewriting the problem as a complex-valued second order system of non-dissipative conservation laws for $u$ and $w := (\kappa(\rho)/\rho)^{\frac{1}{2}} \nabla \rho$.
In contrast, our results are based on the relative entropy framework and a subsequent step in which pressure terms are removed. 
We hope that, thus, our framework can be extended to cover the comparison of weak and strong solutions in future works.

\subsection{Assumptions}
We are not (yet) able to carry out the subsequent analysis for weak solutions and, thus, we 
 will consider strong solutions of \eqref{eq:Kortcap2} with $\kappa(\rho)=C_\kappa$ in the spaces
\begin{equation}\label{ass11}
 \begin{split}
  \rho &\in C^0([0,T],C^3(\T^3,\mathbb{R}_+)) \cap C^{\frac{3}{2}}((0,T),C^1(\T^3,\mathbb{R}_+))\, ,\\
   u &\in C^0([0,T], C^1(\T^3,\mathbb{R}^3)) \cap C^1((0,T),C^0(\T^3,\mathbb{R}^3))\, .
 \end{split}
\end{equation}
We will compare solutions $(\rho, u)$ and $(\bar{\rho},\bar{u}) $ 
corresponding to initial data 
$(\rho_0, {u}_0)$ and $(\bar{\rho}_0,\bar u_0) ,$ respectively, by relative energy.

To this end, we need to assume some uniform bounds. In particular, vacuum has to be avoided uniformly.
\begin{assumption}[Uniform bounds]
 We assume
that there are constants $C_\rho,c_\rho,c_m >0$ such that
\begin{equation}\label{ass12}
 \begin{split}
  c_\rho &\leq \rho(t, x), \bar \rho(t, x) \leq C_\rho  \quad \forall \, t \in [0,T],\,  x \in \T^3,\\
 c_m &\geq \max \big\{ \Norm{  m}_{L^\infty((0,T)\times \T^3)} , \Norm{\bar{ m}}_{L^\infty((0,T)\times \T^3)}   \big\}\, ,
 \end{split}
\end{equation}
and define the following constants:
$ p_M := \Norm{p(\rho)}_{C^2([c_\rho,C_\rho])},\ h_M := \Norm{h(\rho)}_{C^3([c_\rho,C_\rho])}.$
\end{assumption}


\subsection{Continuous dependence on initial data}
Let us state the particular form of the relative energy and relative energy flux between solutions $(\rho, m)$ and $(\bar{\rho},\bar{ m}) $ of \eqref{eq:Kortcap2}
with the choice $F(\rho,q)=h(\rho)+ C_\kappa |q|^2.$
For notational convenience, we omit capillarity related contributions in the relative energy flux:
\begin{equation}\label{def:re2NSK}
\begin{split}
  \eta \Big ( 
           \rho, { m}    
              |
            \bar \rho, \bar m
           \Big)  &=  h(\rho) -h(\bar \rho)-h'(\bar \rho) ( \rho - \bar \rho)
              + \frac{C_\kappa}{2}\norm{\nabla (\rho - \bar \rho)}^2 + \frac{\rho}{2} | u - \bar u|^2, \\
 q \Big (
           \rho,m   
             |
            \bar \rho, \bar m
           \Big) &=  m h'( \rho) + \frac{\norm{ m}^2}{2 \rho^2}  m  
 - m h'(\bar \rho) +\frac{ \norm{ \bar m}^2}{2\bar \rho^2}\bar m
 - \frac{\bar m \cdot m}{\bar \rho \rho}  m + \frac{p(\rho)}{\bar\rho}\bar m - 
\frac{p(\bar \rho)}{\bar \rho} \bar m\, .
\end{split}
\end{equation}

The main challenge we are faced with in this chapter is the non-convexity of $h$ which ensues that 
the relative energy is not suitable for measuring the distance between solutions.
However, if $h$ vanished, the relative energy would be suitable for controlling the difference between solutions.
To be more precise, the relative kinetic energy allows us to control $u - \bar u$ and $\int |\nabla(\rho- \bar \rho)|^2 $ 
is equivalent to the squared $H^1$ distance between $\rho$ and $\bar \rho$ provided $\rho, \bar \rho$ have the same mean value, due to Poincar{\'e}'s inequality.
We restrict ourselves to the case
\begin{equation}\label{mean} \int_{\T^3}( \rho_0 - \bar \rho_0) \operatorname{d}  x =0\end{equation}
such that Poincar{\'e}'s inequality is applicable.
Therefore, we like to introduce the following quantity, which is part of the relative energy:
\begin{equation}\label{def:rreNSK}
  \eta_{\mathrm{R}}\Big ( 
             \rho,  m    
            |
            \bar \rho, \bar m
           \Big) :=  \frac{C_\kappa}{2}\norm{\nabla (\rho- \bar \rho)}^2
+ k (  \rho, m  |   \bar  \rho,  \bar m  ) \, ,
\end{equation}
where $k(\rho, m |  \brho, \barm)$  is the density of the relative kinetic energy $K(\rho, m | \brho, \barm)$
introduced in \eqref{relkinenergy}.
We call $  \eta_{\mathrm{R}}$ the {\it reduced relative energy}.

Due to the properties of the relative kinetic energy we obtain two estimates for the reduced relative energy
\begin{equation}\label{eq:er1}
\begin{split}
  \int_{\T^3} \eta_{\mathrm{R}}\Big ( \rho, m | \bar \rho, \bar m\Big) \operatorname{d}  x&\geq  \frac{C_\kappa}{2 } |\bar \rho- \rho|_{H^1}^2  
              +\frac{c_\rho}{2} \| u - \bar u\|_{L^2}^2 ,\\
\int_{\T^3}\eta_{\mathrm{R}}\Big (  \rho, m | \bar \rho, \bar m\Big)\operatorname{d}  x &\geq \frac{C_\kappa}{4} |\bar \rho-  \rho|_{H^1}^2 
+\frac{C_\kappa c_\rho^2}{8 C_P c_m^2} \|\bar{ m} -  m\|_{L^2}^2.
\end{split}
\end{equation}
where $\norm{\cdot}_{H^1}$ denotes the $H^1$-semi-norm, i.e., $\norm{\rho}_{H^1}:=\norm{\nabla \rho}_{L^2},$ and $C_P$ is the Poincar\'{e} constant on $\T^3.$

Based on  the  relative energy and the relative energy flux we can make the computations for the general case, given in \eqref{reltotekort}, more specific:
\begin{lemma}[Rate of the relative energy]\label{lem:freNSK}
Let $T>0$ be given and let $(\rho, m)$ and $(\bar{\rho},\bar{ m}) $ be strong solutions of \eqref{eq:Kortcap2},
with $\kappa(\rho)=C_\kappa>0,$ corresponding to initial data 
$(\rho_0, m_0)$ and $(\bar{\rho}_0,\bar{ m}_0) ,$ respectively.
Let \eqref{ass12} and \eqref{mean}
 hold.
Then, 
 the rate (of change) of the relative energy defined in \eqref{def:re2NSK} satisfies 
 \begin{equation}
 \frac{\operatorname{d}}{\operatorname{d} t} \int_{\T^3}\eta\Big( \rho, m | \bar \rho, \bar m\Big)\operatorname{d}  x = \int_{\T^3}\sum_{i=1}^4 A_i \operatorname{d}  x \, ,
\end{equation}
with
\begin{equation}
 \begin{split}
 A_1&:=-C_\kappa\frac{\bar m}{\bar \rho} \cdot  (\bar \rho - \rho) \Big (  \nabla \triangle(\bar \rho -  \rho)  \Big )\, , 
\\ 
  A_2&:=\frac{\nabla \bar \rho}{\bar \rho}\otimes \bar u : \Big(  (\bar u - u) \otimes ( \rho \bar u -  \rho u) 
+ \frac{1}{3} \mathbb{I} \big( p(\rho) - p(\bar \rho) - p'(\bar \rho)(\rho - \bar \rho)\big) \Big)\, ,
\\
  A_3&:=\frac{1}{\bar \rho} \nabla \bar m : ( \rho u - \rho  \bar u) \otimes ( \bar u - u)\, ,
\\
  A_4&:=-\frac{\div ( \bar m)}{\bar \rho} \big( p( \rho) - p(\bar \rho) - p'(\bar \rho) (\rho -\bar \rho)
\big)\, .
 \end{split}
\end{equation}
\end{lemma}

Next, we determine the rate of the reduced relative energy based on Lemma \ref{lem:freNSK}.
\begin{lemma}[Rate of the reduced relative energy]\label{lem:rreNSK}
Let the assumptions of Lemma \ref{lem:freNSK} be satisfied.
Then, 
 the rate of the reduced relative energy  defined in \eqref{def:rreNSK} satisfies 
 \begin{equation}
 \frac{\operatorname{d}}{\operatorname{d} t} \int_{\T^3}\eta_{\mathrm{R}}\Big(
            \rho,   m  
             |
          \bar    \rho, \bar  m
       \Big) \operatorname{d}  x = \int_{\T^3}\sum_{i=1}^6 A_i \operatorname{d} x \, ,
\end{equation}
where $A_1,\dots,A_4$ are as in Lemma \ref{lem:freNSK} and
\begin{equation}
 \begin{split}
  A_5&:=\div (\bar m) \big( h'(\rho) - h'(\bar \rho) - h''(\bar \rho)( \rho-\bar \rho) \big)\, ,
\\
  A_6&:= \big( h'( \rho) - h'(\bar \rho)\big) \big( \div(m) - \div(\bar m) \big)\, .
 \end{split}
\end{equation}
\end{lemma}

\begin{proof}
By definition it holds 
 \begin{equation} (\eta_{\mathrm{R}} - \eta) \Big(
             \rho, {  m}    
              |
           \bar  \rho, \bar m
          \Big) = -h(\rho) + h(\bar \rho) + h'(\bar \rho)( \rho-\bar \rho)
   = - \int_{\bar \rho}^{\rho} ( \rho- s) h''(s)\operatorname{d} s\, .
 \end{equation}
Therefore,
\begin{equation}\label{eq:re4NSK}
 \begin{split}
 \frac{\partial}{\partial t} (\eta_{\mathrm{R}} - \eta) \Big(
          \rho, {  m}    
            |
           \bar  \rho, \bar m
          \Big)
=& \bar \rho_t ( \rho-\bar \rho)h''(\bar \rho) - \int_{\bar \rho}^{\rho}  \rho_t h''(s)\operatorname{d} s\\
=& - \div(\bar m) (\rho-\bar \rho)h''(\bar \rho) + \div({ m}) \int_{\bar \rho}^{\rho} h''(s)\operatorname{d} s\\
=& - \div(\bar m) ( \rho-\bar \rho)h''(\bar \rho) + \div({ m})(h'( \rho)-h'(\bar \rho))\\
=& \div(\bar m) (h'(\rho)-h'(\bar \rho) - (\rho-\bar \rho)h''(\bar \rho)) \\
&+ (h'(\rho)-h'(\bar \rho))(\div({ m} ) - \div (\bar m))=A_5+A_6.
 \end{split}
\end{equation}
The assertion of the lemma follows upon combining \eqref{eq:re4NSK} and Lemma \ref{lem:freNSK}.
\end{proof}

\begin{lemma}[Estimate of  the reduced relative energy rate]\label{lem:destNSK}
 Let the assumptions of Lemma \ref{lem:freNSK} be fulfilled and let the initial data $\rho_0, \bar \rho_0$ satisfy
\eqref{mean}.
Then, there exists a constant $C>0$ depending only on $T, C_\kappa ,\bar \rho_0,\bar u_0, c_\rho, C_\rho, c_m$ such that
the rate of change of the reduced relative energy fulfills
\begin{equation}\label{grre}
 \frac{\operatorname{d}}{\operatorname{d} t} \int_{\T^3} \eta_{\mathrm{R}}\Big(\rho,m|\bar \rho, \bar m\Big) \operatorname{d} x
              \leq C \int_{\T^3} \eta_{\mathrm{R}}\Big(\rho,m|\bar \rho, \bar m\Big) \operatorname{d}  x\, .
\end{equation}
\end{lemma}

\begin{proof}
 The proof is based on Lemma \ref{lem:rreNSK} and estimates of the $A_i.$ In order to keep the notation
 manageable we suppress the $t$ dependency of all quantities in the proof.
For example $\| u\|_{C^2}$ refers to $\|u\|_{C^0([0,T],C^2(\T^3))}$
which is bounded and depends on $\rho_0, m_0,T,C_\kappa,$ only.
Let us estimate the $A_i$ one by one:
For $A_1$ we conclude using \eqref{kortstres} and integration by parts 
\begin{equation}\label{est:a2N}
 \begin{split}
&  \int_{{\T^3}} A_1\operatorname{d}   x \\
=& \int_{{\T^3}} -C_\kappa \bar u \cdot ( \nabla \triangle \bar \rho  - \nabla \triangle \rho) (\bar \rho -  \rho) \operatorname{d}   x\\
=& \int_{{\T^3}} C_\kappa    \bar u \cdot\div\Big(\Big( ( \rho  -\bar  \rho)   \triangle(\bar  \rho  -  \rho) - \frac{|\nabla ( \bar \rho  - \rho)|^2 }{2}\Big) \mathbb{I} 
+ \nabla (\bar  \rho  -  \rho) \otimes \nabla ( \bar \rho  -  \rho) 
\Big) \operatorname{d}   x
\\
=& C_\kappa \int_{{\T^3}}   \div (\bar  u)\Big( ( \bar \rho  -  \rho)   \triangle( \bar \rho  - \rho) + \frac{|\nabla (\bar  \rho  -  \rho)|^2 }{2}\Big)  
- \nabla  u : \nabla ( \bar \rho  -  \rho) \otimes \nabla ( \bar \rho  -  \rho) \operatorname{d}   x
\\
 \leq& C_\kappa \Norm{\nabla \div(  \bar u)}_{L^\infty} C_P \norm{\bar \rho - \rho}_{H^1}^2  + \frac{3C_\kappa}{2} \Norm{\nabla   \bar u}_{L^\infty} \norm{\bar\rho - \rho}_{H^1}^2\\
\leq&  5\Norm{  \bar u}_{C^2}C_P \int_{\T^3} \eta_{\mathrm{R}}\Big(\rho,m|\bar \rho, \bar m\Big) \operatorname{d}   x\, .
 \end{split}
\end{equation}
The summand $A_2$ can be estimated as follows:
\begin{equation}\label{est:a3N}
 \begin{split}
   \int_{{\T^3}} A_2\operatorname{d}   x 
=& \int_{{\T^3}} \frac{\nabla\bar \rho}{\bar \rho} \otimes   \bar u : (   \bar u - u) \otimes
( \rho   \bar u -  \rho u)
+ \frac{1}{3}\frac{\nabla \bar \rho}{\bar \rho} \cdot   \bar u (p( \rho)- p(\bar \rho) - p'(\bar \rho)(\rho-\bar \rho))\operatorname{d}   x\\
\leq & \frac{\Norm{\nabla \bar \rho \otimes   \bar u}_{L^\infty}}{c_\rho} \Big(\int_{{\T^3}} \rho |  \bar u - u|^2\operatorname{d}   x + p_M \Norm{\rho - \bar \rho}_{L^2}^2\Big)\\
\leq &\frac{\Norm{\nabla \bar \rho}_{L^\infty} \Norm{  \bar u}_{L^\infty}}{c_\rho} \frac{4p_MC_P}{C_\kappa} 
\int_{\T^3} \eta_{\mathrm{R}}\Big(\rho,m|\bar \rho, \bar m\Big) \operatorname{d}   x\, .
 \end{split}
\end{equation}
 Concerning $A_3$ we find
\[ A_3= \frac{ \rho}{\bar \rho} \nabla \bar m : (u - \bar u)\otimes ( \bar u - u) \]
such that
\begin{equation}\label{est:a4N}
\big| \int_{{\T^3}} A_3\operatorname{d} x\big|\leq  2 \Norm{\frac{\nabla \bar m}{c_\rho}}_{L^\infty}
\int_{\T^3} \eta_{\mathrm{R}}\Big(\rho,m|\bar \rho, \bar m\Big) \operatorname{d} x\, .
\end{equation}
The estimates for $A_4$ and $A_5$ are straightforward, i.e., 
\begin{equation}\label{est:a5N}
 \begin{split}
  \big| \int_{{\T^3}} A_4\operatorname{d} x\big|&
  \leq \frac{1}{c_\rho}\Norm{\div  \bar m}_{L^\infty} p_M \frac{C_P}{C_\kappa} 
\int_{\T^3} \eta_{\mathrm{R}}\Big(\rho,m|\bar \rho, \bar m\Big) \operatorname{d}  x\, ,
\\
  \big| \int_{{\T^3}} A_5\operatorname{d}  x\big|&\leq   \Norm{\div\bar  m}_{L^\infty} h_M \frac{C_P}{C_\kappa} 
\int_{\T^3} \eta_{\mathrm{R}}\Big(\rho,m|\bar \rho, \bar m\Big) \operatorname{d}x\, .
 \end{split}
\end{equation}

Using integration by parts we find for $A_6$
\begin{multline}\label{est:a7N1}
   \int_{{\T^3}} A_{6}\operatorname{d} x
= \int_{{\T^3}} \big( h''(\bar \rho) \nabla \bar\rho - h''(\rho) \nabla  \rho\big)
\cdot \big( m - \bar{m}\big) \operatorname{d}   x\\
= \int_{{\T^3}} \big( h''(\bar \rho) \nabla \bar \rho - h''( \rho) \nabla\bar \rho\big)\cdot \big(m -\bar{  m}\big) \operatorname{d}   x
+ \int_{{\T^3}} \big( h''( \rho) \nabla \bar\rho - h''( \rho) \nabla \rho\big)\cdot \big(m - \bar{  m}\big) \operatorname{d}   x
\end{multline}
such that
\begin{equation}\label{est:a7N2}
\big| \int_{{\T^3}} A_{6}\operatorname{d}   x\big|\leq  
h_M \Norm{\bar \rho}_{C^1} \Norm{\bar \rho - \rho}_{L^2} \Norm{  m - \bar  m}_{L^2}
+ h_M \norm{\bar \rho - \rho}_{H^1} \Norm{ m -\bar   m}_{L^2},
\end{equation}
which implies
\begin{equation}\label{est:a7N3}
\big| \int_{{\T^3}} A_{6}\operatorname{d}   x\big|\leq h_M( \Norm{\bar \rho}_{C^1} + 1) \frac{ C_P c_m^2}{8c_\rho^2C_\kappa} 
\int_{\T^3} \eta_{\mathrm{R}}\Big(\rho,m|\bar \rho, \bar m\Big) \operatorname{d}   x\, ,
\end{equation}
 due to \eqref{eq:er1}.
The assertion of the Lemma follows upon combining \eqref{est:a2N}, \eqref{est:a3N}, \eqref{est:a4N}, 
\eqref{est:a5N}, and \eqref{est:a7N3}.
\end{proof}

Now we are in position to state and prove the main result of this section:
strong solutions to \eqref{eq:Kortcap2} with $\kappa(\rho)=C_\kappa$ depend continuously on their initial data,
provided they satisfy the uniform bounds \eqref{ass12}.

\begin{theorem}[Stability]\label{thrm:mt1}
Let $T,C_\kappa>0$ be given and let $(\rho,   u)$ and $(\bar{\rho},\bar u) $
be strong solutions of \eqref{eq:Kortcap2}, with $\kappa(\rho)=C_\kappa>0,$ corresponding to initial data 
$(\rho_0,   u_0)$ and $(\bar{\rho}_0,\bar{u}_0) ,$ respectively.
Let \eqref{ass12} and
\eqref{mean} hold.
Then, there exists a constant $C =C(T, C_\kappa ,\bar \rho_0,\bar   u_0, c_\rho, C_\rho, c_m)>0$  such that the following estimate is satisfied
 \begin{equation}
\frac{C_\kappa}{2}  \|\rho - \bar \rho\|_{L^\infty(0,T;H^1({\T^3}))}^2 + 
\frac{c_\rho}{2} \Norm{u -   \bar u}_{L^\infty(0,T;L^2({\T^3}))}^2
\leq 
C \int_{\T^3} \eta_{\mathrm{R}}\Big(
             \rho_0, {   m}_0    
           |
          \bar   \rho_0, \bar  m_0
            \Big) \operatorname{d}   x\, .
\end{equation}
\end{theorem}

\begin{proof}
The assertion of the theorem follows by applying Gronwall's Lemma to \eqref{grre} and combining the result with \eqref{eq:er1}$_1$.
\end{proof}

\begin{remark}[Viscosity]
Note that we could also add viscous terms into \eqref{eq:Kortcap2} such that we obtained \eqref{NSK-sec}.
The arguments 
presented here can be extended easily to that case.
In the case with viscosity the results from \cite{Kot08} guarantee the existence of strong solutions for
short times.
\end{remark}

\begin{remark}[Capillarity]
\label{rem:cap1}
The constant $C$ 
in Theorem \ref{thrm:mt1} depends on $C_\kappa$ like $\exp(1/C_\kappa)$ at best,
as can be seen from the estimates of the $A_i$ in the proof of Lemma \ref{lem:destNSK}.
In particular, the constant blows up for $C_\kappa \rightarrow 0.$
\end{remark}


\section{Model convergence}\label{sec:mc}
In this section we employ the relative energy framework to show that the isothermal Navier-Stokes-Korteweg model
\eqref{NSK-sec} 
 is
indeed approximated by the lower order model introduced in \cite{Roh10}, which is given in \eqref{lo-intro}.
Note that \eqref{NSK-sec} is the model investigated in the previous section plus viscosity.
 Before we present our analysis let us digress a bit in order to justify our interest in the
 relation between the two models.
Numerical schemes for the isothermal NSK system \eqref{NSK-sec} have been considered by several authors, see
\cite{GMP_14, BP13,TXKV14} and references therein.
In these works the main effort was directed at overcoming stability issues, which are mainly caused
by the non-convexity of the energy.
Several of the approaches for constructing numerical schemes were based on
Runge-Kutta-discontinuous Galerkin type discretizations.
However, the number of numerical
flux functions which may be used is severely restricted by the non-hyperbolicity of the first order part of \eqref{NSK-sec} which is caused by $h$ being non-convex.

Moreover, discrete energy inequalities for explicit-in-time schemes cannot be proven by standard arguments from hyperbolic theory.
Indeed,
a non-monotone behavior of the energy is observed in numerical experiments.
To overcome these problems the  following family of approximations, parametrized in $\alpha >0,$
of the NSK system was introduced in \cite{Roh10}:
\begin{equation}\label{lo}
 \begin{split}
   \rhoa_t + \div(\rhoa \va) &=0\\
   (\rhoa \va)_t + \div (\rhoa \va \otimes \va) + \nabla (p(\rhoa) + C_\kappa \frac{\alpha}{2} (\rhoa)^2)&=
   \div(\sigma[\va]) + C_\kappa\alpha \rhoa \nabla  \ca\\
  \ca - \frac{1}{\alpha}\triangle \ca &= \rhoa,
 \end{split}
\end{equation}
where $\ca$ is an auxiliary variable without any immediate physical interpretation. 
It is a striking feature of \eqref{lo} that the first order part of \eqref{lo}$_{1,2}$ forms  a hyperbolic system for $\rhoa,\, \ma,$
provided $\alpha$ is sufficiently large.
Numerical studies showing that \eqref{lo} offers numerical advantages over \eqref{NSK-sec}
 and that solutions of \eqref{lo} are similar to those of \eqref{NSK-sec} can be found in \cite{NRS14}.
In particular, examples are presented in \cite{NRS14} which show that explicit-in-time schemes for \eqref{lo} have far better stability properties than explicit-in-time schemes for \eqref{NSK-sec}.
Variational problems related to minima of the energy functional of \eqref{lo}, see Lemma \ref{lem:rel}, were investigated in \cite{BLR95,SV03}.
Based on formal arguments it was conjectured in \cite{Roh10} that for $\alpha \rightarrow \infty$ solutions of  \eqref{lo} converge to solutions of \eqref{NSK-sec}.
Results in this direction were obtained in \cite{Cha14} using Fourier methods and for similar models describing elastic solids in \cite{ERV14,Gie_14a}. 
The result in \cite{ERV14} is obtained using compactness arguments, while \cite{Gie_14a} is based on  a (technically simpler) version of the arguments presented here.
The main result of this section, Theorem \ref{thrm:mt2}, is an estimate for the difference between solutions of \eqref{lo} and \eqref{NSK-sec}.

\subsection{Assumptions on well-posedness and uniform bounds}
We
 complement \eqref{NSK-sec}, \eqref{lo} with initial data 
\begin{equation}\label{ic}\begin{split}
 &\rho(0,\cdot)= \rho_0 ,\quad u (0,\cdot ) = u_0 \quad \text{in } \mathbb{T}^3,\\
&\rhoa(0,\cdot)= \rho_0^\alpha ,\quad \va(0,\cdot) =\va_0 \quad \text{in } \mathbb{T}^3,
\end{split}\end{equation}
for given data $\rhoa_0,\rho_0 \in C^3(\mathbb{T}^3,(0,\infty))$ and 
$\va_0, u_0 \in C^2(\mathbb{T}^3,\mathbb{R}^3)$
which we assume to be related as follows
\begin{equation}\label{ass:ic}
\begin{split}
& \int_{\mathbb{T}^3}( \rhoa_0 - \rho_0) \operatorname{d}  x=0,
\quad \Norm{\rhoa_0 - \rho_0}_{H^1(\mathbb{T}^3)} = \mathcal{O}(\alpha^{-1/2})\, , \\
&\Norm{\rhoa_0}_{H^3(\mathbb{T}^3)} = \mathcal{O}(1),
\quad \Norm{\va_0 -  u_0}_{L^2(\mathbb{T}^3)} = \mathcal{O}(\alpha^{-1/2})\, . 
\end{split}
\end{equation}

Concerning the viscous part of the stress we will require that there is bulk viscosity, i.e.,
\begin{equation}\label{visc} \lambda + \frac{2}{3}\mu >0.\end{equation}

The well-posedness of \eqref{lo} was studied in \cite{Roh10} for two space dimensions on the whole of $\mathbb{R}^2$.
 We will assume  (local-in-time) existence of strong solutions to  \eqref{lo} and \eqref{NSK-sec}
 posed on $\mathbb{T}^3.$ In particular:
\begin{assumption}[Regularity]
 We assume that there is some $T>0$ such that strong solutions of \eqref{lo} and \eqref{NSK-sec} exist and satisfy
\begin{equation}\label{ass1}
 \begin{split}
\rho &\in C^0([0,T],C^3(\mathbb{T}^3,\mathbb{R}_+)) \cap C^1((0,T),C^1(\mathbb{T}^3,\mathbb{R}_+))\, ,\\
   u &\in C^0([0,T],C^{2}(\mathbb{T}^3,\mathbb{R}^3)) \cap C^1((0,T),C^0(\mathbb{T}^3,\mathbb{R}^3))\, ,\\
  \rhoa &\in C^0([0,T],C^1(\mathbb{T}^3,\mathbb{R}_+)) \cap C^1((0,T),C^0(\mathbb{T}^3,\mathbb{R}_+))\, ,\\
   \va &\in C^0([0,T],C^{2}(\mathbb{T}^3,\mathbb{R}^3)) \cap C^1((0,T),C^0(\mathbb{T}^3,\mathbb{R}^3))\, ,\\
   \ca &\in C^0([0,T],C^3(\mathbb{T}^3,\mathbb{R}_+)) \cap C^1((0,T),C^2(\mathbb{T}^3,\mathbb{R}_+))\, .
 \end{split}
\end{equation}
\end{assumption}

\begin{remark}[Regularity]\label{rem:reg}
Note that the regularity assumed in \eqref{ass1} coincides with the regularity asserted in  \cite{Roh10} and \cite{Kot08},
for $T$ small enough.
Therefore, for appropriate $T$, the only assumptions made here are that the change from natural to periodic boundary 
conditions does not deteriorate the 
regularity of solutions and that the time of existence of solutions of \eqref{lo} does not go to zero
for $\alpha \rightarrow \infty$.
The existence of a space derivative of $\rho_t$ follows from the mass conservation equation
and the regularity of $\rho, u.$
\end{remark}
\begin{assumption}[Uniform bounds]
We assume that there are constants $\alpha_0,C_\rho,c_\rho,c_m >0$ such that 
\begin{equation}\label{ass2}
 \begin{split}
  c_\rho &\leq \rho(t, x), \rhoa(t, x) \leq C_\rho  \quad \forall \, t \in [0,T],\, x \in \mathbb{T}^3,\,
  \alpha>\alpha_0\, ,\\
 c_m &\geq \sup \{\Norm{ \ma}_{L^\infty((0,T)\times \mathbb{T}^3)},\, \Norm{  m}_{L^\infty((0,T)\times \mathbb{T}^3)}\}\, .
 \end{split}
\end{equation}
\end{assumption}

\begin{remark}[Uniform a-priori estimates]
 The crucial assumption in \eqref{ass2} is that these estimates hold uniformly in $\alpha$,
 while analogous estimates for fixed $\alpha$ are immediate for sufficiently small times
and appropriate initial data.
We are not aware that there is any mechanism in \eqref{lo} which makes \eqref{ass2} unlikely to hold.
While  it would be desirable to prove \eqref{ass2}, this is beyond the scope of this work.
\end{remark}

\begin{remark}[A-priori estimate on $\ca$]\label{rem:max}
 Note that the maximum principle applied to the screened Poisson equation \eqref{lo}$_3$ immediately implies
\[c_\rho \leq  \ca(t, x)\leq C_\rho  \quad \forall \, t \in [0,T],\, x \in \mathbb{T}^3,\, \alpha>\alpha_0\, ,\]
once we assume \eqref{ass2}.
\end{remark}

Note that, using \eqref{lo}$_3$,  \eqref{lo}$_2$ can be rewritten  as
\begin{equation}
 (\rhoa \va)_t + \div (\rhoa \va \otimes \va) + \nabla p(\rhoa) = \div(\sigma[\va]) + C_\kappa \rhoa \nabla \triangle \ca.
\end{equation}
Moreover,  \eqref{lo} conserves momentum, since a stress tensor exists, analogous to \eqref{stressep}.

As already noted in \cite{Roh10}, solutions of \eqref{lo} satisfy an energy inequality.
In order to keep this paper self contained we state the energy inequality and sketch its proof.

\begin{lemma}[Energy balance for \eqref{lo}]\label{lem:rel}
 Let $(\rhoa,\va, \ca)$ be  a strong solution of \eqref{lo}. Then, the following energy balance law is satisfied:
\begin{multline}
0 \geq - \sigma[\va] : \nabla \va=
\big( h(\rhoa) + \frac{\rho}{2}\norm{\va}^2 +
\frac{C_\kappa}{2} \norm{\nabla \ca}^2 + \frac{\alpha C_\kappa}{2} \norm{\rhoa - \ca}^2\big)_t\\
+ \div \big( \va \big( \rhoa h'(\rhoa) + \frac{1}{2} \rhoa \norm{\va}^2
- C_\kappa \rhoa\triangle\ca\big)- C_\kappa \ca_t \nabla \ca
 + C_\kappa \rhoa \nabla \rhoa\cdot \nabla \va
- \sigma[\va] \va\big)
\end{multline}
and 
\begin{equation}\label{divva}
 \Norm{ \div  \va }_{L^2([0,T]\times \mathbb{T}^3)}^2 \leq 
 \frac{1}{\lambda + \frac{2}{3}\mu} \int_{\mathbb{T}^3} h(\rhoa_0) + \frac{\rhoa_0}{2} \norm{\va_0}^2 
+ \frac{C_\kappa\alpha} {2}\norm{\rhoa_0 - \ca_0}^2 + \frac{C_\kappa}{2}\norm{\ca_0}^2 \operatorname{d}  x\, .
\end{equation}
\end{lemma}
\begin{proof}
 Equation \eqref{lo}$_1$ is multiplied by $h'(\rhoa) - \frac{1}{2} |\va|^2 + \alpha C_\kappa (\rhoa - \ca)$ and \eqref{lo}$_2$ is multiplied by $\va$.
Then, both equations are added and integration by parts is used. The first assertion of the lemma follows upon noting that
\[ \nabla \ca \cdot \nabla \ca_t - \alpha  \ca_t (\rhoa - \ca) - \div ( \ca_t \nabla \ca) =0\, .\] 
The second assertion of the lemma is a result of the first assertion and 
the identity
\begin{equation}\label{dissip}
 \sigma[u] : \nabla u = (\lambda + \frac{2}{3} \mu) (\div u)^2 + 2\mu \nabla^o u : \nabla^o u \geq 0\, ,
\end{equation}
where $\nabla^o$ denotes the trace free part of the Jacobian.
\end{proof}

\subsection{Elliptic approximation}
In this section we study properties of the screened Poisson operator in \eqref{lo}$_3$.
To quantify the approximation of $\rhoa$ by $\ca$ we recall the following result from \cite{Gie_14a}. The solution operator to
$\operatorname{Id}- \frac{1}{\alpha}\triangle$ on $\mathbb{T}^3$ is denoted by $G_\alpha:$
\begin{lemma}[Elliptic approximation, \cite{Gie_14a}]\label{lem:ea}
The operator $G_\alpha$ has the following properties:
\begin{enumerate}
 \item[(a)]For any $f \in L^2({\T^3})$ the following estimate is fulfilled
\begin{equation}\label{est:g}
 \| G_\alpha[f]\|_{L^2({\T^3})}\leq \| f\|_{L^2({\T^3})}\, .
\end{equation}
\item[(b)] 
For any $k \in \mathbb{N}$ and $f \in H^k(\T^3)$ it is  $G_\alpha[f] \in H^{k+2}(\T^3)$. 
\item[(c)]
For all $f \in H^1(\T^3)$ the following holds:
\begin{equation}\label{eq:div}
G_\alpha[f_x] =(G_\alpha[f])_x \quad \text{ and } \quad  \| f - G_\alpha[f]\|_{L^2({\T^3})}^2
 \leq \frac{2}{\alpha} |f|_{H^1({\T^3})}^2\, .
\end{equation}
\item[(d)] In case $f \in H^2(\T^3)$ the following (stronger) estimate is satisfied:
\[  \| f - G_\alpha[f]\|_{L^2({\T^3})}^2
 \leq \frac{1}{\alpha^2} |f|_{H^2({\T^3})}^2\, .\]
\end{enumerate}
\end{lemma}

\begin{remark}\label{rem:divva}
 As $\Norm{\rhoa_0}_{H^1}$ is bounded by \eqref{ass:ic}, assertion (c) of Lemma \ref{lem:ea} implies
 that the initial energy of the 
lower order model is bounded independent of $\alpha.$ Due to \eqref{divva} this implies that
$\Norm{\div \va}_{L^2([0,T]\times \mathbb{T}^3)}$ is bounded independent of $\alpha.$ 
\end{remark}

\subsection{Relative energy}
In this section we study the relative energy and relative energy flux between a solution
$(\rho ,   u)$ of  \eqref{NSK-sec}  and a solution $(\rhoa,\va, \ca)$ of \eqref{lo}.
They are based on the energies and energy fluxes 
of the systems \eqref{lo} and \eqref{NSK-sec} determined in Lemmas \ref{lem:ret} and \ref{lem:rel},
but we omit the higher order terms, i.e., those depending on $C_\kappa,$ in the relative energy flux:
\begin{equation}\label{def:re}
\begin{split}
  \eta^\alpha &:=  h(\rhoa)  + \frac{C_\kappa}{2}\norm{\nabla \ca}^2
  + \frac{\alpha C_\kappa}{2} \norm{\rhoa - \ca}^2
-h(\rho)   -  \frac{C_\kappa}{2}\norm{\nabla \rho}^2
\\& \quad
  -h'(\rho) (\rhoa - \rho)  
- C_\kappa  \nabla \rho \cdot \nabla(\ca -\rho) + \frac{\rhoa}{2}|u - \va|^2,
\\
  q^\alpha &:= \ma h'(\rhoa) + \frac{\norm{\ma}^2}{2(\rhoa)^2}  \ma  -   m h'(\rho) - \frac{\norm{  m}^2}{2\rho^2}  m
- h'(\rho)(\ma -   m)
\\&\quad
 +\frac{\norm{  m}^2}{2\rho^2} (\ma -   m)
 - \frac{  m}{\rho}\cdot \Big(\frac{\ma \otimes \ma}{\rhoa} + p(\rhoa) -\frac{  m \otimes   m}{\rho} - p(\rho) 
\Big)\, .
\end{split}
\end{equation}
Several terms in \eqref{def:re} cancel out, such that
\begin{equation}\label{def:re2}
\begin{split}
  \eta^\alpha &=  h(\rhoa) -h(\rho)-h'(\rho) (\rhoa - \rho)  + \frac{C_\kappa}{2}\norm{\nabla( \ca -  \rho)}^2 + \frac{\alpha C_\kappa}{2} \norm{\rhoa - \ca}^2
 + \frac{\rhoa}{2}|u - \va|^2,
\\
  q^\alpha &= \ma h'(\rhoa) + \frac{\norm{\ma}^2}{2(\rhoa)^2}  \ma  
 - \ma h'(\rho) +\frac{ \norm{  m}^2}{2\rho^2}   m - \frac{  m \cdot \ma}{\rhoa \rho} \ma - \frac{p(\rhoa)}{\rho}   m +  \frac{p(\rho)}{\rho}   m\, .
\end{split}
\end{equation}

As in Section \ref{sec:cd}, the relative energy is not suitable for measuring the distance between solutions
and, again, we introduce the {\it reduced relative energy} and the {\it relative kinetic energy}:
\begin{equation}\label{def:rre}
 \begin{split}
  \eta^\alpha_{\mathrm{R}} &:=  \frac{C_\kappa}{2}\norm{\nabla \ca- \nabla \rho}^2 + \frac{\alpha C_\kappa}{2} \norm{\rhoa - \ca}^2 +  K^\alpha,\\
 K^\alpha &:=  \frac{\rhoa}{2}|u - \va|^2.
 \end{split}
\end{equation}

Note that calculations analogous to the derivation of \eqref{eq:er1} imply
\begin{equation}\label{eq:er12}
\begin{split}
\int_{\mathbb{T}^3}  \eta_{\mathrm{R}}^\alpha \operatorname{d}   x&\geq  \frac{C_\kappa}{2}|\ca- \rho |_{H^1}^2 + \frac{\alpha C_\kappa}{2} \Norm{  \ca- \rhoa}_{L^2}^2 +\frac{\rhoa}{2} \|   u - \va\|_{L^2}^2\, ,\\
\int_{\mathbb{T}^3}  \eta_{\mathrm{R}}^\alpha\operatorname{d}   x &\geq \frac{C_\kappa}{4}| \ca- \rho |_{H^1}^2 + \frac{\alpha C_\kappa}{4} \Norm{  \ca- \rhoa}_{L^2}^2
+\frac{C_\kappa c_\rho^2}{8 C_P c_m^2} \|m -   \ma\|_{L^2}^2\, .
\end{split}
\end{equation}

Based on  the  relative energy and the relative energy flux we can show the following estimate whose proof
is based on the same principles as the derivation of \eqref{eq:RelEnKorGenFinal}.
However, an additional difficulty is presented by the fact that the functions which are compared solve different
PDEs.

\begin{lemma}[Rate of the relative energy]\label{lem:fre}
Let $T>0$ be given such that there are strong solutions $(\rho , u)$ and
$(\rhoa,\va, \ca)$ of \eqref{NSK-sec} and \eqref{lo}, respectively,
 satisfying \eqref{ass1} and \eqref{ass2}.
Then, 
 the rate (of change) of the relative energy $\eta^\alpha$ defined in \eqref{def:re} fulfills
 \begin{equation}
 \frac{\operatorname{d}}{\operatorname{d} t} \int_{\mathbb{T}^3}\eta^\alpha \operatorname{d}   x= \int_{\mathbb{T}^3}\sum_{i=1}^6 A_i^\alpha \operatorname{d}   x\,  ,
\end{equation}
with
\begin{equation}
 \begin{split}
  A_1^\alpha&:=\frac{\div \sigma[  u]}{\rho}\cdot (  u - \va) (\rhoa -\rho) +
  (  u - \va)\cdot (\div \sigma[  u] - \div \sigma[\va])\, ,
\\
  A_2^\alpha&:=-C_\kappa\frac{  m}{\rho}\cdot \nabla \triangle(\rho - \ca) (\rho - \rhoa)\, ,
\\ 
  A_3^\alpha&:=\frac{\nabla \rho}{\rho}\otimes   u : \Big(  (  u - \va) \otimes (\rhoa   u - \rhoa \va) 
+ \frac{1}{3} \mathbb{I} \big( p(\rhoa) - p(\rho) - p'(\rho)(\rhoa - \rho)\big) \Big)\, ,
\\ 
  A_4^\alpha&:=\frac{1}{\rho} \nabla   m : (\rhoa \va - \rhoa   u) \otimes (  u - \va)\, ,
\\
  A_5^\alpha&:=-\frac{\div(   m)}{\rho} \big( p(\rhoa) - p(\rho) - p'(\rho) (\rhoa - \rho)\big)\, ,
\\
  A_6^\alpha&:=C_\kappa (\rhoa_t - \ca_t) \triangle  \rho\,   .
 \end{split}
\end{equation}
\end{lemma}
The proof of this Lemma is given in Appendix \ref{sec:proofmc}.

\begin{remark}
 Note that all but two of the $A_i^\alpha$ in Lemma \ref{lem:fre} correspond to terms in Lemma \ref{lem:freNSK}.
The term $A_1^\alpha$ is due to viscosity and $ A_6^\alpha$ is due to the different regularizations.
These are the only terms having no counterparts in  Lemma \ref{lem:freNSK}.
\end{remark}

Our next step is to derive a representation of the rate of the reduced relative energy from Lemma \ref{lem:fre}.
\begin{lemma}[Rate of the reduced relative energy]\label{lem:rre}
Let the assumptions of Lemma \ref{lem:fre} be satisfied.
Then,
 the rate of the reduced relative energy $\eta^\alpha_{\mathrm{R}}$ defined in \eqref{def:rre} satisfies
 \begin{equation}
 \frac{\operatorname{d}}{\operatorname{d} t} \int_{\mathbb{T}^3}\eta^\alpha_{\mathrm{R}} \operatorname{d}x
 = \int_{\mathbb{T}^3}\sum_{i=1}^8 A_i^\alpha \operatorname{d}x \, ,
\end{equation}
where $A_1,\dots,A_6$ are as in Lemma \ref{lem:fre} and
\begin{equation}
 \begin{split}
  A_7^\alpha&:=\div (  m) \big( h'(\rhoa) - h'(\rho) - h''(\rho)(\rhoa-\rho) \big)\, ,
\\
  A_8^\alpha&:= \big( h'(\rhoa) - h'(\rho)\big) \big( \div(\ma) - \div(  m) \big)\, .
 \end{split}
\end{equation}
\end{lemma}
The proof of Lemma \ref{lem:rre} is analogous to the proof of Lemma \ref{lem:rreNSK}.

\begin{lemma}[Estimate of the reduced relative energy rate]\label{lem:dest}
 Let the assumptions of Lemma \ref{lem:fre} be satisfied. Then, there exist
 a constant $C>0$ and a function $E \in L^1(0,T),$ both independent of $\alpha>\alpha_0,$ such that
the rate (of change) of the reduced relative energy satisfies the following estimate
\begin{equation}\label{grreNSK}
 \frac{\operatorname{d}}{\operatorname{d} t} \int_{\mathbb{T}^3} \eta^\alpha_{\mathrm{R}} \operatorname{d}   x \leq
 C  \int_{\mathbb{T}^3} \eta^\alpha_{\mathrm{R}}\operatorname{d}   x + \frac{E}{\alpha}\, .
\end{equation}

\end{lemma}

\begin{proof}
 The proof is based on Lemma \ref{lem:rre} and estimates of the $A_i^\alpha.$ For brevity we suppress the $t$ dependency of all quantities in the proof.
The terms $A_3^\alpha,\, A_4^\alpha,\, A_5^\alpha,\, A_7^\alpha$ can be estimated analogous to the corresponding terms in the proof of Lemma \ref{lem:destNSK}.
Let us estimate the remaining $A_i^\alpha$ one by one:
we have, using integration by parts,
\begin{multline}
\label{est:a1N}
 \int_{{\T^3}} A_1^\alpha\operatorname{d}   x \leq \frac{\Norm{  u}_{C^2}}{c_\rho} \Norm{  u - \va}_{L^2({\T^3})}
 \Norm{\rho - \rhoa}_{L^2({\T^3})}
  -\int_{\T^3} \big( \nabla   u - \nabla \va\big) 
  (\sigma[  u] - \sigma [\va])\operatorname{d}   x\\
\leq \frac{\Norm{  u}_{C^2}}{c_\rho}
\Norm{  u - \va}_{L^2({\T^3})} \Norm{\rho - \bar \rho}_{L^2({\T^3})}\\
\leq \frac{\Norm{  u}_{C^2}}{c_\rho}\max\Big\{\frac{1}{c_\rho}, \frac{4C_P}{C_\kappa} \Big\} 
\int_{\T^3} \eta_{\mathrm{R}}^\alpha \operatorname{d} x
 \end{multline}
as, by the properties of the Lam\'{e} coefficients,
$\nabla(\va -  u) : (\sigma[\va] - \sigma [ u]) \geq 0.$
For $A_2^\alpha$ we find using \eqref{kortstres}
\begin{equation}
\label{est:a21}
 \begin{split}
 & \int_{\mathbb{T}^3} A_2^\alpha\operatorname{d}   x \\
=& \int_{\mathbb{T}^3} -C_\kappa  u \cdot ( \nabla \triangle \rho  - \nabla \triangle \ca) (\rho - \rhoa) \operatorname{d}   x\\
=& \int_{\mathbb{T}^3} -C_\kappa (\rho - \ca)   u \cdot \nabla \triangle(  \rho  -  \ca)  -C_\kappa (\ca - \rhoa)   u \cdot \nabla \triangle(  \rho  -  \ca) \operatorname{d}   x\\
=& C_\kappa \int_{\mathbb{T}^3} -    u\cdot \div\Big(\Big( (  \rho  -  \ca)   \triangle(  \rho  -  \ca) + \frac{|\nabla (  \rho  -  \ca)|^2 }{2}\Big) \mathbb{I} 
- \nabla (  \rho  -  \ca) \otimes \nabla (  \rho  -  \ca) 
\Big)\\
& \qquad-(\ca - \rhoa)   u \cdot \nabla \triangle \rho  +  (\ca - \rhoa)   u \cdot \nabla \triangle  \ca \operatorname{d}   x
\\
=& C_\kappa\int_{\mathbb{T}^3}  \div (  u)\Big( (  \rho  -  \ca)   \triangle(  \rho  -  \ca) + \frac{|\nabla (  \rho  -  \ca)|^2 }{2}\Big)  
- \nabla  u : \nabla (  \rho  -  \ca) \otimes \nabla (  \rho  -  \ca) 
\\
& \qquad-(\ca - \rhoa)   u \cdot \nabla \triangle \rho  + \frac{1}{\alpha} \triangle \ca   u \cdot \nabla \triangle  \ca \operatorname{d}   x\\
 \leq& C_\kappa \Norm{\nabla \div(  u)}_{L^\infty}\Norm{\rho - \ca}_{H^1}^2  + \frac{3C_\kappa}{2} \Norm{\nabla   u}_{L^\infty} \Norm{\rho - \ca}_{H^1}^2\\
&+ C_\kappa \|\ca - \rhoa\|_{L^2} \Norm{\rho}_{C^3} \Norm{  u}_{C^0} - \int_{\mathbb{T}^3} \frac{C_\kappa}{\alpha} \div(  u) (\triangle \ca)^2 \operatorname{d}   x\\
\leq& 3 \Norm{  u}_{C^2} C_\kappa \Norm{\rho - \ca}_{H^1}^2 
+  C_\kappa \|\ca - \rhoa\|_{L^2} \Norm{\rho}_{C^3} \Norm{  u}_{C^0}
+ C_\kappa \alpha \Norm{  u}_{C^1} \Norm{\rhoa -\ca}_{L^2}^2\\
\leq& \big(6 \Norm{  u}_{C^2} C_P + \frac{2}{\alpha} \Norm{\rho}_{C^3} \Norm{  u}_{C^0} \big) \int_{\T^3} \eta_{\mathrm{R}}^\alpha \operatorname{d} x
\, .
 \end{split}
\end{equation}
For $A_6^\alpha$ we  find 
\begin{multline}\label{est:a4}
   \int_{\mathbb{T}^3} A_6^\alpha\operatorname{d}   x
= C_\kappa\int_{\mathbb{T}^3} \triangle \rho (\ca_t - \rhoa_t )\operatorname{d}   x\\
= C_\kappa\int_{\mathbb{T}^3} \triangle \rho (G_\alpha[\div(\ma)] - \div(\ma))\operatorname{d}   x
= - C_\kappa \int_{\mathbb{T}^3} \nabla \triangle\rho \cdot (G_\alpha [\ma] - \ma) \operatorname{d}   x\\
= - C_\kappa \int_{\mathbb{T}^3} \nabla \triangle\rho \cdot (G_\alpha [  m] -   m) \operatorname{d}   x  - C_\kappa \int_{\mathbb{T}^3} \nabla \triangle\rho \cdot (G_\alpha [\ma -   m] - (\ma -   m)) \operatorname{d}   x
\end{multline}
such that elliptic regularity for the operator $G_\alpha$ implies
\begin{multline}\label{est:a41}
\big| \int_{\mathbb{T}^3} A_6^\alpha\operatorname{d}   x\big|\leq
\Norm{\rho}_{C^3} \frac{1}{\alpha} \norm{  m}_{H^2} + 2 C_\kappa\Norm{\rho}_{C^3} \Norm{\ma -   m}_{L^2}\\
\leq \Norm{\rho}_{C^3} \frac{1}{\alpha} \norm{  m}_{H^2} + 2 C_\kappa\Norm{\rho}_{C^3} \frac{16C_P c_m^2}{c_\rho^2}\int_{\T^3} \eta_{\mathrm{R}}^\alpha \operatorname{d} x\, ,
\end{multline}
see Lemma \ref{lem:ea}.
In order to estimate $A_8^\alpha$ we decompose it as $-A_8^\alpha =A_{81}^\alpha+A_{82}^\alpha$ with
\begin{equation}
 \begin{split}
A_{81}^\alpha &= \big( h'(\rho) - h'(\ca)\big) \big(\div(\ma) - \div(  m)\big)\, ,\\
A_{82}^\alpha &= \big( h'(\ca) - h'(\rhoa)\big) \big(\div(\ma) - \div(  m)\big)\, .
 \end{split}
\end{equation}
Using integration by parts we find
\begin{multline}\label{est:a81}
   \int_{\mathbb{T}^3} A_{81}^\alpha\operatorname{d}   x
= \int_{\mathbb{T}^3} \big( h''(\rho) \nabla \rho - h''(\ca) \nabla \ca\big)
\cdot\big(  m - \ma\big) \operatorname{d}   x\\
= \int_{\mathbb{T}^3} \big( h''(\rho) \nabla \rho - h''(\ca) \nabla \rho\big)
\cdot \big(  m - \ma\big) \operatorname{d}   x
+ \int_{\mathbb{T}^3} \big( h''(\ca) \nabla \rho - h''(\ca) \nabla \ca\big)
\cdot \big(  m - \ma\big) \operatorname{d}   x
\end{multline}
such that
\begin{equation}\label{est:a811}
\big| \int_{\mathbb{T}^3} A_{81}^\alpha\operatorname{d}   x\big|\leq
h_M \Norm{\rho}_{C^1} \Norm{\rho - \ca}_{L^2} \Norm{\ma -   m}_{L^2}
+ h_M \norm{\rho - \ca}_{H^1} \Norm{\ma -   m}_{L^2},
\end{equation}
because of Remark \ref{rem:max}.
We infer
\begin{equation}\label{est:a812}
\big| \int_{\mathbb{T}^3} A_{81}^\alpha\operatorname{d}   x\big|\leq
h_M( \Norm{\rho}_{C^1} + 1) \frac{ C_P c_m^2}{4C_\kappa c_\rho^2} \int_{\T^3} \eta_{\mathrm{R}}^\alpha \operatorname{d} x\, ,
\end{equation}
 due to \eqref{eq:er12}.
Finally, we turn to $A_{82}^\alpha$ and obtain 
\begin{equation*}
 A_{82}^\alpha = (h'(\ca)-h'(\rhoa)) \rhoa \div \va + (h'(\ca)-h'(\rhoa))\nabla \rhoa \cdot \va
- (h'(\ca)-h'(\rhoa)) \div   m
\end{equation*}
such that, using Young's inequality,
\begin{multline}
 \label{est:a821}
\big| \int_{\mathbb{T}^3} A_{82}^\alpha\operatorname{d}   x\big|\leq
\frac{C_\kappa\alpha}{2} \Norm{\rhoa-\ca}_{L^2(\mathbb{T}^3)}^2 \\
+ \frac{2h_M}{\alpha C_\kappa} \big( \Norm{\div   m}_{L^2(\mathbb{T}^3)}^2
+ C_\rho \Norm{\div \va}_{L^2(\mathbb{T}^3)}^2\Big)
+ \big| \int_{\mathbb{T}^3} A_{83}^\alpha\operatorname{d}   x\big|\, ,
\end{multline}
where $A_{83}^\alpha:=(h'(\ca)-h'(\rhoa))\nabla \rhoa \cdot \va,$ for $\alpha$ sufficiently large 
and some constant $C_\rho>0$, according to \eqref{ass2}.
We have, using integration by parts,
\begin{equation}\label{est:a83}
 \begin{split}
  \big| \int_{\mathbb{T}^3} A_{83}^\alpha\operatorname{d}   x\big|
&= \big| \int_{\mathbb{T}^3}h'(\ca)\nabla (\rhoa - \ca) \cdot \va - \nabla(h(\rhoa) -h(\ca)) \cdot \va \operatorname{d}   x\big|\\
&\leq \frac{C_\kappa\alpha}{2} \Norm{\rhoa-\ca}_{L^2(\mathbb{T}^3)}^2 + \Big(\frac{3h_M^2}{C_\kappa \alpha} \Norm{\nabla \ca}_{L^2(\mathbb{T}^3)}^2 + \frac{2}{C_\kappa \alpha}\Norm{\div \va}_{L^2(\mathbb{T}^3)}^2\big)\, .
 \end{split}
\end{equation}
Note that $\frac{C_\kappa\alpha}{2} \Norm{\rhoa-\ca}_{L^2(\mathbb{T}^3)}^2 \leq \int_{\T^3} \eta_{\mathrm{R}}^\alpha \operatorname{d} x.$

The assertion of the lemma follows upon combining \eqref{est:a21},  \eqref{est:a41},
\eqref{est:a812},\eqref{est:a821}, and \eqref{est:a83} because of \eqref{ass1}, \eqref{ass2}, and Remark \ref{rem:divva}.
\end{proof}

\begin{remark}[Boundary conditions]\label{rem:bc}
 If we used natural boundary conditions instead of periodic ones,
 we would obtain additional (non-vanishing) boundary terms in  the estimate of $A_6^\alpha$,
and it is not clear how to estimate them properly.
\end{remark}

\begin{theorem}[Model convergence]\label{thrm:mt2}
Let $T,\mu,C_\kappa>0$ and $\lambda \in \mathbb{R}$ be fixed such that \eqref{visc} is satisfied.
Let initial data $(\rho_0,  u_0) \in H^3(\mathbb{T}^3)\times H^2(\mathbb{T}^3)$ be given 
such that a strong solution $(\rho,   u)$ of \eqref{NSK-sec} exists in the sense of \eqref{ass1}.
Let the sequence $(\rhoa_0,\va_0)$ be such that \eqref{ass:ic} is satisfied and that strong solutions 
$(\rhoa, \va,\ca)$ of \eqref{lo} exist in the sense of \eqref{ass1}.
Let, in addition, \eqref{ass2} be satisfied.
Then, there exists a constant $C>0$ depending only on $T,\lambda, \mu,C_\kappa, \rho_0,  u_0$ 
such that for sufficiently large $\alpha$ the following estimate holds:
 \begin{equation*}
\frac{C_\kappa}{2}  \norm{\rho - \ca}_{L^\infty(0,T;H^1(\mathbb{T}^3))}^2 + 
\frac{C_\kappa \alpha}{2}  \Norm{\rhoa - \ca}_{L^\infty(0,T;L^2(\mathbb{T}^3))}^2 +
\frac{c_\rho}{2} \Norm{\va -   u}_{L^\infty(0,T;L^2(\mathbb{T}^3))}^2
\leq \frac{C}{\alpha}\,   .
\end{equation*}
\end{theorem}

\begin{proof}
Integrating \eqref{grreNSK} in time we obtain
 \begin{equation}
 \int_{\mathbb{T}^3} \eta^\alpha_{\mathrm{R}}(t,  x) \operatorname{d}   x
 \leq C \int_0^t\int_{\mathbb{T}^3} \eta^\alpha_{\mathrm{R}}(s,  x)\operatorname{d}   x \operatorname{d} s 
+ \int_{\mathbb{T}^3} \rre(0,  x) \operatorname{d}   x+ \int_0^t\frac{|E(s)|}{\alpha}\operatorname{d} s
\end{equation}
such that Gronwall's inequality implies
 \begin{equation}\label{Gw}
 \int_{\mathbb{T}^3} \rre(t,  x) \operatorname{d}   x \leq
 \Big(  \int_{\mathbb{T}^3} \rre(0,  x) \operatorname{d}   x+
 \int_0^t\frac{|E(s)|}{\alpha}\operatorname{d} s\Big) e^{Ct}                      .
\end{equation}
In order to infer the assertion of the theorem, from \eqref{Gw} we need to estimate the integral of $\rre(0,\cdot).$
 Because of \eqref{ic} we have
\begin{multline}\label{re0}
 \int_{\mathbb{T}^3} \rre(0,  x) \operatorname{d}   x \leq
 \frac{\alpha C_\kappa}{2} \Norm{\ca(0,\cdot) - \rhoa_0}_{L^2(\mathbb{T}^3)}^2 \\
 + \frac{C_\kappa}{2C_P} \|\ca(0,\cdot) - \rho_0\|_{H^1(\mathbb{T}^3)}^2
+ C_\rho  \Norm{\va_0 -   u_0}_{L^2(\mathbb{T}^3)}^2.
\end{multline}
By definition
$\ca(0,\cdot)=G_\alpha[\rho_0^\alpha]$ such that
\begin{equation}\label{ereg}
 \begin{split}
   \Norm{\ca(0,\cdot) - \rhoa_0}_{L^2(\mathbb{T}^3)}^2 &\leq \frac{1}{\alpha^2} \norm{ \rhoa_0}_{H^2(\mathbb{T}^3)}^2,\\
\frac{1}{C_P}\|\ca(0,\cdot) - \rho_0\|_{H^1(\mathbb{T}^3)}^2 &\leq \frac{2}{\alpha^2}
\norm{ \rhoa_0}_{H^3(\mathbb{T}^3)}^2 +  \Norm{\rhoa_0 - \rho_0}_{H^1(\mathbb{T}^3)}^2.
 \end{split}
\end{equation}
Inserting \eqref{ereg} and \eqref{ass:ic} into \eqref{re0} we obtain
\begin{equation}\label{re5}
 \int_{\mathbb{T}^3} \rre(0,  x) \operatorname{d}   x = \mathcal{O}(\alpha^{-1})\, .
\end{equation}
Upon using \eqref{re5} in \eqref{Gw} we find
\begin{equation}\label{Gw1}
 \int_{\mathbb{T}^3} \rre(t,  x) \operatorname{d}   x \leq 
 \Big(   \frac{C}{\alpha}+ \frac{1}{\alpha}\| E\|_{L^1(0,t)}\Big) e^{Ct}                      
\end{equation}
for some constant $C>0$, independent of $\alpha$.
Equation \eqref{Gw1} implies the assertion of the theorem as
\begin{multline}\label{Gw2}
 \frac{C_\kappa}{2C_P}  \|\rho(t,\cdot) - \ca(t,\cdot)\|_{H^1(\mathbb{T}^3)}^2 + 
\frac{C_\kappa \alpha}{2}  \Norm{\rhoa(t,\cdot) - \ca(t,\cdot)}_{L^2(\mathbb{T}^3)}^2 \\+
\frac{c_\rho}{2} \Norm{\va(t,\cdot) -   u(t,\cdot)}_{L^2(\mathbb{T}^3)}^2 \leq
\int_{\mathbb{T}^3} \rre(t,  x) \operatorname{d}   x      \,           .
\end{multline}
\end{proof}


\begin{remark}[Parameter dependence]\label{rem:cap2}
 All the estimates derived in this section heavily rely on the capillary regularization terms which are
 scaled with $C_\kappa.$ 
In particular, the estimate in Theorem  \ref{thrm:mt2} depends sensitively on $C_\kappa.$
Indeed, the constants $C$ in  Theorem  \ref{thrm:mt2} scales like $\exp(1/C_\kappa)$  
for  $C_\kappa \rightarrow 0.$
Thus, the results established here are only helpful in the diffuse case $C_\kappa>0$ and cannot
be transferred to the sharp interface limit case $C_\kappa \rightarrow 0.$
\end{remark}



\appendix

\section{Noether's theorem and the Korteweg stress}
\label{app-noether}

In this section we show that the relation \eqref{stresshyp}  (or \eqref{formula}) for the Euler-Korteweg system \eqref{euler-kort} can
be seen as a direct consequence of Noether's theorem and the invariance of the Korteweg functional
\begin{equation}
\label{kortfunctional}
 \cE_\Omega (\rho) = \int_\Omega F(\rho, \nabla \rho) \, d\rho
\end{equation}
under translations.
This relation is pointed out by Benzoni-Gavage \cite[App]{B10} and is proved there under the hypothesis that $\rho(\cdot)$ is an extremum of $\cE_\Omega (\rho)$ 
under the constraint of prescribed mass $\int \rho$. 

Here, we show that it is a direct consequence of the invariance under translations for the functional $ \cE_\Omega (\rho)$ with no further assumptions.
Note that,  if we perform the change of variables
$$
\rho^\star = \rho  \, , \quad x^\star = x + (h_1 , ... , h_d)
$$
and if $\Omega^\star$ is the image of $\Omega$, then the functional \eqref{kortfunctional} satisfies $\cE_{\Omega^\star}  (\rho^\star ) = \cE_\Omega (\rho)$.

Indeed, the proof of \eqref{formula} and the existence of the Korteweg tensor is an application of the results associated 
with variations of functionals on variable domains 
that lead to Noether's theorem (see Gelfand and Fomin \cite[Sec 37]{GF63}).
We fix $k$, and consider the change of variables
\begin{equation}
\label{specialvar}
\begin{cases}
x_i^* = x_i +  h_k  \delta_{ik}  & i= 1, ..., d\, , \\
\rho^* = \rho
\end{cases}
\end{equation}
(with no summation over the index $k$). 
The first variation of the functional $\cE_\Omega (\rho)$ along variations compatible with \eqref{specialvar} is zero.
Applying Theorem 1  in \cite[Sec 37.4]{GF63},  which computes the first variation of $\cE_\Omega (\rho)$ when the variation
of the domain is taken into account,  we see that the invariance together with the arbitrariness of the domain $\Omega$ implies the identity
$$
- \frac{\del \rho}{\del x_k} \left ( F_\rho - \sum_i \frac{\del}{\del x_i} F_{q_i} \right ) - \sum_i \frac{\del}{\del x_i} \left ( F_{q_i} \frac{\del \rho}{\del x_k} - F \delta_{ik} \right ) 
= 0\, ,
$$
for $k = 1, ..., d$.
After an integration by parts, this formula readily yields
\begin{equation}
\label{formulaapp}
- \rho \frac{\del}{\del x_k} \left ( F_\rho - \sum_i \frac{\del}{\del x_i} F_{q_i} \right ) = \sum_i \frac{\del}{\del x_i} 
\left ( \big (F - \rho F_\rho + \rho \div F_q  \big ) \delta_{ik} - F_{q_i} \frac{\del \rho}{\del x_k} \right )
\end{equation}
which coincides with \eqref{formula} or  \eqref{stresshyp}.


\section{The relative energy transport identity for the Euler-Korteweg system}\label{subsec:relenKort}
We present here the relative entropy calculation for the Euler-Korteweg system when both $(\rho, \, m)$ and $(\bar \rho, \, \bar m)$ are smooth solutions of \eqref{eq:Kortcap2};
the framework of weak solution $(\rho, \, m)$ is discussed in Section  \ref{subsec:weakdissKort}. This calculation will determine explicitly the relative-energy
flux  (a term  omitted in the integral version valid for weak solutions).

In what follows, we often omit the dependence of the potential energy
\begin{equation*}
F(\rho,\nabla\rho)= h(\rho) + \frac12 \kappa(\rho)|\nabla\rho|^2,
\end{equation*}
 the stress 
\begin{equation*}
S = - \big( \rho F_\rho(\rho,\nabla\rho) -  F(\rho,\nabla\rho) -\rho \dx (F_q(\rho,\nabla\rho))  \big)\mathbb{I} - F_q(\rho,\nabla\rho)\otimes\nabla\rho
\end{equation*}
 and their derivatives on the variables $\rho$ and $\nabla\rho$, and we denote with $\bar F$ and $\bar S$ these quantities when evaluated for the solution $\bar \rho$.

We start by analyzing the linear part of the potential energy about $\bar \rho$,  that is:
\begin{equation*}
\bar F_\rho (\rho - \bar\rho) + \bar F_q \cdot \nabla(\rho - \bar\rho) = 
\Big (\bar F_\rho  - \dx\big ( \bar F_q\big) \Big )  (\rho - \bar\rho) + \dx \big ( \bar F_q (\rho - \bar\rho) \big)\, . 
\end{equation*}
%
 Using the equation satisfied by the difference $(\rho - \bar\rho)$, after some calculation
we obtain
\begin{align}
& \partial_t \Big  ( \bar F_\rho (\rho - \bar\rho) + \bar F_q \cdot \nabla(\rho - \bar\rho) \Big ) + \dx \Big (\big (\bar F_\rho  - \dx  ( \bar F_q ) \big )   (m - \bar m) +   \bar F_q  \dx (m - \bar m) \Big) \nonumber\\
&\ =\nabla \big (\bar F_\rho  - \dx  ( \bar F_q ) \big) \cdot (m - \bar m)
- \big(\bar F_{\rho\rho}\dx \bar m + \bar F_{\rho q}\cdot \nabla(\dx\bar m) \big)(\rho -\bar \rho) \nonumber\\
&\ \ - \bar F_{q \rho}\cdot\nabla(\rho - \bar\rho)\dx\bar m
- \bar F_{q_j q_i} \partial_{x_j}(\rho -\bar\rho)\partial_{x_i}(\dx\bar m)\, .
\label{eq:linearKorGen1}
\end{align}
Using
\begin{equation}\label{eq:potenKortGen}
\partial_t F(\rho,\nabla\rho) + \dx  \Big ( 
     m\big( F_\rho- \dx   (F_q ) \big)    + F_q \dx m 
          \Big ) = m\cdot \nabla \big(  F_\rho- \dx   (F_q )\big)\, 
\end{equation}
 for both $(\rho, \, m)$ and $(\bar \rho, \, \bar m)$ and \eqref{eq:linearKorGen1}, we get
\begin{align*}
& \partial_t F(\rho, \nabla\rho\left | \bar \rho, \nabla\bar \rho \right. ) + \dx \Big (m\big (  F_\rho  - \dx  (   F_q )- (\bar F_\rho  - \dx  ( \bar F_q )) \big )     +    ( F_q - \bar F_q ) \dx m  \Big) \nonumber\\
&\ =m \cdot \nabla \big ( F_\rho  - \dx  (   F_q )- (\bar F_\rho  - \dx  ( \bar F_q )) \big) \nonumber\\
&\ \  + \big(\bar F_{\rho\rho}\dx \bar m + \bar F_{\rho q}\cdot \nabla(\dx\bar m) \big)(\rho -\bar \rho)  
+ \bar F_{q \rho}\cdot\nabla(\rho - \bar\rho)\dx\bar m \nonumber\\
&\ \ 
+ \bar F_{q_j q_i} \partial_{x_j}(\rho -\bar\rho)\partial_{x_i}(\dx\bar m) \nonumber\\
&\ = \rho \left ( \frac{m}{\rho} - \frac{\bar m }{\bar\rho} \right) \cdot \nabla \big ( F_\rho  - \dx  (   F_q )- (\bar F_\rho  - \dx  ( \bar F_q )) \big) \nonumber\\
&\ \  + \frac{\bar m }{\bar\rho} \cdot \rho \nabla \big ( F_\rho  - \dx  (   F_q )- (\bar F_\rho  - \dx  ( \bar F_q )) \big)
\nonumber\\
&\ \  + \big(\bar F_{\rho\rho}\dx \bar m + \bar F_{\rho q}\cdot \nabla(\dx\bar m) \big)(\rho -\bar \rho)  
+ \bar F_{q \rho}\cdot\nabla(\rho - \bar\rho)\dx\bar m \nonumber\\
&\ \ + \bar F_{q_j q_i} \partial_{x_j}(\rho -\bar\rho)\partial_{x_i}(\dx\bar m) \nonumber\\ 
&\ = D + I_1 + I_2\, ,
\end{align*}
where
\begin{align*}
  D &= \rho \left ( \frac{m}{\rho} - \frac{\bar m }{\bar\rho} \right) \cdot \nabla \big ( F_\rho  - \dx  (   F_q )- (\bar F_\rho  - \dx  ( \bar F_q )) \big)\, ;
\\
  I_1 &= \frac{\bar m }{\bar\rho} \cdot \rho \nabla \big ( F_\rho  - \dx  (   F_q )- (\bar F_\rho  - \dx  ( \bar F_q )) \big)\, ;
\\
  I_2 &= \big(\bar F_{\rho\rho}\dx \bar m + \bar F_{\rho q}\cdot \nabla(\dx\bar m) \big)(\rho -\bar \rho)  
+ \bar F_{q \rho}\cdot\nabla(\rho - \bar\rho)\dx\bar m  \\
&\ 
  + \bar F_{q_j q_i} \partial_{x_j}(\rho -\bar\rho)\partial_{x_i}(\dx\bar m)\, .
\end{align*}


Recalling
\begin{equation}\label{eq:Sigmagenrewrit}
\rho \partial_{x_j} \big ( F_\rho  - \dx  (   F_q ) \big) = - \partial_{x_k}S_{kj}\, ,
\end{equation}
%
%
$I_1$ is rewritten as:
\begin{align*}
 I_1 &= \frac{\bar m_j }{\bar\rho}  \Big( \rho  \partial_{x_j} \big ( F_\rho  - \dx  (   F_q ) \big )-  \bar\rho \partial_{x_j} \big (\bar F_\rho  - \dx  ( \bar F_q ) \big) \Big) 
 - (\rho -\bar\rho) \frac{\bar m}{\bar\rho}\cdot\nabla (\bar F_\rho  - \dx  ( \bar F_q )) 
 \\
 &= -  \frac{\bar m_j }{\bar\rho} \partial_{x_k}\big ( S_{kj} - \bar S_{kj}\big) - (\rho -\bar\rho) \frac{\bar m}{\bar\rho}\cdot\nabla (\bar F_\rho  - \dx  ( \bar F_q )) \, ,
   \end{align*}
and therefore
\begin{align}
& \partial_t F(\rho, \nabla\rho\left | \bar \rho, \nabla\bar \rho \right. ) 
\nonumber\\
& \ \ + \dx \left ( \vphantom{\frac{\bar m  }{\bar\rho}}m\big (  F_\rho  - \dx  (   F_q )- (\bar F_\rho  - \dx  ( \bar F_q )) \big )     
+    ( F_q - \bar F_q ) \dx m  \right.  
 + \left. \big ( S - \bar S\big)\frac{\bar m  }{\bar\rho} \right) \nonumber\\
&= D 
- \partial_{x_k} \left (\frac{\bar m_j }{\bar\rho}\right )\Big (
 \big (\rho F_\rho -F  - \rho  \dx  (   F_q ) -(\bar\rho \bar F_\rho -\bar F  - \bar\rho  \dx  (   \bar F_q )) \big )\delta_{kj} 
 +    F_{q_k}\partial_{x_j}\rho 
 - \bar F_{q_k}\partial_{x_j}\bar \rho
\Big)  \nonumber\\
 &\ \  - (\rho -\bar\rho) \frac{\bar m}{\bar\rho}\cdot\nabla (\bar F_\rho  - \dx  ( \bar F_q )) 
  + I_2 
  \nonumber \\
  &= D - I_{1a} - I_{1b}  + I_2\, ,
\label{eq:RelPotEnKorGen2}
\end{align}
where
\begin{align*}
  I_{1a} &=  \partial_{x_k} \left (\frac{\bar m_j }{\bar\rho}\right )\Big (
 \big (\rho F_\rho -F  - \rho  \dx  (   F_q ) -(\bar\rho \bar F_\rho -\bar F  - \bar\rho  \dx  (   \bar F_q )) \big )\delta_{kj} 
 +    F_{q_k}\partial_{x_j}\rho 
 - \bar F_{q_k}\partial_{x_j}\bar \rho
\Big)\, ;
\\
  I_{1b} &=  (\rho -\bar\rho) \frac{\bar m}{\bar\rho}\cdot\nabla (\bar F_\rho  - \dx  ( \bar F_q ))\, .
\end{align*}

After some rearrangement of  derivatives $I_2$ becomes 
\begin{align*}
 I_2 &= 
 \dx \left ( \frac{\bar m}{\bar\rho}\right ) 
  \Big ( \bar\rho\big(\bar F_{\rho\rho} - \partial_{x_i}(\bar F_{\rho q_i})  \big)(\rho-\bar\rho) + 
\bar\rho \big(\bar F_{\rho q_j} - \partial_{x_i}(\bar F_{q_j q_i})  \big)\partial_{x_j}(\rho-\bar\rho) 
 \\
 &\ -  \bar\rho\big(\bar F_{\rho q_i}  \partial_{x_i}(\rho-\bar\rho) +\bar F_{q_j q_i}  \partial_{x_i x_j}(\rho-\bar\rho)\big) \Big)
 \\
 &\ + 
 \frac{\bar m_j}{\bar\rho} \Big (  \partial_{x_j}\big(\bar F_{\rho}\big)(\rho-\bar\rho) + 
 \partial_{x_j}\big(\bar F_{q_l}\big)\partial_{x_l}(\rho-\bar\rho)
   - \partial_{x_l}\big(  \partial_{x_j}\bar\rho \big(  \bar F_{\rho q_l} (\rho -\bar \rho) + \bar F_{q_kq_l} \partial_{x_k}(\rho -\bar\rho) \big)\big) \Big)
 \\
 &\
+\dx \Big (\big (\bar F_{\rho q}(\rho -\bar \rho) +  \bar F_{qq}\nabla (\rho -\bar\rho)\big) \dx\bar m\Big )
\\
&= I_{2a} + I_{2b} - I_{2c} + I_{2d}\, ,
   \end{align*}
where 
\begin{align*}
  I_{2a} &=  \dx \left ( \frac{\bar m}{\bar\rho}\right ) 
  \Big ( \bar\rho\big(\bar F_{\rho\rho} - \dx(\bar F_{\rho q})  \big)(\rho-\bar\rho) + 
\bar\rho \big(\bar F_{\rho q_j} - \dx(\bar F_{q_j q})  \big)\partial_{x_j}(\rho-\bar\rho) 
 \\
 &\ -  \bar\rho\big(\bar F_{\rho q_i}  \partial_{x_i}(\rho-\bar\rho) +\bar F_{q_j q_i}  \partial_{x_i x_j}(\rho-\bar\rho)\big) \Big)\, ;
  \\
  I_{2b} &=  \frac{\bar m_j}{\bar\rho} \Big (  \partial_{x_j}\big(\bar F_{\rho}\big)(\rho-\bar\rho) + 
 \partial_{x_j}\big(\bar F_{q_l}\big)\partial_{x_l}(\rho-\bar\rho)\Big )\, ;
  \\
  I_{2c} &=  \frac{\bar m_j}{\bar\rho} \Big ( \partial_{x_l}\big(  \partial_{x_j}\bar\rho \big(  \bar F_{\rho q_l} (\rho -\bar \rho) + \bar F_{q_kq_l} \partial_{x_k}(\rho -\bar\rho) \big)\big) \Big)\, ;
  \\
  I_{2d} &=  \dx \Big (\big (\bar F_{\rho q}(\rho -\bar \rho) +  \bar F_{qq}\nabla (\rho -\bar\rho)\big) \dx\bar m\Big )\, .
\end{align*}

We shall now rearrange the various terms defined above as follows:
\begin{align*}
 -I_{1b} + I_{2b} &= 
   \frac{\bar m_j}{\bar\rho} \partial_{x_l}\partial_{x_j}\big((\rho -\bar\rho)\bar F _{q_l} \big)  
   -  \partial_{x_l}\left(\frac{\bar m_j}{\bar\rho}\partial_{x_j}(\rho-\bar\rho)\bar F _{q_l} \right)
   + \partial_{x_l}\left(\frac{\bar m_j}{\bar\rho}\right) \big(\bar F _{q_l}\partial_{x_j}(\rho-\bar\rho)  \big)\, ;
   \\
   -I_{2c} &=  -\partial_{x_l}\left (\frac{\bar m_j}{\bar\rho}  \partial_{x_j}\bar\rho \Big(  \bar F_{\rho q_l} (\rho -\bar \rho) + \bar F_{q_kq_l} \partial_{x_k}(\rho -\bar\rho) \Big) \right ) 
   \\
   &\ + \partial_{x_l}\left (\frac{\bar m_j}{\bar\rho}\right)  \Big( \partial_{x_j}\bar\rho \bar F_{\rho q_l} (\rho -\bar \rho) + \partial_{x_j}\bar\rho\bar F_{q_kq_l} \partial_{x_k}(\rho -\bar\rho) \Big)\, .
\end{align*}
Using the above relations in  \eqref{eq:RelPotEnKorGen2} we obtain
\begin{align}
& \partial_t F(\rho, \nabla\rho\left | \bar \rho, \nabla\bar \rho \right. ) + \dx J_1 = D 
\nonumber\\
&\ -
\dx\left (\frac{\bar m }{\bar\rho}\right )\Big (
 \big (\rho F_\rho -F  - \rho  \dx  (   F_q ) -(\bar\rho \bar F_\rho -\bar F  - \bar\rho  \dx  (   \bar F_q )) \big )
 \nonumber\\
&\  -
 \bar\rho\big(\bar F_{\rho\rho} - \dx(\bar F_{\rho q})  \big)(\rho-\bar\rho) - 
\bar\rho \big(\bar F_{\rho q_j} - \dx(\bar F_{q_j q})  \big)\partial_{x_j}(\rho-\bar\rho) 
 \nonumber\\
 &\ +  \bar\rho\bar F_{\rho q_i}  \partial_{x_i}(\rho-\bar\rho) +\bar\rho\bar F_{q_j q_i}  \partial_{x_i x_j}(\rho-\bar\rho)
 \Big)
 \nonumber\\
 &\ -  \partial_{x_l}\left (\frac{\bar m_j}{\bar\rho}\right)\Big(   F_{q_l}\partial_{x_j}\rho 
 - \bar F_{q_l}\partial_{x_j}\bar \rho  - \partial_{x_j}\bar\rho \bar F_{\rho q_l} (\rho -\bar \rho) - \partial_{x_j}\bar\rho\bar F_{q_kq_l} \partial_{x_k}(\rho -\bar\rho) - \bar F_{q_l}\delta_{jk}\partial_{x_k}(\rho-\bar\rho) \Big)
 \nonumber\\
 &\ + \frac{\bar m_j}{\bar\rho} \partial_{x_l}\partial_{x_j}\big((\rho -\bar\rho)\bar F _{q_l} \big)  
   -  \partial_{x_l}\left(\frac{\bar m_j}{\bar\rho}\partial_{x_j}(\rho-\bar\rho)\bar F _{q_l} \right) 
   \nonumber\\
   & = D - R_1 - R_2 +R_3\, ,
\label{eq:RelPotEnKorGen3}
\end{align}
where
\begin{align*}
  J_1 &=  
m\big (  F_\rho  - \dx  (   F_q )- (\bar F_\rho  - \dx  ( \bar F_q )) \big )     +    ( F_q - \bar F_q ) \dx m    \nonumber\\
& \ +  \big ( S - \bar S\big)\frac{\bar m  }{\bar\rho} 
- \big (\bar F_{\rho q}(\rho -\bar \rho) +  \bar F_{qq}\nabla (\rho -\bar\rho)\big) \dx\bar m
  \nonumber\\
  &\ + 
  \left (\frac{\bar m}{\bar\rho} \cdot \nabla_{x}\bar\rho\right ) \Big(  \bar F_{\rho q} (\rho -\bar \rho) + \bar F_{qq}\nabla (\rho -\bar\rho) \Big)\, ;
   \\
  R_{1} &=  \dx\left (\frac{\bar m }{\bar\rho}\right )\Big (
 \big (\rho F_\rho -F  - \rho  \dx  (   F_q ) -(\bar\rho \bar F_\rho -\bar F  - \bar\rho  \dx  (   \bar F_q )) \big )
 \nonumber\\
&\  -
 \bar\rho\big(\bar F_{\rho\rho} - \dx(\bar F_{\rho q})  \big)(\rho-\bar\rho) - 
\bar\rho \big(\bar F_{\rho q_j} - \dx(\bar F_{q_j q})  \big)\partial_{x_j}(\rho-\bar\rho) 
 \nonumber\\
 &\ +  \bar\rho\bar F_{\rho q_i}  \partial_{x_i}(\rho-\bar\rho) +\bar\rho\bar F_{q_j q_i}  \partial_{x_i x_j}(\rho-\bar\rho)
 \Big)\,
 ;
  \\
  R_{2} &=  \partial_{x_l}\left (\frac{\bar m_j}{\bar\rho}\right)\Big(   F_{q_l}\partial_{x_j}\rho 
 - \bar F_{q_l}\partial_{x_j}\bar \rho  - \partial_{x_j}\bar\rho \bar F_{\rho q_l} (\rho -\bar \rho) - \partial_{x_j}\bar\rho\bar F_{q_kq_l} \partial_{x_k}(\rho -\bar\rho) - \bar F_{q_l}\delta_{jk}\partial_{x_k}(\rho-\bar\rho) \Big)\, ;
  \\
  R_{3} &=   \frac{\bar m_j}{\bar\rho} \partial_{x_l}\partial_{x_j}\big((\rho -\bar\rho)\bar F _{q_l} \big)  
   -  \partial_{x_l}\left(\frac{\bar m_j}{\bar\rho}\partial_{x_j}(\rho-\bar\rho)\bar F _{q_l} \right)\, .
\end{align*}

Now we recall the notations 
\begin{align*}
 H(\rho , q) &=  q\otimes F_q(\rho,q)\,   ,\\ 
 s(\rho , q) & = \rho F_\rho(\rho,q)   + q\cdot F_q(\rho,q) - F(\rho,q)\, ,\\
 r(\rho , q) & = \rho F_q(\rho,q)\, ,
\end{align*}
and we readily obtain 
\begin{align*}
&R_2 = \nabla \left (\frac{\bar m}{\bar\rho} \right) : H(\rho, \nabla\rho\left | \bar \rho, \nabla\bar \rho \right. )\, ;
\\
 &\bar\rho\big(\bar F_{\rho\rho} - \dx(\bar F_{\rho q}) \big) =  
 s_\rho(\bar \rho , \nabla\bar\rho) - \dx \big (\bar\rho \bar F_{\rho q} \big )\, ;
 \\
 &\bar\rho \big(\bar F_{\rho q_j} - \dx(\bar F_{q_j q})  \big) = 
 s_{q_j}(\bar \rho , \nabla\bar\rho) - \dx \big (\bar\rho \bar F_{q_j q} \big )\, .
\end{align*}
Therefore 
\begin{align*}
  -R_1 + R_3 &=  \dx \left (\frac{\bar m}{\bar\rho} \dx \big( (\rho -\bar\rho)  \bar F_q\big) - \left(\frac{\bar m}{\bar\rho} \cdot\nabla ( \rho -\bar\rho)\right) \bar F_q \right ) 
   - \dx \left (\frac{\bar m}{\bar\rho}\right ) s(\rho, \nabla\rho\left | \bar \rho, \nabla\bar \rho \right. )
  \\
  & \ + \dx \left (\frac{\bar m}{\bar\rho}\right ) \Big (  \dx \big( \rho F_q\big) - \dx \big( \bar\rho \bar F_q\big) 
   - \dx \big (\bar\rho \bar F_{\rho q} \big ) (\rho - \bar\rho)-  \dx \big (\bar\rho \bar F_{q_j q} \big )\partial_{x_j}(\rho -\bar\rho)
   \\
   &\ - \bar\rho\bar F_{\rho q_i}  \partial_{x_i}(\rho-\bar\rho) - \bar\rho\bar F_{q_j q_i}  \partial_{x_i x_j}(\rho-\bar\rho) - \dx \big( (\rho -\bar\rho)  \bar F_q\big)
  \Big) 
  \\
  &= - \dx J_2 - \dx \left (\frac{\bar m}{\bar\rho}\right ) s(\rho, \nabla\rho\left | \bar \rho, \nabla\bar \rho \right. ) - \nabla \dx \left (\frac{\bar m}{\bar\rho}\right ) \cdot r(\rho, \nabla\rho\left | \bar \rho, \nabla\bar \rho \right. )\, ,
  \end{align*}
where
\begin{equation*}
 J_2 =   - \left (\frac{\bar m}{\bar\rho} \dx \big( (\rho -\bar\rho)  \bar F_q\big) - \left(\frac{\bar m}{\bar\rho} \cdot\nabla ( \rho -\bar\rho)\right) \bar F_q \right ) - \dx \left (\frac{\bar m}{\bar\rho}\right )   r(\rho, \nabla\rho\left | \bar \rho, \nabla\bar \rho \right. )\, .
\end{equation*}

\medskip
Using these identities in \eqref{eq:RelPotEnKorGen3} we finally express the relative potential energy in the form:
\begin{align}
& \partial_t F(\rho, \nabla\rho\left | \bar \rho, \nabla\bar \rho \right. ) + \dx J = D 
- \dx \left (\frac{\bar m}{\bar\rho}\right ) s(\rho, \nabla\rho\left | \bar \rho, \nabla\bar \rho \right. ) 
\nonumber\\
&\ 
- \nabla \left (\frac{\bar m}{\bar\rho} \right) : H(\rho, \nabla\rho\left | \bar \rho, \nabla\bar \rho \right. ) 
- \nabla \dx \left (\frac{\bar m}{\bar\rho}\right ) \cdot r(\rho, \nabla\rho\left | \bar \rho, \nabla\bar \rho \right. )\, ,
\label{eq:RelPotEnKorGenFinal}
\end{align}
where
\begin{align}
  J &=  J_1 + J_2 
  \nonumber
  \\
  &= 
  m\big (  F_\rho  - \dx  (   F_q )- (\bar F_\rho  - \dx  ( \bar F_q )) \big )     +    ( F_q - \bar F_q ) \dx m    
  \label{defJ} 
 \nonumber\\
& \ +  \big ( S - \bar S\big)\frac{\bar m  }{\bar\rho} 
- \big (\bar F_{\rho q}(\rho -\bar \rho) +  \bar F_{qq}\nabla (\rho -\bar\rho)\big) \dx\bar m
  \nonumber\\
  &\ + 
  \left (\frac{\bar m}{\bar\rho} \cdot \nabla_{x}\bar\rho\right ) \Big(  \bar F_{\rho q} (\rho -\bar \rho) + \bar F_{qq}\nabla (\rho -\bar\rho) \Big)
    \nonumber\\
  &\ - \left (\frac{\bar m}{\bar\rho} \dx \big( (\rho -\bar\rho)  \bar F_q\big) - \left(\frac{\bar m}{\bar\rho} \cdot\nabla ( \rho -\bar\rho)\right) \bar F_q \right ) - \dx \left (\frac{\bar m}{\bar\rho}\right )   r(\rho, \nabla\rho\left | \bar \rho, \nabla\bar \rho \right. )
\end{align}
and
\begin{equation*}
D= \rho \left ( \frac{m}{\rho} - \frac{\bar m }{\bar\rho} \right) \cdot \nabla \big ( F_\rho  - \dx  (   F_q )- (\bar F_\rho  - \dx  ( \bar F_q )) \big)\, .
\end{equation*}

The density of the relative kinetic energy is given by
\begin{equation}\label{def:rke}
\frac{1}{2}\rho \left | \frac{m}{\rho} -  \frac{\bar m}{\bar\rho}\right | ^2 .
\end{equation}
We recall  \eqref{eq:entrKortGen2} and \eqref{eq:potenKortGen} (see also \eqref{eq:Sigmagenrewrit}) and thus  the kinetic energy alone satisfies
\begin{equation*}
\partial_t \left  (\frac{1}{2} \frac{|m|^2}{\rho} \right ) + \dx  \left ( \frac{1}{2}m\frac{|m|^{2}}{\rho^{2}} \right ) = - 
m\cdot \nabla \big(  F_\rho- \dx   (F_q )\big)\, .
\end{equation*}
We write the difference for the equations satisfied by $\frac{m}{\rho}$ and $\frac{\bar m}{\bar \rho}$, 
\begin{align*}
&  \partial_t \left ( \frac{m}{\rho}-   \frac{\bar m}{\bar\rho}\right )+ \left( \frac{m}{\rho}\cdot\nabla\right )\left( \frac{m}{\rho}-   \frac{\bar m}{\bar\rho}\right ) + \left(\left( \frac{m}{\rho}-   \frac{\bar m}{\bar\rho}\right )\cdot\nabla \right)  \frac{\bar m}{\bar\rho}\\
&\ = -  \nabla \big (
 F_\rho  - \dx  (   F_q ) -   \big(\bar F_\rho  - \dx  (  \bar F_q )\big)
 \big)\, ,
  \end{align*}
and multiply  by $\frac{m}{\rho} -  \frac{\bar m}{\bar\rho}$  to obtain
\begin{align*}
&\frac{1}{2} \partial_t \left|  \frac{m}{\rho} -  \frac{\bar m}{\bar\rho}\right|^2+ \left (\frac{m}{\rho}\cdot\nabla\right)\left(\frac{1}{2}\left | \frac{m}{\rho} -  \frac{\bar m}{\bar\rho}  \right |^2 \right)+
 \left (\frac{m_i}{\rho} -  \frac{\bar m_i}{\bar\rho}\right )\left (\frac{m_j}{\rho} -  \frac{\bar m_j}{\bar\rho}\right )\partial_{x_j}\left ( \frac{\bar m_i}{\bar\rho}\right ) \\
 &\  =     - \left (\frac{m}{\rho} -  \frac{\bar m}{\bar\rho}\right )\cdot   \nabla \big (
 F_\rho  - \dx  (   F_q ) -   \big(\bar F_\rho  - \dx  (  \bar F_q )\big)
 \big) \, .
  \end{align*}
Using the conservation law \eqref{eq:Kortcap2}$_1$, we end up with
\begin{equation}
\begin{aligned}
  \partial_t \left( \frac{1}{2} \rho \left | \frac{m}{\rho}  -  \frac{\bar m}{\bar\rho}\right | ^2  \right)  &+ 
  \dx \left(\frac{1}{2}m\left | \frac{m}{\rho} -  \frac{\bar m}{\bar\rho}\right | ^2 \right)
  \\
  &= -D -  \rho \left( \frac{m}{\rho}-   \frac{\bar m}{\bar\rho}\right )\otimes \left( \frac{m}{\rho}-   \frac{\bar m}{\bar\rho}\right ) :\nabla  \left(    \frac{\bar m}{\bar\rho}\right )\, .
\end{aligned}
 \label{eq:relkinKortGen}
 \end{equation}
Adding \eqref{eq:RelPotEnKorGenFinal} and \eqref{eq:relkinKortGen} leads to the identity for the transport of the relative energy,
\begin{align}
& \partial_t \left ( \frac{1}{2} \rho \left | \frac{m}{\rho} -  \frac{\bar m}{\bar\rho}\right | ^2 +  F(\rho, \nabla\rho\left | \bar \rho, \nabla\bar \rho \right. ) \right )+ \dx \left ( \frac{1}{2}m\left | \frac{m}{\rho} -  \frac{\bar m}{\bar\rho}\right | ^2  + J \right )
\nonumber\\
&\ = -  \rho \left( \frac{m}{\rho}-   \frac{\bar m}{\bar\rho}\right )\otimes \left( \frac{m}{\rho}-   \frac{\bar m}{\bar\rho}\right ) :\nabla  \left(    \frac{\bar m}{\bar\rho}\right )
- \dx \left (\frac{\bar m}{\bar\rho}\right ) s(\rho, \nabla\rho\left | \bar \rho, \nabla\bar \rho \right. ) 
\nonumber\\
&\ 
- \nabla \left (\frac{\bar m}{\bar\rho} \right) : H(\rho, \nabla\rho\left | \bar \rho, \nabla\bar \rho \right. ) 
- \nabla \dx \left (\frac{\bar m}{\bar\rho}\right ) \cdot r(\rho, \nabla\rho\left | \bar \rho, \nabla\bar \rho \right. )\, ,
\label{eq:RelEnKorGenFinal}
\end{align}
where the flux $J$ is defined in \eqref{defJ}.

\section{Proof of Lemma \ref{lem:fre}}\label{sec:proofmc}
The sole purpose of this appendix is giving the details of the proof of Lemma \ref{lem:fre}.
\begin{proof}[Proof of Lemma \ref{lem:fre}]
We start with  direct calculations
\begin{equation}\label{eq:re1a}
\begin{split}
  \eta^\alpha_t 
=& h'(\rhoa)\rhoa_t + \frac{\ma \cdot \ma_t}{\rhoa} - \frac{\norm{\ma}^2}{2(\rhoa)^2}\rhoa_t + \alpha C_\kappa (\rho - \ca)(\rhoa_t - \ca_t)\\
&+ C_\kappa \nabla \ca \cdot \nabla \ca_t - h(\rho)\rho_t + C_\kappa \nabla \rho \cdot \nabla \rho_t -  h''(\rho)\rho_t (\rhoa - \rho) - h'(\rho)\rhoa_t\\
&+h'(\rho)\rho_t + \frac{  m \cdot   m_t}{\rho^2}\rhoa - \frac{\norm{  m}^2}{\rho^3}\rhoa \rho_t + \frac{\norm{  m}^2}{\rho^2}\rhoa_t\\
&-C_\kappa \nabla \rho_t \cdot\nabla \ca - C_\kappa \nabla \rho \cdot \nabla \ca_t - \frac{  m_t \cdot \ma}{\rho} 
- \frac{  m \cdot \ma_t}{\rho} + \frac{  m \cdot \ma }{\rho^2} \rho_t
\end{split} 
\end{equation}
and
\begin{equation}\label{eq:re1b}
\begin{split}
\div   q^\alpha &=
 \div(\ma)h'(\rhoa) + \ma \cdot \nabla h'(\rhoa) + \frac{\ma \otimes \ma}{(\rhoa)^2} : \nabla \ma 
 + \frac{\norm{\ma}^2}{2(\rhoa)^2} \div(\ma) \\
 &
- \frac{\norm{\ma}^2}{(\rhoa)^3} \ma \cdot \nabla \rhoa
 -\nabla h'(\rho)\cdot \ma  - h'(\rho) \div(\ma) + \frac{  m \otimes \ma}{\rho^2} : \nabla   m\\
&
 + \frac{\norm{  m}^2}{2\rho^2}\div(\ma)
 - \frac{\norm{  m}^2}{\rho^3}\ma \cdot \nabla \rho - \frac{  m \otimes \ma}{\rho \rhoa} : \nabla \ma  
 - \frac{\ma\otimes \ma}{\rho \rhoa} : \nabla   m 
\\&
- \frac{  m \cdot \ma}{\rho \rhoa} \div(\ma)
 + \frac{  m \cdot \ma}{\rho^2 \rhoa} \ma \cdot \nabla \rho + \frac{  m \cdot \ma}{\rho (\rhoa)^2} \ma \cdot \nabla \rhoa
 - \frac{\div(  m)}{\rho} p(\rhoa) 
\\ &
+ \frac{  m \cdot \nabla \rho}{\rho^2} p(\rhoa)  - \frac{  m\cdot \nabla p(\rhoa)}{\rho}
 + \frac{\div(  m)}{\rho} p(\rho) - \frac{  m \cdot \nabla \rho}{\rho^2} p(\rho)  + \frac{  m\cdot \nabla p(\rho)}{\rho}\, .
\end{split} 
\end{equation}
Inserting the evolution equations \eqref{NSK-sec} and \eqref{lo} in \eqref{eq:re1a} we obtain
\begin{equation}\label{eq:re2}
\begin{split}
  \eta^\alpha_t
=& \frac{\norm{\ma}^2}{(\rhoa)^3} \nabla \rho \cdot \ma
 - \frac{\ma\otimes \ma }{(\rhoa)^2} : \nabla \ma
- \frac{\norm{\ma}^2}{(\rhoa)^2}\div(\ma) 
- \frac{\ma \cdot \nabla p(\rhoa)}{\rhoa} 
+ \frac{\ma}{\rhoa} \cdot\div(\sigma[\va])
\\
&
+ C_\kappa\ma \cdot \nabla \triangle \ca
+ \frac{\norm{\ma}^2}{2(\rhoa)^2}\div(\ma) 
+ C_\kappa \triangle \ca \div(\ma) 
+ C_\kappa\div(\nabla \ca \ca_t)
\\
& 
-C_\kappa \nabla \rho \cdot \nabla \div(  m)  
+ h''(\rho) \div(  m) (\rhoa - \rho)
+\norm{  m}^2 \frac{\rhoa}{\rho^4}   m \cdot \nabla \rho
 - \frac{\rhoa}{\rho^3}   m \otimes   m : \nabla   m
\\
&- \frac{\rhoa}{\rho^3} \norm{  m}^2 \div(  m)  
- \frac{\rhoa}{\rho^2}   m\cdot \nabla p(\rho)
+ \frac{\rhoa}{\rho^2}   m \cdot \div(\sigma[  u])
 + C_\kappa \frac{\rhoa}{\rho}  m \cdot \nabla \triangle \rho 
+ \frac{\rhoa}{\rho^3}\norm{  m}^2 \div(  m)
\\
&- \frac{\norm{  m}^2}{2\rho^2}  \div (\ma) 
+ C_\kappa \nabla \div(  m) \cdot \nabla \ca 
+ C_\kappa \triangle \rho (\ca_t - \rhoa_t)
-C_\kappa\div(\ca_t \nabla \rho)
\\
&
+ C_\kappa \nabla \rho \cdot \nabla \div(\ma)
-\frac{  m \cdot \ma}{\rho^3} \nabla \rho \cdot   m
+ \frac{\ma\otimes   m}{\rho^2} : \nabla   m 
+ \frac{\div (  m)}{\rho^2}  m \cdot  \ma
 +\frac{\ma}{\rho}\cdot \nabla p(\rho)
\\
&
-\frac{\ma}{\rho} \cdot\div(\sigma[  u]) 
- C_\kappa\ma \cdot \nabla \triangle \rho
- \frac{\ma \cdot   m}{\rho (\rhoa)^2}\nabla \rhoa \cdot \ma 
+ \frac{  m \otimes \ma}{\rho \rhoa} :\nabla   \ma
\\
&
+ \frac{  m \cdot \ma}{\rho\rhoa} \div(\ma) 
+ \frac{  m}{\rho}\cdot\nabla p(\rhoa) 
- \frac{  m}{\rho}\cdot\div(\sigma[\va]) 
- C_\kappa \frac{\rhoa}{\rho}   m\cdot \nabla \triangle \ca 
- \frac{\div(  m)}{\rho^2}   m \cdot \ma .
\end{split} 
\end{equation}
Adding  \eqref{eq:re1b} and \eqref{eq:re2} we observe that several terms cancel out and we obtain
\begin{equation}
 \label{eq:re3}
\begin{split}
 &  \eta^\alpha_t + \div   q^\alpha
\\
&= \frac{\ma}{\rhoa} \cdot\div(\sigma [\va]) 
+ \frac{\rhoa}{\rho^2}  m \cdot \div(\sigma [  u]) 
- \frac{\ma}{\rho} \cdot\div(\sigma [  u])
- \frac{  m}{\rho} \cdot\div(\sigma [\va])
\\
&
+C_\kappa \div \Big( (\ma -   m) \triangle( \ca - \rho) + \nabla (\ca - \rho) \ca_t + \nabla \rho \div(\ma-   m) 
+ \div(  m) \nabla \ca
\Big)
\\
&+
 C_\kappa \Big(\frac{\rhoa}{\rho}  -1\Big)   m \cdot \nabla \triangle (\rho - \ca) 
+ C_\kappa \nabla \rho \cdot \nabla(\rhoa_t - \ca_t)
\\
&
+ \frac{1}{\rho^2}   m \otimes \nabla   \rho :  \Big( \frac{  m}{\rho}- \frac{\ma}{\rhoa} \Big)\otimes \Big( \frac{  m \rhoa}{\rho} -\ma\Big)
\\
&
+ \frac{1}{\rho^2}   m \cdot \nabla   \rho \big( p(\rhoa) - p(\rho) - p'(\rho)(\rhoa - \rho) 
\big)
\\
&
+ \frac{1}{\rho} \nabla   m : \Big( \ma -   m \frac{\rhoa}{\rho}\Big)\otimes \Big( \frac{  m}{\rho} - \frac{\ma}{\rhoa}\Big)
\\
&
- \frac{1}{\rho} \div(  m)  \big( p(\rhoa) - p(\rho) - p'(\rho)(\rhoa - \rho) 
\big)\, .
\end{split}
\end{equation}
Due to the periodic boundary conditions, \eqref{eq:re3} implies the assertion of the lemma.
\end{proof}

\def\cprime{$'$}

\end{document}